\documentclass[a4paper,11pt]{article}
\usepackage{amsmath,amsthm,amssymb,enumerate}
\usepackage{graphicx}
\usepackage{euscript,mathrsfs}
\usepackage{dsfont}
\usepackage{xcolor}
\usepackage{empheq}
\usepackage{bm}
\usepackage{geometry}
\usepackage{color}
\catcode`\@=11 \@addtoreset{equation}{section}

\catcode`\@=12

\usepackage{upgreek}
\usepackage{cancel}
\usepackage{soul}
\usepackage[normalem]{ulem}

\usepackage{tcolorbox}
\tcbuselibrary{skins}

\usepackage{stmaryrd}
\allowdisplaybreaks

\usepackage{braket,amsfonts}
\usepackage[colorlinks=true,linkcolor=blue,filecolor=magenta,urlcolor=cyan]{hyperref}

\usepackage{tikz}
\usetikzlibrary{patterns}
\usepackage{booktabs}
\usepackage{tikz,pgfplots,pgfplotstable}
\pgfplotsset{compat=1.15}
\usetikzlibrary{matrix,arrows.meta,calc,shapes}
\tikzset{
  centered/.style = { align=center, anchor=center },
     empty/.style = { font=\sffamily\Large, centered, text width=2cm },
       box/.style = { font=\sffamily, fill=green, centered },
    result/.style = { font=\sffamily\scriptsize, fill=black!20, centered},
     arrow/.style = { very thick, color=red, ->, >=Triangle},
}

\usepackage{listings} 

\usepackage{scalerel,stackengine}
\stackMath
\newcommand\reallywidehat[1]{%
\savestack{\tmpbox}{\stretchto{%
  \scaleto{%
    \scalerel*[\widthof{\ensuremath{#1}}]{\kern-.6pt\bigwedge\kern-.6pt}%
    {\rule[-\textheight/2]{1ex}{\textheight}}
  }{\textheight}%
}{0.5ex}}%
\stackon[1pt]{#1}{\tmpbox}%
}

\newtheorem{Theorem}{Theorem}[section]

\newtheorem{Lemma}[Theorem]{Lemma}
\newtheorem{Corollary}[Theorem]{Corollary}

\newtheorem{Definition}[Theorem]{Definition}

\newtheorem{Remark}[Theorem]{Remark}

\newcommand{\M}{\mathbb{M}}
\newcommand{\Mh}{{\M}_h}
\newcommand{\Mhk}{{\M}_h^k}
\newcommand{\Mk}{{\M}^k}
\newcommand{\bfw}{\mathbf{w}}

\newcommand{\bfq}{\mathbf{q}}

\newcommand{\bfe}{\mathbf{e}}

\newcommand{\detJ}{J}
\newcommand{\Jacob}{{\mathbb F}}
\newcommand{\mJkm}{ \frac{\eta_h^{k-1}  }{  \eta_h^k }   }

\newcommand{\bfn}{\mathbf{n}}

\newcommand{\xx}{x_1}
\newcommand{\xy}{x_2}

\newcommand{\xref}{\widehat{x}}
\newcommand{\xrefx}{\xref_1}
\newcommand{\xrefy}{\xref_2}
\newcommand{\Xkm}{X_k^{k-1}}
\newcommand{\Xkminv}{X_{k-1}^k}

\newcommand{\eu}{e_\vu}
\newcommand{\euk}{e_\vu^k}

\newcommand{\tr}{\,{\rm tr}}

\newcommand{\rmS}{{\rm S}}
\newcommand{\rmT}{{\rm T}}

\newcommand{\bfphi}{\boldsymbol{\varphi}}
\newcommand{\hbfphi}{\widehat{\bfphi}}
\newcommand{\hphi}{\widehat{\varphi}}
\newcommand{\hq}{\widehat{q}}

\newcommand{\bftau}{\mathbb{T}}
\newcommand{\hbftau}{\widehat{\bftau}}

\newcommand{\sumj}{\sum_{j=1}^d}
\newcommand{\summ}{\sum_{k=1}^m}

\newcommand{\sumN}{\sum_{k=1}^N}

\newcommand{\Div}{{\rm div}\,}
\newcommand{\Grad}{\nabla}
\newcommand{\Divref}{\Div_{\xref}}
\newcommand{\Gradref}{\Grad_{\xref}}

\newcommand{\Laph}{ \pd_{\xx,h}^2 }
\newcommand{\pdx}{ \pd_{\xx}}
\newcommand{\Lapx}{ \pdx^2}

\newcommand{\dt}{\,{\rm d} t }

\newcommand{\dvol}{\,{\rm d}x}
\newcommand{\dvolh}{\,{\rm d}x}
\newcommand{\dvolref}{\,{\rm d}\xref}
\newcommand{\dSx}{\,{\rm d}{\scalebox{0.8}{$S$}}(x)}
\newcommand{\dSxref}{\,{\rm d}{\scalebox{0.8}{$S$}}(\xref)}
\newcommand{\ds}{\,{\rm d}\xx}

\newcommand{\Abs}[1]{ \left| #1 \right|}
\newcommand{\abs}[1]{ | #1 |}
\newcommand{\Bigabs}[1]{ \Big| #1 \Big|}
\newcommand{\biggabs}[1]{ \bigg| #1 \bigg|}
\newcommand{\norm}[1]{\left\lVert#1\right\rVert}
\newcommand{\Ov}[1]{\overline{#1}}

\newcommand{\aleq}{\stackrel{<}{\sim}}

\newcommand{\vrf}{\varrho_f}
\newcommand{\vrs}{\varrho_s}
\newcommand{\vu}{\mathbf{u}}

\newcommand{\vw}{\bfw}
\newcommand{\ve}{\bfe}
\newcommand{\er}{{\ve}_2}

\newcommand{\vn}{\mathbf{n}}

\newcommand{\hvu}{\widehat{\vu}}
\newcommand{\hvw}{\widehat{\vw}}
\newcommand{\hvv}{\widehat{\vv}}
\newcommand{\hv}{\widehat{v}}

\newcommand{\hp}{\widehat{p}}

\newcommand{\FSI}{W_\eta}
\newcommand{\hFSI}{\widehat{W}_\eta}
\newcommand{\FSIh}{W_{\eta_h}}
\newcommand{\hFSIh}{\widehat{W}_{\eta_h}}
\newcommand{\vc}[1]{{\bf #1}}

\newcommand{\vv}{\vc{v}}

\newcommand{\I}{\mathbb{I}}
\newcommand{\Id}{\I}

\newcommand{\R}{\mathbb{R}}

\definecolor{Cgrey}{rgb}{0.85,0.85,0.85}
\definecolor{Cblue}{rgb}{0.50,0.85,0.85}
\definecolor{Cred}{rgb}{1,0,0}
\definecolor{fancy}{rgb}{0.10,0.85,0.10}
\definecolor{forestgreen}{rgb}{0.13, 0.55, 0.13}

\newcommand{\vx}{{\bm x}}
\newcommand{\vxref}{\widehat{\bm x}}

\newcommand{\TS}{\tau}
\newcommand{\pd}{\partial}
\newcommand{\pdt}{\pd _t}
\newcommand{\pdtt}{\pd _t^2}

\newcommand{\mdt}{\pd^M_t}

\newcommand{\PDt}{D_t}
\newcommand{\MDt}{D^M_t}

\newcommand{\Qfh}{Q^f_h}
\newcommand{\Vfh}{V^f_h}
\newcommand{\Vsh}{V^s_h}
\newcommand{\Vshz}{V^s_{0,h}}
\newcommand{\Vfsih}{V^{fsi}_h}
\newcommand{\hQfh}{\widehat{Q}^f_h}
\newcommand{\hVfh}{\widehat{V}^f_h}
\newcommand{\hVfsih}{\widehat{V}^{fsi}_h}

\newcommand{\ALE}{\mathcal{A}_{\eta}}

\newcommand{\ALEh}{\mathcal{A}_{\eta_h}}

\newcommand{\ALEhinv}{\ALEh^{-1}}

\newcommand{\intS}[1] {\int_\Sigma #1 \ds }
\newcommand{\intD}[1] {\int_D #1 {\rm d}x}

\newcommand{\intSB}[1] {\int_\Sigma \left( #1 \right) \ds }

\newcommand{\intOfh}[1]{\int_{\Ofh} #1 \dvolh}
\newcommand{\intOfhk}[1]{\int_{\Ofh^k} #1 \dvolh}

\newcommand{\intOfhkB}[1]{\int_{\Ofh^k} \left( #1 \right)\dvolh}

\newcommand{\intOf}[1]{\int_{\Of} #1 \dvol}
\newcommand{\intOfk}[1]{\int_{\Ofk} #1 \dvolh}
\newcommand{\intO}{\intOf}
\newcommand{\intOB}[1]{\int_{\Of} \left( #1 \right)\dvol}
\newcommand{\Of}{\Omega_\eta}
\newcommand{\Ofk}{\Omega_{\eta^k}}
\newcommand{\Ofh}{\Omega_{\eta_h}}
\newcommand{\Ofhk}{\Omega_{\eta_h^k}}

\newcommand{\intOref}[1]{\int_{\Oref} #1 \dvolref}
\newcommand{\intOrefB}[1]{\int_{\Oref}\left( #1 \right)\dvolref}

\newcommand{\Oref}{{\widehat{\Omega}}}
\newcommand{\gridf}{\mathcal{T}_h}
\newcommand{\grids}{ \Sigma_h}
\newcommand{\calP}{ \mathcal{P}}
\newcommand{\calL}{ \mathcal{L}}

\newcommand{\Ecal}{{\mathcal{E}_{\eta_h}}}
\newcommand{\Bcal}{{\mathcal{B}}}

\newcommand{\PiS }{ \mathcal{P}_h^s}
\newcommand{\Pif}{{\mathcal{P}_h^f}}
\newcommand{\Pifh}{{\widehat{\mathcal{P}}_h^f}}
\newcommand{\PiF}{{\Pi_h^f}}
\newcommand{\Piq}{ \Pi_h^Q}
\newcommand{\Riesz}{{\mathcal{R}_h^s}}

\begin{document}


\title{Stability and error estimates of a linear numerical scheme approximating nonlinear fluid--structure interactions
\thanks{All authors thank for the support of the ERC-CZ Grant LL2105 CONTACT, the program GJ19-11707Y of the Czech national grant agency (GA{\v C}R) and the  Charles University Research program No. UNCE/SCI/023. S. S. and B. S. also thank the Primus research program PRIMUS/19/SCI/01.
}
}

\author{Sebastian Schwarzacher\thanks{Department of Analysis, Faculty of Mathematics and Physics, Charles University (schwarz@karlin.mff.cuni.cz, ktuma@karlin.mff.cuni.cz)}
\and Bangwei She\thanks{Academy for Multidisciplinary Studies, Capital Normal University; Institute of Mathematics of the Czech Academy of Sciences
(she@math.cas.cz).
} $^{ , \dagger}$
\and Karel T\r{u}ma$^{\dagger}$
}
\date{\today}
\maketitle

\begin{abstract}
In this paper, we propose a linear and monolithic finite element method for the approximation of an incompressible viscous fluid interacting with an elastic and deforming plate. 
We use the arbitrary Lagrangian--Eulerian (ALE) approach that works in the reference domain, meaning that no re-meshing is needed during the numerical simulation. 
For time discretization, we employ the backward Euler method. For space discretization, we respectively use P1-bubble, P1, and P1 finite elements for the approximation of the fluid velocity, pressure, and structure displacement. 
We show that our method fulfills the geometrical conservation law and dissipates the total energy on the discrete level. Moreover, we prove the (optimal) linear convergence with respect to the sizes of the time step $\TS$ and the mesh $h$. 
We present numerical experiments involving a substantially deforming fluid domain that do validate our theoretical results. A comparison with a fully implicit (thus nonlinear) scheme indicates that our semi-implicit linear scheme is faster and as accurate as the fully implicit one, at least in stable configurations.  
\medskip


\end{abstract}


\tableofcontents

\section{Introduction}\label{sec:1}

Fluid--structure interaction (FSI) problems occur in many engineering applications,  from aero-elasticity to civil engineering and bio-mechanical problems, such as  the design of aircraft wings, wind turbines and heat exchangers, the response of bridges and skyscrapers to wind force,  blood flow in arteries, see \cite{Bazilevs, Bodnar, Tezduyar} among others. 

Numerical simulation of FSI problems has been largely studied and great progress has been achieved during the past decades; see, for examples, \cite{Basting,Boris1,Hundertmark,Landajuela,Richter} and references therein. 
Concerning the numerical stability analysis, we would like to mention the nice results of Luk\'{a}\v{c}ov\'{a}-Medvid'ov\'{a} et al. \cite{Lukacova}, Buka\v{c} and Muha \cite{Boris2}, Lozovskiy et al. \cite{Lozovskiy, Lozovskiy2}, Hecht and Pironneau \cite{Pironneau}, and Wang et al. \cite{Pironneau2} as examples. 
However, in terms of convergence analysis, there are certainly many more efforts to be made. To our best knowledge, only a few results are available on this topic; see Buka\v{c} and Muha \cite{Boris2}, Burman et al. \cite{Burman, Burman2}, Fern\'{a}ndez and Mullaert \cite{Fernandez}, and Seboldt and Buka\v{c} \cite{Seboldt}. In this direction, all available literature results 
are not only under the assumption that the displacement of the solid structure is infinitesimal but also based on the ignorance of the convection of the fluid motion. The main target of this paper is to show the convergence of a numerical approximation without these restrictions. 

For that reason, we study an archetypical setting of fluid-structure interaction. In our setting a one-dimensional plate is situated on the top of a two-dimensional container filled with a viscous incompressible liquid governed by the Navier--Stokes equations. The plate is governed by a hyperbolic equation driven by fluid traction. It may deform largely and therefore the Eulerian fluid domain is time-changing. This implies a severe nonlinear coupling between the structure and fluid equation. 

In order to solve the FSI problem numerically, we introduce a {\em linear}, implicit-explicit (semi-implicit), and monolithic finite element method. For time discretization, we take the backward Euler method. For space discretization, we start with the so-called arbitrary Lagrangian--Eulerian (ALE) mapping and transfer the time-dependent domain to a fixed reference grid. Then, we use an inf-sup stable finite element pair (P1-bubble/P1) on the reference domain for the fluid, and P1 elements for the structure. Our aims of the paper are to design an energy {\em stable} scheme and to show the (optimal) convergence rate of the numerical solution. 

 
The key point in the construction of the stability of our linear and semi-implicit scheme is that we keep the scheme {\em implicit with respect to the velocities}. In particular, the velocities of the solid structure and fluid are coupled implicitly in time, see also a similar construction of Lozovskiy et al. \cite{Lozovskiy}. Nevertheless, the scheme is linear as we take the fluid domain explicitly, which means it is given by the deformation of the plate of the previous time-step. Further, the convective term of the fluid is linearized in a stable manner. 

To some extent the current paper can be viewed as a numerical counterpart of Schwarzacher and Sroczinski \cite{SchSro20}, where the authors investigated the distance between a weak solution and a strong solution, while the aim of this paper is to investigate the distance between a numerical solution obtained by a finite element method  and a smooth solution. 
In order to adapt this result to a discrete numerical scheme, 
rather sophisticated analytic tools have to be invented. In particular, good projection operators have to be invented for a smooth solution. 
The challenge comes from the change of the fluid domain in time, which results in several non-trivial analytic difficulties on all levels when studying the convergence rate. Roughly, there are three different sources of errors that have to be estimated: i) the mismatch between the continuous geometry and the discrete geometry; ii) the respective different divergence-free constraints; iii) the projector of the fluid-velocity which has to fit a rather particular choice of a projector according to the structure equation. The first point is overcome by a change of variables. The second point is already very technical. For that, we introduce a Fortin operator for variable geometries in order to inherit the discrete solenoidality from the continuous one. Then the divergence-free condition destroyed by a change of variable is resolved by a Bogovskij correction recently developed by Kampschulte et al. ~\cite{KamSchSpe22}. The last point, the mismatch between the interpolation operator of the fluid at the boundary turns out to be the hardest to overcome. The reason is that the structure equation is of the fourth order in space. Hence, a discrete bi-Laplacian naturally appears. In order to gain suitable estimates for the structure part, a very particular choice of projector, the so-called Riesz projection operator, has to be used. Further, we have to solve a discrete Stokes problem in order to find a suitable projector of the fluid velocity that possesses these particular boundary values.

\begin{tcolorbox}
  The \emph{main result} of the paper is that there is a monolithic, linear, and fully-discrete scheme of the FSI problem \eqref{pde_f}--\eqref{pde_bdc}, which satisfies under suitable conditions:
\begin{itemize}
\item Total energy stability (see Theorem~\ref{Thm_Sta})
\item Linear convergence w.r.t. the space and time discretization parameters (see Theorem~\ref{theorem_conv_rate})
\end{itemize} 
The \emph{highlight} of the paper reads:
\begin{itemize}
\item The proposed semi-implicit scheme is linear; the absence of nonlinear iterations makes the approach computationally cheaper than the fully implicit scheme, yet provides equally good results.
\item We show the energy stability of the method. 
\item We show the linear convergence rate of the method with respect to the computational parameters $\TS$ (the time step) and $h$ (mesh size). Such a result has not been achieved in literature for any structure interacting with fluids described by the Navier--Stokes equations.   
\item The discrete structure displacement is defining the real-time geometry of the Eulerian fluid domain. 
\item The convergence rate is optimal as demonstrated by numerical experiments, see Section~\ref{sec:num}.
\end{itemize}
\end{tcolorbox}

\subsection{Problem formulation}
In this paper, we are interested in the interaction between an incompressible viscous fluid and a thin elastic structure, which is part of the fluid boundary. 
More precisely, we consider the motion of an incompressible and viscous fluid flow in a time-dependent domain 
\[ \Of(t)=\left\{\vx =(\xx,\xy) \in \Sigma \times (0, \eta(t,\xx)) \right\} \subset \R^2, 
\]
where $\Sigma = (0,L_1)$, $L_1$ is the length of the domain, $\eta=\eta(t,\xx)>0$ represents the height of the upper boundary $\Gamma_S(t)$ of the fluid domain $\Of$. 
For the sake of simplicity, we assume that i) the flow is periodic in the $\xx$-direction; ii) the upper boundary is formed by an elastic structure that can move in the $\xy$-direction; iii) the bottom boundary $\Gamma_D$ is a solid wall; iv) initially $\Of(0)= \Sigma \times[0,1]$.

In this paper, we shall use the ALE method and directly work on a time-independent reference domain $\Oref= \Of(0)$ instead of the time-dependent domain $\Of(t)$. To this end, we introduce an ALE mapping $\ALE$ that maps the reference domain $\Oref$ to the time dependent domain $\Of$, i.e.
\begin{equation*}
\ALE:\Oref \mapsto \Of, \quad 
(\xrefx, \xrefy) \mapsto (\xx,\xy)=\ALE(t,\vxref) 
=\left(\xrefx , \eta \xrefy \right),
\end{equation*}
see Figure~\ref{fig_ALE} for a graphical illustration of the domain and ALE mapping. 
%
\begin{figure}[h!]\centering
\vspace{-1.5cm}
\begin{tikzpicture}[scale=1.0]
\draw[->] (0,0)--(5.7,0);
\draw[->] (0,0)--(0,2.6);
\draw[very thick, green](0,2)--(0,0);
\draw[very thick, blue](0,0)--(5,0);
\draw[very thick, green](5,0)--(5,2);
\draw[very thick, red](5,2)--(0,2);
\path node at (2.5,1) {$\Oref$};
\path node at (2.5,-0.3) {\textcolor{blue}{$\Gamma_D$}};
\path node at (2.5,2.3) {\textcolor{red}{$\widehat{\Gamma}_S=\Sigma$}};
\path node at (-0.2,-0.2) {$0$};
\path node at (5.5,-0.3) {$\xrefx$};
\path node at (5,-0.3) {$L_1$};
\path node at (-0.2,2) {$1$};
\path node at (-0.32,2.5) {$\xrefy$};
\draw[very thick,->] (5.3,1)--(6.7,1);
\path node at (6,1.3) {$\ALE$};
\draw[->] (7,0)--(12.7,0);
\draw[->] (7,0)--(7,2.6);
\draw[very thick, green](7,2)--(7,0);
\draw[very thick, blue](7,0)--(12,0);
\draw[very thick, green](12,0)--(12,2);
\draw[thick, red]   (7, 2) .. controls (8.5, 0.5) and (11, 4) .. (12, 2);
\path node at (9.5,1) {$\Omega(t)$};
\path node at (9.5,-0.3) {\textcolor{blue}{$\Gamma_D$}};
\path node at (8.7,2.2) {\textcolor{red}{$\Gamma_S(t)$}};
\draw[<->,cyan ] (10.75,0)--(10.75,2.6);
\path node at (11,1) {\textcolor{cyan}{$\eta$}};
\path node at (6.8,-0.2) {$0$};
\path node at (12.5,-0.3) {$\xx$};
\path node at (12,-0.3) {$L_1$};
\path node at (6.68,2.5) {$\xy$};
\end{tikzpicture}
\vspace{-0.5cm}
\caption{Time dependent domain and the ALE mapping}\label{fig_ALE}
\end{figure}
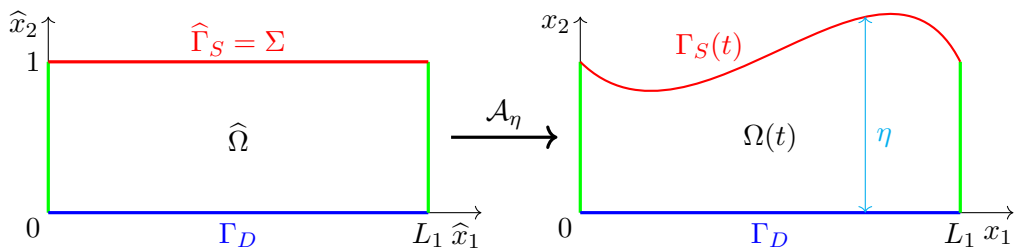

\paragraph{Fluid model.}The motion of the incompressible viscous fluid is described by the Navier--Stokes equations
\begin{equation}\label{pde_f}
\left\{ 
\begin{aligned}
\vrf \left( \pdt \vu + (\vu \cdot \Grad) \vu \right) - \Div \bftau(\vu, p) =0,  &\quad \text{ in } (0,T) \times \Of
\\
\Div \vu =0 , &\quad \text{ in } (0,T) \times \Of, 
\\
\vu =0 , &\quad \text{ on } (0,T) \times \Gamma_D, 
\end{aligned}
\right.
\end{equation}
where $\vrf$, $\vu = \vu(t, \vx)$, and $p= p(t, \vx)$ are the fluid density (given constant), velocity field, and pressure, respectively. The Cauchy stress $\bftau$ reads\footnote{We adopt the following notations: $(\Grad \vu)_{ij} = \pd_j u_i$, $\left(\vv \cdot (\Grad \vu) \right)_i =  \sumj v_j \pd_i u_j$ and  $\big(\vv \cdot \Grad \vu\big)_i =\big((\vv \cdot \Grad) \vu\big)_i = \sumj v_j \pd_j u_i $. Note that $\vv \cdot \Grad \vu \neq \vv \cdot (\Grad \vu)$ but $\vv \cdot \Grad \vu = (\vv \cdot \Grad) \vu = \Grad \vu \cdot \vv$. 
}
\[\bftau=  2\mu (\Grad \vu)^\rmS -p\I\quad \mbox{ with the constant viscosity coefficient } \mu>0, 
\]
and the superscript $S$ denotes the symmetric operator for a matrix-valued function $\mathbb A$, meaning that
${\mathbb A}^\rmS ={\mathbb A}+ {\mathbb A}^\rmT.$
\paragraph{Structure model.}
The motion of the elastic structure is given by 
\begin{equation}\label{pde_S}
\left\{ 
\begin{aligned}
\vrs \pdt \xi  + \calL (\eta) = f,  &\quad \text{ on } (0,T) \times  \Sigma, 
\\
\Lapx  \eta=0, &\quad \text{ on } (0,T) \times  \pd \Sigma, 
\end{aligned}
\right.
\end{equation}
where $\vrs>0$ is the density of the structure, $\xi=\pdt \eta$ is the velocity of the structure, $f$ is the interaction force acting on the structure due to fluid motion, and 
\[ \calL (\eta) = - \gamma_1  \Lapx \eta - \gamma_2 \Lapx \zeta -\gamma_3   \Lapx \pdt \eta,  \quad  \zeta = - \Lapx \eta, 
\]
where $\gamma_1 >0,\; \gamma_2 > 0, \; \gamma_3 \geq 0 $ are given constants. 

Further, the initial data of the problem read 
\begin{equation}\label{pde_ini}
\vu(0) =\vu_0 \text{ in } \Of(0)\quad \text{ and}  \quad \eta(0,\cdot)=\eta_0, \; \xi (0,\cdot)=\xi_0 \; \text{in}\; \Sigma.
\end{equation}
\paragraph{Coupling conditions.}
Finally, to close the system, we require coupling conditions at the fluid-structure interface, which are the so-called kinematic and dynamic boundary conditions: 
\begin{itemize}
\item the kinematic coupling condition
\begin{subequations}\label{pde_bdc}
\begin{equation}\label{pde_bdc1}
\vu(\vx) = \xi(\xx) \er ,   \quad \forall\;  \vx=(\xx,\eta) \in \Gamma_S.
\end{equation}
\item the dynamic coupling condition
\begin{equation}\label{pde_bdc2}
f= - \er \cdot \big( \detJ \bftau(\vu,p)\circ \ALE  \Jacob^{-T} \big)\cdot \er ,
\end{equation}
\end{subequations}
where 
$\Jacob=\Jacob(\eta)$ is the Jacobian of the mapping $\ALE$ and $\detJ =\detJ(\eta)$ is the corresponding determinant. In the current setting, we have 
\begin{equation}\label{Jacob}
\Jacob (\eta) =  \Gradref \ALE = 
\left(\begin{array}{cc}  1 & 0 \\ \xrefy \pdx \eta  & \eta  \end{array} \right) 
\quad \mbox{and} \quad 
\detJ (\eta)  =\det(\Jacob (\eta) ) = \eta. 
\end{equation} 
\end{itemize}

The {\bf plan of the paper} is the following. In Section~2 we discuss the weak formulation and stability of our FSI problem on the continuous level. 
In Section~3 we introduce the numerical method. In Section~4 we prove the  stability of the numerical solution on the discrete level. In Section~5 we introduce interpolation operators that are specially designed to fit both 
the divergence-free velocity field and the kinematic coupling condition. These operators are essential in Section~6, where we show the convergence rate of the numerical solution towards a strong solution. In Section~7 we present the numerical experiments. Finally, in Section~8 we give a short conclusion of the achievements in the paper.

\section{Weak formulation and stability}\label{Sec_wf}
In this section, we introduce a weak formulation of the FSI problem \eqref{pde_f}--\eqref{pde_bdc} and prove that a solution to the weak formulation is energy stable.  

To begin, we introduce the standard notations $W^{k,p}(D)$ and $L^p(D)$ on a generic domain $D$ for the Sobolev space and Lebesgue space, respectively. Further, we denote by $W^{k,p}_0$ the functions with zero traces on the boundary. 
In order to specify functions on the reference domain, we shall use the superscript `` {$\widehat{ } $} ''. For example, for a generic function $v=v(\vx)$ defined on $\Of$ we write on the reference domain that 
$\widehat{v}   =  v (\ALE(\vxref)) = v \circ \ALE.$ 
Next, we recall the Piola transformation~\cite{Ciarlet_elas} for the mapping $\ALE$:
\begin{equation}\label{Piola}
\begin{aligned}
&\dvol = \eta \dvolref, \quad 
\dSx =  \abs{\eta \Jacob^{-T} \widehat{\bfn} } \dSxref, \quad 
\bfn = \frac{\eta \Jacob^{-T} \widehat{\bfn} }{ \abs{\eta \Jacob^{-T} \widehat{\bfn} }}, \quad 
\\& \Divref(\eta \Jacob^{-1}) = {\bf 0}, \quad 
 \Div_x \bfq \circ \ALE = \frac{1}{\eta}\Divref  \left( \eta \Jacob^{-1}  \widehat{\bfq} \right) =  \Gradref\widehat{\bfq} : \Jacob^{-T}, 
\quad 
\Grad q\circ \ALE =  \Gradref  \widehat{q}  \;  \Jacob^{-1} ,
\end{aligned}
\end{equation}
where $\dvol$ (resp. $\dSx$) is the volume (resp. face) integral in the time-dependent domain, 
$\dvolref$ (resp. $\dSxref$) is the volume integral in the reference domain,
$\bfq$ and $q$ are generic vector-valued and scalar functions, respectively. Note that we have emphasized here the dependence of the differential operators $\Grad$ and $\Div$ with respect to $x$ and $\xref$. Hereinafter, if no confusion occurs, we shall simply write $\Grad$ (resp. $\Div$) instead of both $\Grad_x$ and $\Gradref$ (resp. $\Div_x$ and $\Divref$). 

Now, we define a new velocity field $\vw$ that describes the change of the fluid domain (ALE mapping) in time. It reads
\begin{equation}\label{w_con}
\hvw(\vxref):= \pdt \ALE =  \left(0 , \pdt \eta\, \xrefy \right) \quad \mbox{and}
\quad 
\vw(\vx) :=\hvw(\vxref) \circ \ALE^{-1}(\vx) = \left(0 , \xy\, \pdt \eta /\eta \right).
\end{equation}
Then, it is easy to observe the so-called Euler expansion
\begin{equation}\label{euler}
\Div \vw = \pdt \eta / \eta.
\end{equation}
According to the chain rule, we have
\begin{equation}\label{chain}
\begin{aligned}
&\pdt \widehat{v}(\vxref) = \frac{d}{dt} v (\ALE(\vxref)) = \pdt v(\vx) + \pdt \ALE(\vxref) \cdot \Grad v(\vx) 
\\&= \pdt v(\vx) + \vw(\vx) \cdot \Grad v(\vx) =:\mdt v(\vx) ,
\end{aligned}
\end{equation}
where $\mdt$ represents a material-type time derivative. 

With the above notations, it is easy to check the Reynolds transport theory
\begin{equation}\label{RT}
\begin{aligned}
&\pdt \intO{v} = \intO{\pdt v } +\int_{\pd \Of} v \vw \cdot \vn \dSx 
\\&= \intOB{\pdt v + \Div(v \vw)} 
 = \intOB{\mdt v + v \Div \vw }
.
\end{aligned}
\end{equation} 
Further, we denote $\vv$ as the relative velocity of the fluid with respect to the fluid domain. It reads 
\[  \vv = \vu -\vw \quad \mbox{satisfying the boundary condition } \quad  \vv|_{_{\pd \Of}} =0.\]
Thanks to the above boundary condition and the incompressibility condition~\eqref{pde_f}$_2$, we observe for any differentiable test function $\bfphi$ that 
\begin{equation*}
\begin{aligned}
& \intO{ \big((\vv \cdot \Grad) \vu \big)\cdot \bfphi  } 
=\intO{ \Big( \bfphi \cdot (\Grad \vu) \cdot \vv  + \frac12 ( \underbrace{  \Div \vv +\Div \vw }_{ =\Div \vu =0}  ) \vu  \cdot \bfphi \Big) }
\\ & 
=  \intO{\Div \vw\; \frac{\vu}2  \cdot \bfphi} + \frac12 \intO{ \big( \bfphi \cdot (\Grad \vu)  -  \vu \cdot (\Grad \bfphi) \big) \cdot   \vv   } .
\end{aligned}
\end{equation*}
Accordingly, we may reformulate the time derivative and convective terms in the following way
\begin{equation}\label{form1}
\begin{aligned}
& \intO{ \big(\pdt \vu + (\vu \cdot \Grad) \vu \big) \cdot \bfphi }
= \intO{\big(\pdt \vu + (\vw  \cdot \Grad) \vu \big) \cdot \bfphi} + \intO{ \big((\vv \cdot \Grad) \vu \big)\cdot \bfphi  }
\\ & 
= \intO{ \Big( \mdt \vu  +\Div \vw \; \frac{\vu}2 \Big) \cdot \bfphi } 
+ \frac12 \intO{ \big( \bfphi \cdot (\Grad \vu)  -  \vu \cdot (\Grad \bfphi) \big) \cdot   \vv   } .
\end{aligned}
\end{equation}
Finally, we introduce the following abbreviation for the sake of simplicity
\begin{equation}\label{abbs}
a_s(\eta, \zeta, \xi, \psi) =   \intSB{ \gamma_1 \pdx  \eta  \pdx  \psi  + \gamma_2 \pdx \zeta   \pdx \psi 
 + \gamma_3  \pdx  \xi  \pdx   \psi }  \quad \mbox{with } \zeta=-\Lapx \eta.
\end{equation} 
\subsection{Weak formulation}
Before introducing the weak formulation, we introduce the space of coupled test functions to accommodate the no-slip boundary condition \eqref{pde_bdc1}.
\[
\FSI=\left\{(\bfphi,\psi)\in W^{1,2}(\Omega_{\eta}) \times L^2(\Sigma):\psi(x) \er =\bfphi(x,\eta(x)),\, \bfphi=0\text{ on }\Gamma_D\right\}.
\]
Now we are ready to present the weak formulation of the FSI problem \eqref{pde_f}--\eqref{pde_bdc}. 
\begin{Definition}[Weak formulation of the FSI problem on $\Of$]\label{Def_wf} 
Let $(p,\vu,\xi,\eta)$ be a ``solution" to the FSI problem \eqref{pde_f}--\eqref{pde_bdc}. 
We say the following formula is a weak formulation of the FSI problem.
\begin{subequations}\label{wf}
\begin{equation}\label{wf1}
\intO{    q  \Div \vu } =0  \quad \mbox{for all } \quad q \in L^2(\Of);
\end{equation}
\begin{equation}\label{wf2}
\begin{aligned}
&
\vrf \intO{\Big(\mdt \vu + \Div \vw \;  \frac{\vu}2 \Big) \cdot \bfphi } 
+ \frac12 \intO{ \big( \bfphi \cdot (\Grad \vu)  -  \vu \cdot (\Grad \bfphi) \big) \cdot   \vv   } 
\\&+  \intO{\bftau(\vu,p) : \Grad \bfphi  }
+ \vrs\intS{\pdt \xi  \psi}  + a_s(\eta,\zeta, \xi, \psi) =0,
\end{aligned}
\end{equation}
\end{subequations}
for all 
$(\bfphi, \psi) \in \FSI$, where $a_s$ is given in \eqref{abbs}, and $\zeta=-\Lapx \eta$.
\end{Definition}
 Note that \eqref{wf1} is directly obtained by $\intOf{\eqref{pde_f}_2 \times q}$ for $q\in L^2(\Of)$ and \eqref{wf2} is obtained by calculating $\intOf{\eqref{pde_f}_1 \times \bfphi}+ \intS{\eqref{pde_S}_1 \times \psi}$ for the coupled test functions $(\bfphi,\psi) \in \FSI$, where we have used the following identity due to the Piola transformation \eqref{Piola}, the fact that the test functions $\bfphi$ and $\psi$ are coupled, and the coupling condition \eqref{pde_bdc2}.
\begin{align*}
&\int_{\Gamma_S}{ \bfphi  \cdot  \bftau \cdot \vn }\dSx 
= \intS{ \bfphi\circ \ALE \cdot ( \bftau\circ \ALE) \cdot (\eta\Jacob^{-T}\er) }
\\&= \intS{ \psi  \er \cdot ( \eta \bftau\circ \ALE \Jacob^{-T}) \cdot \er    }
= -\intS{f \psi}.
\end{align*} 

\subsection{Weak formulation on the reference domain}
By means of a change of variables, we may reformulate the weak formulation \eqref{wf} from the current domain $\Of$ onto the reference domain $\Oref$. 
\begin{Lemma}[Weak formulation of the FSI problem on $\Oref$]\label{Lwfref} 
Let $(p,\vu,\xi,\eta)$ satisfy the weak formulation \eqref{wf} of the FSI problem on $\Of$ with the test functions $(q,\bfphi, \psi) \in L^2(\Of) \times \FSI$. 
Let $(\hp,\hq,\hvu,\hbfphi) =  (p,q,\vu,\bfphi) \circ \ALE$.
Then there hold
\begin{subequations}\label{wf_ref}
\begin{equation}\label{wf1_ref}
\intOref{    \hq  \Grad \hvu :  \M} =0 ;
\end{equation}
\begin{equation}\label{wf2_ref}
\begin{aligned}
&\vrf \intOref{ \left( \eta \pdt \hvu  +   \pdt \eta  \frac{\hvu}2 \right)\cdot \hbfphi } 
+ \frac12 \vrf \intOref{   \big( \hbfphi \cdot (\Grad \hvu)   -  \hvu \cdot (\Grad \hbfphi)   \big) \cdot  \Jacob^{-1} \cdot \hvv \eta }
\\&
+ \intOref{   \hbftau(\hvu,\hp) : \big(\Grad \hbfphi \Jacob^{-1} \big)\eta } 
+ \vrs\intS{\pdt \xi  \psi} + a_s(\eta, \zeta, \xi, \psi)=0,
\end{aligned}
\end{equation}
\end{subequations}
where $a_s$ is given in \eqref{abbs},  $\zeta=-\Lapx \eta$, and 
\begin{equation}\label{tauref}
\hbftau(\hvu,\hp) = \bftau(\vu,p) \circ \ALE 
=  2 \mu ( \Grad \hvu \Jacob^{-1} )^\rmS -  \hp \Id, \quad 
\M:=\M(\eta):=\eta \Jacob^{-T} 
= \left(\begin{array}{cc} \eta   & -\xrefy \pdx \eta  \\ 0  & 1 \end{array} \right) 
 .
\end{equation}
\end{Lemma}
\begin{proof}
First, recalling \eqref{chain} we know that
\begin{align*}
\mdt \vu \circ \ALE = \pdt \hvu ,
\end{align*}
which together with the Euler expansion \eqref{euler} yield
\[  \vrf \intO{\Big(\mdt \vu + \Div \vw \; \frac{\vu}2  \Big) \cdot \bfphi } 
=\vrf \intOref{ \left( \eta \pdt \hvu  +   \pdt \eta  \frac{\hvu}2 \right)\cdot \hbfphi } 
\]
Next, using the  Piola transformation \eqref{Piola} we have
\[
 \intO{  \bfphi \cdot (\Grad \vu) \cdot \vv  } = 
  \intOref{ \hbfphi \cdot (  \Grad \hvu \Jacob^{-1}  )   \cdot \hvv \eta }, 
\]
which indicates
\[
  \frac12 \vrf \intO{ \big( \bfphi \cdot (\Grad \vu)  -  \vu \cdot (\Grad \bfphi) \big) \cdot   \vv   } 
=
 \frac12 \vrf \intOref{    \big( \hbfphi \cdot (\Grad \hvu)   -  \hvu \cdot (\Grad \hbfphi)   \big) \cdot  \Jacob^{-1} \cdot \hvv \eta }.
\]
Analogously, we find 
\begin{align*}
 & \intO{\bftau : \Grad \bfphi  } = \intO{ \big(2 \mu( \Grad \vu)^\rmS  - p \Id \big): \Grad \bfphi  } 
\\&= \intOref{ \big( 2 \mu  ( \Grad \hvu \Jacob^{-1} )^\rmS  -\hp \Id \big) : \big(\Grad \hbfphi \Jacob^{-1} \big)\eta }
=\intOref{\hbftau: \big(\Grad \hbfphi \Jacob^{-1} \big)\eta},
\\
\text{and}\quad &
\intO{    q  \Div \vu  } 
=\intOref{    \hq  \Grad \hvu : \Jacob^{-T}  \eta } 
= \intOref{   \hq  \Grad \hvu  : \M }.
\end{align*}
Consequently, collecting the above equalities we derive \eqref{wf_ref} from \eqref{wf}, which completes the proof. 
\end{proof}

\subsection{Energy stability}
Finally, we are ready to show the stability of the FSI problem \eqref{pde_f}--\eqref{pde_bdc}. Indeed, any solution to its weak formulation \eqref{wf} (or equivalently \eqref{wf_ref}) satisfies the following energy stability.
\begin{Lemma}[Stability of the continuous problem] \hfill

Let $(p,\vu, \xi,\eta )$ be a solution to \eqref{wf} (or $(\hp,\hvu,  \xi,\eta )$ be a solution to \eqref{wf_ref}). Then the following energy estimates hold. 
\begin{equation}
\begin{aligned}
& \pdt \intO{ \frac12 \vrf |\vu|^2  }  
+   \pdt \intS{ \frac12\left( \vrs |\xi|^2 +  \gamma_1 |\pdx \eta|^2 +  \gamma_2 |\Lapx \eta|^2 \right)  } 
\\& + 2\mu \intO{ |(\Grad \vu)^\rmS |^2 }
+\intS{\gamma_3 |\pdx \xi|^2 } =0. 
\end{aligned}
\end{equation}
\end{Lemma}
\begin{proof}
First, we derive by setting $(q, \bfphi, \psi) =(p,\vu,  \xi )$ in the weak formulation \eqref{wf} that
\begin{equation*}
\begin{aligned}
& \vrf \intOB{ \mdt \frac{|\vu|^2}{2}  + \frac{|\vu|^2}{2} \Div \vw  } 
+ 2\mu \intO{ |(\Grad \vu)^\rmS |^2 }
\\&
+ \vrs \intS{\pdt \frac{|\xi|^2}{2}} + \intSB{\gamma_1 \pdx \eta \pdx \xi +\gamma_2 \Lapx \eta \Lapx \xi + \gamma_3 |\pdx \xi|^2} =0.
\end{aligned}
\end{equation*}
Then we finish the proof after noticing $\xi =\pdt \eta$ and employing the Reynolds transport theory \eqref{RT} with $v=\frac{|\vu|^2}{2}$.  
\end{proof}

\section{Numerical method}
\label{sec:set}
In this section, we discretize the weak formulation introduced in the last section by a suitable finite element method. 

\subsection{Time discretization.}
We start with time discretization. Let $\TS$ be the time increment and  $t^k=k \TS$ for $k=0,1,\ldots, N(\equiv T/{\TS})$. Then we denote by $(p_h^k,\vu_h^k,\xi_h^k,\eta_h^k)$ the numerical approximation of the FSI problem at time $t^k$. Further, for any set of pointwise in time approximation $\{v_h^k\}_{k=0}^N$ we extend it to the whole time interval $[0, T]$ in the following way
\begin{equation}\label{tdis}
    v_h(t)=v_h^k\text{ in }[t_{k},t_{k+1}).
\end{equation}
The discrete Eulerian domain $\Ofh^k =\Ofhk$ at time $t^k$ is determined by the discrete ALE mapping: 
\begin{align*}
 \ALEh^k: \quad \Oref  \mapsto \Ofh^k, \quad \ALEh^k(\vxref) = (\xrefx,  \eta_h^{k} \xrefy ).
\end{align*}
Again $\ALEh$ is a piecewise constant in time function in the sense of \eqref{tdis}. 
For a generic function $v_h$ (including test functions) defined on $\Ofh$ we have 
$\widehat{v}_h =  v_h \circ \ALEh$ on $\Oref$.
Here we emphasize that $\widehat{\eta}_h = \eta_h$ and $\widehat{\xi}_h = \xi_h$ as their domain of definition $\Sigma$ is time independent. 

To approximate the time derivatives $\pdt v$ and $\mdt v$ we introduce 
\begin{equation}\label{time_D}
\begin{aligned}
&\PDt {v}_h^k   =\frac{{v}_h^k  -{v}_h^{k-1}  }{\TS} & \approx \quad & \pdt v ^k ,  
\\
&\MDt v_h^k  = \frac{v_h^k -v_h^{k-1}\circ \Xkm}{\TS}  & \approx \quad  &\mdt v ^k,
\end{aligned}
\end{equation} 
where $X_i^j = \ALEh^j \circ (\ALEh^i)^{-1}$ denotes the mapping from $\Ofh^i$ to $\Ofh^j$.

\subsection{Discrete Reynolds transport theory.}
Analogous to the continuous definitions \eqref{Jacob} and \eqref{w_con} and the identity \eqref{euler} we have 
  \begin{equation}\label{Jacobh}
\Jacob_h = \Jacob (\eta_h) =  \Gradref \ALEh = 
\left(\begin{array}{cc}  1 & 0 \\ \xrefy \pdx \eta_h  & \eta_h  \end{array} \right) , \quad 
 \det(\Jacob_h) = \eta_h. 
\end{equation}  
\begin{equation}\label{mesh_velo}
\begin{aligned}
 \hvw_h= \PDt \ALEh = (0, \PDt \eta_h \xrefy ) ,\; 
 \vw_h =\hvw_h \circ (\ALEh)^{-1} =\left(0, \frac{\PDt \eta_h}{\eta_h}\xy\right),
\end{aligned}
\end{equation}
and 
\begin{equation}\label{euler2}
     \Div \vw_h = \frac{ \PDt \eta_h }{ \eta_h}.
\end{equation}
Next, realizing the equality
\[ \PDt \widehat{v}_h^k = (\MDt v_h^k) \circ \ALEh^k
\]
we observe the discrete analogue of the Reynolds transport theorem \eqref{RT}, see also \cite[Lemma 1]{SS_FSI}.
\begin{Lemma}[Discrete Reynolds transport] \label{lm_DRT}
\hfill

Let $\PDt$ and $\MDt$ be given in \eqref{time_D}, then for any $k=1,\cdots,N_T$ we have
\begin{equation}\label{DRT}
\PDt \intOfhk{v_h^k } =\intOfhkB{\MDt v_h^k + \Div \vw_h^k  v_h^{k-1} \circ \Xkm} .
\end{equation}
\end{Lemma}
\begin{Remark}
    Choosing $v_h=1_{D_h}$ in Lemma \ref{lm_DRT} for any $D_h \subset \Ofh$, we have the geometric conservation law, i.e., 
    \begin{equation*}
    \PDt |D_h| =  \PDt \intOfh{1} =\intOfh{\Div \vw_h } = \int_{\pd D_h}{\vw_h \cdot \vn} \dSx.
    \end{equation*}
    Further, as we keep in our numerical scheme that $\vw_h=\vu_h$ on the boundary, where the velocity field is weakly divergence-free, we have 
    \begin{equation*}
         \PDt |\Ofh| = \int_{\pd \Ofh}{\vw_h \cdot \vn} \dSx
          = \int_{\pd \Ofh}{\vu_h \cdot \vn} \dSx
          =\intOfh{\Div \vu_h } =0.
    \end{equation*}
    Indeed, our method does fulfill the above equality, as our boundary condition is $\vu_h^{k-1} = \vw_h^k $ and we use a flat reference geometry. Therefore 
       \begin{equation*}
         \PDt |\Ofh^k| = \int_{\pd \Ofh^k}{\vw_h^k \cdot \vn} \dSx
          = \int_{\Sigma}\PDt \eta_h^k\ds
          = \int_{\Sigma}\xi_h^{k-1}\ds
          =\int_{\Ofh^{k-1}}{\Div \vu_h^{k-1} } =0,
    \end{equation*} 
    where the last equality is due to the weakly divergence-free condition \eqref{SKM1} with the choice of test function $q=1$. 
\end{Remark}

\subsection{Spatial discretization}
Let $\gridf$ be a shape regular and quasi-uniform triangulation of the reference domain $\Oref$, where $h$ stands for the maximum diameter of all elements of $\gridf$. Let $\grids$ be the surface mesh of $\gridf$ on the top boundary $\widehat{\Gamma}_S$. We denote by $K \in \gridf$ a generic element in $\gridf$ and by $\sigma \in \grids$ a generic face element in $\grids$. 
Moreover, we introduce the following function spaces on $\Oref$ 
\begin{align}\label{discrete_spaces}
& \hVfh = \left\{\hbfphi \in W^{1,2}(\Oref;\R^d) \middle|  \hbfphi \in \calP^1(K)   \oplus B_1(K)  , \; \forall K \in \gridf, \hbfphi|_{\Gamma_D}=0
\right\},
\quad
\\& B_1(K)=\left\{\phi \in \calP^3(K) \middle|  \phi(a_i)=0, \quad \text{where } a_i, \;  i=1,2,3, \ \text{are\ vertices\ of}\ K \in \gridf \right\}, 
\\&
\hQfh = \left\{\hq \in {C^0} (\Oref) \middle|  \hq \in \calP^1(K), \; \forall K \in \gridf \right\}, 
\\ & 
\Vsh= \left\{ \psi \in W^{1,2}_0(\Sigma)\middle| \psi \in \calP^1(\sigma), \; \forall \sigma \in \grids \right\}  ,
\end{align}
where $\calP^n(K)$ (resp. $\calP^n(\sigma)$) denote polynomials of degree not greater than $n$ on $K$ (resp. on $\sigma$).  
Further, we denote  
\[
\hFSIh= \left\{(\hbfphi,\psi)\in  \hVfh \times \Vsh \middle| \hbfphi(\xrefx,1)=\psi(\xrefx)\er\right\}.
\]
Finally, we denote $\hVfsih =\hQfh \times \hFSIh$ and $\Vfsih(t)=\hVfsih \circ \ALEh^{-1}(t) $. 

Let us point out that by using the linear finite element space $\Vsh$ for  the structure displacement $\eta$, we cannot directly discretize the bi-Laplacian term. Therefore, we decide to approximate the bi-Laplacian via duality, which still requires a discrete Laplace operator. To this end,  we introduce the following discrete Laplace operator for $\eta_h \in \Vsh$ by seeking $\Laph \eta_h \in \Vshz := \Vsh \cap L^2_0(\Sigma)$
such that 
\begin{equation}\label{laph}
 \intS{ \Laph \eta_h\; \psi }  +\intS{ \pdx \eta_h \pdx \psi } =0 \quad \mbox{for all } \psi \in \Vsh.
\end{equation}
Here, we would like to point out that $\Vshz$ is a finite dimensional space and in view of the assumptions on the grid the stiffness matrix of \eqref{laph} is invertible.  Thus, it admits a unique solution.

\subsection{The numerical method}
With the notations introduced above, we propose a monolithic finite element method for the discretization of the weak formulation \eqref{wf_ref}.
\paragraph{Scheme-R(\sl{A monolithic finite element method on the reference domain $\Oref$}).}
For $k=1,\dots, N$ we seek $(\hp_h^k, \hvu_h^k,\xi_h^k, \eta_h^{k+1}) \in \hVfsih  \times \Vsh $ with 
$\xi_h^k= \PDt \eta_h^{k+1}$ such that for all $(\hq,\hbfphi, \psi) \in \hVfsih$ there hold
\begin{subequations}\label{SKM_ref}
\begin{equation}\label{SKM_ref1}
\intOref{    \hq  \Grad \hvu_h^k :  \Mhk} =0 ,
\end{equation}
\begin{equation}\label{SKM_ref2}
\begin{split}
&\vrf\intOref{ \PDt \hvu_h^k \cdot \hbfphi \eta_h^{k} } 
+\frac12 \vrf \intOref{  \PDt \eta_h^k \hvu_h^{k*}  \cdot \hbfphi  } 
\\& 
+\frac12 \vrf \intOref{  \big( \hbfphi \cdot (\Grad\hvu_h^k)      -  \hvu_h^k \cdot (\Grad\hbfphi)   \big)  \cdot  (\Jacob_h^k)^{-1}  \cdot \hvv_h^{k-1} \eta_h^k } 
\\& 
+  \intOref{  \hbftau(\hvu_h^k,\hp_h^k)  : \left( \Grad \hbfphi  (\Jacob_h^k)^{-1} \right) \eta_h^k  }
\\& + \vrs  \intS{ \PDt \xi_h^k \psi} +a_s(\eta_h^{k+1}, \zeta_h^{k+1}, \xi_h^k ,\psi)
=  0,
\end{split}
\end{equation}
\end{subequations}
where $\hvu_h^{k*} =2 \hvu_h^{k-1} - \hvu_h^{k}$, $\hvv_h^{k-1}=\hvu_h^{k-1} - \hvw_h^k$, $\M_h^k = \M(\eta_h^k)$ is given in \eqref{tauref}, $a_s$ is given in \eqref{abbs}, $\zeta_h = - \Laph \eta_h$ is
 the (minus) discrete Laplace uniquely defined by \eqref{laph}, 
 $\hbftau$ is given in \eqref{tauref}, and the discrete initial data are given by  $\vu_h^0 =\Pif \vu^0$, $\xi_h^0 =\Riesz  \xi_0$, $\eta_h^0 =\Riesz  \eta_0$, and $\eta_h^1 =\eta_h^0 +\TS  \xi_h^0$.   
Here $\Pif: W^{1,2}(\Oref) \mapsto \hVfh$ is a suitable projection operator and $\Riesz: W^{1,2}(\Oref) \mapsto \Vsh$ is a Riesz projection operator to be clarified in the next section. 
\medskip

Note that \textbf{Scheme-R} approximates the FSI problem \eqref{pde_f}--\eqref{pde_bdc} based on the weak formulation~\eqref{wf_ref} in the reference domain $\Oref$. It is linear and belongs to the monolithic approach. Practically, it is more convenient to work with the reference domain as it is time-independent and no need for re-meshing. Nevertheless, many researchers appreciate working with the current domain $\Ofh$ (approximation of $\Of$). To this end, we present the following equivalent formulation of \textbf{Scheme-R} on the current domain. 
\paragraph{Scheme-C(\sl{A monolithic finite element method on the current (push-forward) domain $\Ofh$}).} 
Given the initial data \eqref{pde_ini} we set $\vu_h^0 =\Pif \vu^0$, $\xi_h^0 =\Riesz  \xi_0$, $\eta_h^0 =\Riesz  \eta_0$, and $\eta_h^1 =\eta_h^0 +\TS  \xi_h^0$.   
Then for $k=1,\dots, N$  we seek  $( p_h^k,\vu_h^k,\xi_h^k, \eta_h^{k+1}) \in \Vfsih \times \Vsh$ with 
$\xi_h^k= \PDt \eta_h^{k+1}$ such that for all $(q,\bfphi,\psi) \in \Vfsih$ there hold
\begin{subequations}
\label{SKM}
\begin{equation}\label{SKM1}
\intOfk{   q \Div \vu_h^k   } =0  ,
\end{equation}
\begin{equation}\label{SKM2}
\begin{split}
&  
 \vrf \intOfhk{ \big(\MDt \vu_h^k + \Div \vw_h^k \frac{\vu_h^{k*}}{2} \big) \cdot \bfphi}
\\& 
+\frac12 \vrf \intOfhk{   \big( \bfphi \cdot (\Grad \vu_h^k)   - \vu_h^k \cdot ( \Grad \bfphi)     \big)  \cdot \vv_h^{k-1}\circ \Xkm } 
\\& 
+ \intOfhk{ \bftau(\vu_h^k, p_h^k) : \Grad \bfphi  } 
+ \vrs  \intS{ \PDt \xi_h^k\psi}  
+a_s(\eta_h^{k+1},  \zeta_h^{k+1}, \xi_h^k ,\psi)
 =  0,
\end{split}
\end{equation}
\end{subequations}
where $\vv_h^{k-1} = \vu_h^{k-1} - \vw_h^k \circ\Xkminv$,  
$\vu_h^{k*} = 2\vu_h^{k-1}\circ\Xkm - \vu_h^k$,  $a_s$ is given in \eqref{abbs}, and $\zeta_h = - \Laph \eta_h$ is the (minus) discrete Laplace given by \eqref{laph}. 
\medskip 

\begin{Remark}\hfill

\begin{enumerate}
\item We omit the proof on how to identify the equivalence of \textbf{Scheme-R} and \textbf{Scheme-C} as it is similar to the proof of Lemma~\ref{Lwfref}.
\item 
In \textbf{Scheme-C} (or equivalently \textbf{Scheme-R}) we solve for each time step $t^k$, $k\in\{1,\dots, N\}$, the fluid variables $(\vu_h^k, p_h^k)$ in an explicit domain $\Ofh^k$ and solve the structure variable $\eta_h^{k+1}$,  which determines the fluid domain $\Ofh^{k+1}$ of the next time step $t^{k+1}$. 
This differs from many monolithic schemes defined in an implicit domain $\Ofh^k$ (or their equivalent form in the reference domain) when $\eta_h^k$ instead of $\eta_h^{k+1}$ is unknown at time step $t^k$. 
Such a kind of solver ``time splitting'' helps us to define a linear scheme without destroying the stability of the numerical solutions, see Theorem~\ref{Thm_Sta}.
\end{enumerate}
\end{Remark}
\begin{Remark}[{On the extension to 3D/2D}]
Many parts of our analysis are also valid when considering a three-dimensional fluid domain with a two-dimensional plate attached to it. However, the regularity of the (approximated) fluid domain is essentially weaker a priori. Observe that if the plate is two-dimensional, the discrete domain in space can not even be assumed to be uniformly Lipschitz continuous, as in two dimensions $W^{2,2}$ does not embed into Lipschitz functions.
\end{Remark}

\begin{Remark}
Note that we approximate the boundary deformation $\eta$ with a piecewise linear finite element space, resulting in a linear ALE mapping and a linear deformation of the fluid domain. Therefore, the geometry of the fluid domain is automatically captured at every time step, as every element $K\in \gridf$ is preserved as a triangle. 

Let us point out that the fourth order derivative in the structure (due to the bi-Laplacian term) is avoided by the introduction of a discrete Laplace operator, which maps a piecewise linear function space into the same space, see \eqref{laph}.  
\end{Remark}

\section{Stability}
\label{sec:stab}
In this section, we show the stability of the \textbf{Scheme-C} (or equivalently \textbf{Scheme-R}). We start with the following observation by recalling the discrete Laplace operator \eqref{laph}.
\begin{equation}\label{S1}
\begin{aligned}
&  \intS{\pdx \xi_h^{k} \pdx ( - \Laph \eta_h^{k+1})} 
= \intS{\Laph \xi_h^{k}  \Laph \eta_h^{k+1} } 
= \intS{\PDt \Laph \eta_h^{k+1} \Laph \eta_h^{k+1}  } 
\\& = \frac12  \intS{ \left( \PDt|\Laph \eta_h^{k+1}|^2  + \TS  |\PDt \Laph  \eta_h^{k+1}|^2 \right)  } 
\\& = \PDt  \intS{ \frac12|\Laph \eta_h^{k+1}|^2   } + \frac{\TS}2  \intS{ | \Laph  \xi_h^{k}|^2   } ,
\end{aligned}
\end{equation}
where we have used the algebraic equality 
\begin{equation}\label{algab}
(a-b) a =\frac12 (a^2 -b^2) + \frac12 (a-b)^2.
\end{equation}
Then, recalling \eqref{abbs} with the test function $\psi =\xi_h$ and thanks to \eqref{S1}, we find
\begin{equation}\label{norm_s}
\begin{split}
&  a_s(\eta_h^{k+1}, -\Laph \eta_h^{k+1}, \xi_h^{k}, \xi_h^{k})
\\&  =\PDt  \left( \frac{\gamma_1}{2}  \norm{ \pdx \eta_h^{k+1}}_{L^2(\Sigma)}^2 +  \frac{\gamma_2}{2}  \norm{ \Laph \eta_h^{k+1} }_{L^2(\Sigma)}^2 \right)
 + \gamma_3  \norm{ \pdx \xi_h^{k} }_{L^2(\Sigma)}^2 + D_s^{k},
\\&
\mbox{where }    D_s^{k} =  \frac{\gamma_1 \TS}{2} \intS{  \abs{\pdx \xi_h^{k}}^2 }
+ \frac{\gamma_2 \TS }{2} \intS{ \abs{ \Laph \xi_h^{k}}^2  } \geq 0.
\end{split}
\end{equation}

Now we are ready to show the energy estimates. 
\begin{tcolorbox}
\begin{Theorem}[Energy estimates]\label{Thm_Sta} 
Let $\{(p_h^{k},\vu_h^{k},\xi_h^k,\eta_h^{k+1})\}_{k=1}^{N}$ be the solution of \textbf{Scheme-C} (or equivalently, let $\{(\hp_h^{k},\hvu_h^{k},\xi_h^k,\eta_h^{k+1})\}_{k=1}^{N}$ be the solution of \textbf{Scheme-R}). Then we have the following stability result for all $m=1,\dots,N$
\begin{equation}\label{eq_Sta}
\begin{aligned}
    E_h^m + \TS \summ \left( 2\mu \norm{ (\Grad \vu_h^{k})^\rmS}_{L^2(\Ofh^k)}^2 + \gamma_3 \norm{\pdx  \xi_h^k }_{L^2(\Sigma)}^2 + D_{num}^k\right) =
    E_h^0
\end{aligned}
\end{equation}
where for any $k=0,\dots,N$ the total energy $E_h^k$ and the numerical dissipation $D_{num}^k$ read
\[
\begin{aligned}
& E_h^k = \frac{\vrf}{2}\norm{\vu_h^{k}}_{L^2(\Ofh^k)}^2  
+ \frac{\vrs}{2}\norm{\xi_h^k}_{L^2(\Sigma)}^2 
+ \frac{\gamma_1}{2}\norm{\pdx \eta_h^{k+1}}_{L^2(\Sigma)}^2 
+ \frac{\gamma_2}{2}\norm{\Laph \eta_h^{k+1}}_{L^2(\Sigma)}^2 
\\&
D_{num}^k=
 \frac{\vrf \TS }{2}  \intOfhk{  \mJkm \abs{\MDt \vu_h^k}^2  }
+ \frac{\vrs \TS}{2}  \norm{\PDt \eta_h^{k+1} }_{L^2(\Sigma)}^2
+ \frac{\gamma_1 \TS }{2}  \norm{\pdx  \xi_h^k }_{L^2(\Sigma)}^2 
+ \frac{\gamma_2\TS}{2}   \norm{\Laph  \xi_h^k }_{L^2(\Sigma)}^2 
\end{aligned}
\]
\end{Theorem}
\end{tcolorbox}
\begin{proof}
We test the numerical method \eqref{SKM} by $(q,\bfphi,\psi) = (p_h^k,\vu_h^k,\xi_h^k)$ to get 
\begin{equation}\label{k0}
\begin{aligned}
\vrf \intOfhkB{ \MDt \vu_h^k \cdot \vu_h^k +  \Div  \vw_h^k \Big( \vu_h^{k-1} \circ \Xkm  - \frac{ \vu_h^k}{2} \Big)\cdot \vu_h^k } 
\\ + 2 \mu   \norm{ (\Grad \vu_h^{k})^\rmS}_{L^2(\Ofh^k)}^2    + \vrs  \intS{ \PDt \xi_h^k \xi_h^k }  +  a_s(\eta_h^{k+1}, \zeta_h^{k+1}, \xi_h^k ,  \xi_h^k ) 
= 0.
\end{aligned}
\end{equation} 
Next, using the equality \eqref{algab} we get 
\begin{equation}\label{k1}
 \intS{ \PDt  \xi_h^k  \xi_h^k } = \frac12  \intS{ \PDt | \xi_h^k |^2}
+ \frac{\TS}2  \intS{ |\PDt  \xi_h^k |^2}, 
\end{equation} 
and
\begin{equation}\label{k2}
\begin{aligned}
& 
\intOfhkB{ \MDt \vu_h^k \cdot \vu_h^k  +  \Div  \vw_h^k  \Big( \vu_h^{k-1} \circ \Xkm  - \frac{ \vu_h^k}{2} \Big) \cdot \vu_h^k } 
\\& = 
\intOfhkB{ \MDt \frac{\abs{\vu_h^k}^2}{2}    + \frac{\TS}{2}\abs{\MDt \vu_h^k }^2  
+   \Div  \vw_h^k  ( \vu_h^{k-1} \circ \Xkm \cdot \vu_h^k -\frac{ \abs{\vu_h^k}^2}{2}  ) } 
\\&
=\PDt \intOfhk{ \frac12 |\vu_h^k|^2} + I_0,
\end{aligned}
\end{equation}
where we have also used  the discrete Reynolds transport formula \eqref{DRT}. Here, the term $I_0$ reads
\begin{align*}
& I_0 =  \intOfhkB{
  \frac{\TS}{2}\abs{\MDt \vu_h^k }^2  
+  \Div  \vw_h^k  \bigg(  \vu_h^{k-1} \circ \Xkm \cdot \vu_h^k - \frac{ \abs{\vu_h^k}^2}{2}  -  \frac{   \abs{ \vu_h^{k-1} \circ \Xkm}^2}{2}   \bigg)
 }
 \\& = 
  \intOfhk{
 \left( \frac{\TS}{2}  -  \frac{\TS^2}2  \Div  \vw_h^k  \right) \abs{\MDt \vu_h^k }^2  
 }
= \frac{\TS}2 \intOfhk{ \mJkm |\MDt \vu_h^k |^2   },
\end{align*}
where  we have used the \eqref{euler2}.

 Substituting \eqref{k1} and \eqref{k2} into \eqref{k0} and owing to \eqref{norm_s}, we derive
\begin{equation}\label{k3}
    \PDt E_h^k +  2\mu \norm{ (\Grad \vu_h^{k})^\rmS}_{L^2(\Ofh^k)}^2 + \gamma_3 \norm{\pdx  \xi_h^k }_{L^2(\Sigma)}^2 + D_{num}^k = 0.
\end{equation}
Finally, computing $\TS\, \summ \eqref{k3}$ yields \eqref{eq_Sta}, 
which completes the proof. 
\end{proof}
\medskip
The above stability estimate can be rewritten in the reference domain as
\begin{equation}\label{eq_Sta2}
\begin{aligned}
&\vrf  \intOref{  \frac12 \eta_h^m   |\hvu_h^{m}|^2  }  +\vrs \intS{ \frac12 \abs{ \xi_h^m}^2 }
+  \frac{\gamma_1}{2}  \norm{ \pdx \eta_h^{m+1} }_{L^2(\Sigma)}^2 
+  \frac{\gamma_2}{2}  \norm{ \Laph \eta_h^{m+1} }_{L^2(\Sigma)}^2
\\&
+ 2 \mu \TS \summ \intOref{ \eta_h^k \abs{(\Grad \hvu_h^k (\Jacob_h^k)^{-1})^\rmS }^2 } 
+  \gamma_3  \TS \summ   \norm{ \pdx  \xi_h^k }_{L^2(\Sigma)}^2 
\\&
+\frac{\TS^2 }2  \summ \vrf  \intOref{  \eta_h^{k-1} \abs{ D_t \hvu_h^k}^2 }
+\frac{\TS ^2}{2} \summ \vrs \intS{ \abs{ \PDt  \xi_h^k }^2   } 
\\&
 + \frac{\gamma_1 \TS ^2}{2} \summ  \intS{  \abs{\pdx  \xi_h^k }^2 }
+ \frac{\gamma_2 \TS ^2}{2}\summ  \intS{ \abs{ \Laph  \xi_h^k }^2  }  
\\& = 
\vrf  \intOref{  \frac12 \eta_h^0 |\vu_h^{0}|^2 } + \vrs \intS{ \frac12 \abs{ \xi_h^{0}}^2 } 
+  \frac{\gamma_1}{2}  \norm{ \pdx \eta_h ^{1}}_{L^2(\Sigma)}^2 
+  \frac{\gamma_2}{2}  \norm{ \Laph \eta_h^{1} }_{L^2(\Sigma)}^2,
 \end{aligned}
\end{equation}
\medskip
Note that $\eta_h$ appears on the left-hand-side (LHS) of the energy estimates \eqref{eq_Sta} (see also \eqref{eq_Sta2}) and determines if all terms on the LHS of the energy balance are non-negative or not. Therefore, it is important to preserve the positivity of $\eta_h$ in order to get a priori estimates. Actually, there exists a $T_0>0$ such that for all $T \leq T_0$ we have no contact between the upper surface and the bottom surface of $\Of$, see \cite[Lemma 5]{SS_FSI}. More precisely, if $\eta_h(0)>\underline{\eta}$, for every $c$ there exists a $T$, such that
 \begin{equation}\label{noc}
\eta_h \geq \underline{\eta}-c \quad \forall \; t \in [0,T].
\end{equation}
From Theorem~\ref{Thm_Sta} and the above assumption, we have the following uniform estimates.  
\begin{Corollary}\label{co_est}
Let the initial data satisfy $\vu_0 \in W^{1,2}( \Of(0);\R^2), \; \eta_0 \in W^{1,2}(\Sigma), \text{ and }  \xi_0 \in W^{1,2}(\Sigma)$. 
Let $(p_h,\vu_h,\xi_h,\eta_h)  = \{(p_h^{k},\vu_h^{k},\xi_h^{k},\eta_h^{k+1})\}_{k=1}^{N}$ be a  solution to \textbf{Scheme-C} (or equivalently $(\hp_h,\hvu_h,\xi_h,\eta_h)$ be a solution to \textbf{Scheme-R}) with $(\TS,h ) \in (0,1)^2$ and let \eqref{noc} hold.  Then we have the following uniform bounds.
\begin{equation}\label{ests}
\begin{aligned}
 &
\norm{\xi_h}_{L^\infty (0,T;L^2(\Sigma))} 
+ \norm{\pdx \eta_h}_{L^\infty (0,T;L^2(\Sigma))} 
+ \norm{\Laph \eta_h}_{L^\infty (0,T;L^2(\Sigma))}  \aleq 1,
\\& 
\norm{\eta_h^{-1}}_{L^\infty((0,T)\times \Sigma)} +
 \norm{\eta_h}_{L^\infty((0,T)\times \Sigma)} + 
\norm{\pdx \eta_h}_{L^\infty((0,T)\times \Sigma)} \aleq 1,
\\& 
\norm{\Jacob_h}_{L^\infty((0,T)\times \Sigma;\R^{2 \times 2})} 
+ \norm{\Jacob_h^{-1}}_{L^\infty((0,T)\times \Sigma;\R^{2 \times 2})}  
\aleq 1 ,
\\& 
\norm{\hvu_h}_{L^\infty (0,T;L^2(\Oref))} + 
\norm{(\Grad \hvu_h (\Jacob_h)^{-1})^\rmS  }_{L^2((0,T)\times \Oref)} \aleq 1 ,
\\& 
  \norm{\Grad \hvu_h }_{L^2((0,T)\times \Oref)} \aleq 1, \quad 
  \norm{\hvu_h}_{L^2(0,T;L^{q_1}(\Oref))} \aleq 1,
\\&
   \norm{ \xi_h}_{L^2 (0,T; L^\infty (\Sigma))} \aleq \norm{\Grad \hvu_h }_{L^2((0,T)\times \Oref)} \aleq 1 ,
\\&
  \norm{\hvw_h}_{L^\infty (0,T;L^2(\Oref))} +   \norm{\hvw_h}_{L^2(0,T;L^\infty (\Oref))} \aleq  1,
\\&
    \norm{\hvv_h}_{L^2(0,T;L^{q_1}(\Oref))} 
 +  \norm{\hvv_h}_{L^\infty (0,T;L^2(\Oref))} \aleq 1, \quad 
 \norm{\hvv_h}_{L^{q_2}(0,T;L^{q_1}(\Oref;\R^2))}\aleq 1 
.
\end{aligned}
\end{equation} 
for any $q_1\in [1,\infty)$ and  $q_2 \in [1,\infty)$. 
\end{Corollary}
\begin{proof}
Noticing that $\eta_h^1=\eta_h^0 + \TS \xi_h^0$, $\TS<1$, and the algebraic inequality $(a+b)^2 \leq 2(a^2+b^2)$, we know that 
\begin{align*}
& \frac{\gamma_1}{2}\norm{\pdx \eta_h^1}_{L^2(\Sigma)}^2 
+ \frac{\gamma_2}{2}\norm{\Laph \eta_h^1}_{L^2(\Sigma)}^2 
\\& 
\leq 
\gamma_1 \left( \norm{\pdx \eta_h^0}_{L^2(\Sigma)}^2 +
\TS^2 \norm{\pdx \xi_h^0}_{L^2(\Sigma)}^2 \right)
+ \gamma_2\left(\norm{\Laph \eta_h^0}_{L^2(\Sigma)}^2 
 + \TS^2 \norm{\Laph \xi_h^0}_{L^2(\Sigma)}^2  \right)
\\& < 
\gamma_1 \left( \norm{\pdx \eta_h^0}_{L^2(\Sigma)}^2 +
 \norm{\pdx \xi_h^0}_{L^2(\Sigma)}^2 \right)
+ \gamma_2\left(\norm{\Laph \eta_h^0}_{L^2(\Sigma)}^2 
 + \norm{\Laph \xi_h^0}_{L^2(\Sigma)}^2  \right)
 \\&\leq c(\gamma_1, \gamma_2, \norm{ \eta_0}_{W^{1,2}(\Sigma)}, \norm{ \xi_0}_{W^{1,2}(\Sigma)}) , 
\end{align*}
where we have used the stability of the Riesz projection operator in the last step, see \eqref{RieszE}. 
Therefore, the right-hand side of the energy estimate \eqref{eq_Sta} is uniformly bounded by a positive constant. Then, we have \eqref{ests}$_1$ and  \eqref{ests}$_4$ after noticing $\eta_h \geq \underline{\eta}>0$. Further, by Korn's inequality, Sobolev's inequalities \eqref{Sob},  the assumption \eqref{noc}, and triangular inequality, we get all the rest estimates. 
\end{proof}

\section{Interpolation operators}
A critical difficulty in convergence analysis is the appropriate choice of interpolation operators for the couple $(\vu,\xi)$, that inherit not only the kinematic coupling condition at the fluid-structure interface but also the divergence-free property of the velocity field. 
Before digging into this issue, we recall some analytic estimates that we will use frequently. First, the discrete Sobolev inequalities
\begin{subequations}\label{Sob}
\begin{equation}\label{P4}
 \norm{\hv}_{L^{q_1}(\Oref)} \aleq \norm{\hv}_{W^{1,2}(\Oref)} 
\mbox{ for } \hv \in W^{1,2}(\Oref),  \quad q_1\in[1, \infty), 
\end{equation}
\begin{equation}\label{P5}
 \norm{v}_{L^{\infty}(\Sigma)} \aleq \norm{v}_{W^{1,2}(\Sigma)}  \mbox{ for } v \in W_0^{1,2}(\Sigma).
\end{equation}
\end{subequations}
Next, we recall the standard projection error, see e.g. Boffi et. al~\cite{boffi}
\begin{equation}\label{proe}
\begin{aligned}
\norm{\Piq p - p }_{W^{k,s}} \aleq h \norm{p}_{W^{k+1,s}}, \; k=1,2, \; s \in [1,\infty]. 
\end{aligned}
\end{equation}
where $\Piq: \;  L^2(\Oref)   \; \mapsto \;   \Qfh$ is any suitable projection operator satisfying the above (see for example~\cite[Section 2.2]{boffi}), which we will use for the pressure. The projection operators for the velocities (both for the fluid and the solid) are much more complicated in this framework. It starts already with the necessity of a careful choice of the interpolation operator for solid deformation.

\subsection{Interpolation operator for the solid deformation}
For the solid we will use the Riesz projection as the interpolation operator. 
Let $\eta\in W^{1,2}(\Sigma)$, our Riesz projection operator $\Riesz$ reads
\begin{align}\label{Riesz1}
\intS{\pdx (\Riesz \eta - \eta)  \pdx \psi} =0  \quad \forall \; \psi \in \Vsh \text{ with } \intS{ \Riesz \eta}=\intS{ \eta}.
\end{align}
Recalling the discrete Laplace \eqref{laph} we know that for any $\eta \in W^{2,2}(\Sigma)$ it holds
\begin{align}\label{Riesz2}
\intS{(\Laph \Riesz \eta - \Lapx \eta)  \; \psi} = 0 \quad \forall \; \psi \in \Vsh.
\end{align}
Setting $\psi$ as $\Riesz \eta$ and $\Laph \Riesz \eta$ respectively in \eqref{Riesz1} and \eqref{Riesz2}, we obtain by H\"older's inequality that
\begin{equation}\label{RieszE}
\norm{\pdx \Riesz \eta}_{L^2(\Sigma)} \aleq \norm{\pdx \eta}_{L^2(\Sigma)}
\quad \mbox{ and } \quad 
\norm{\Laph \Riesz \eta}_{L^2(\Sigma)} \aleq \norm{\Lapx \eta}_{L^2(\Sigma)}, 
\end{equation}
where by $a\aleq b$ we mean $a\leq c b$ for a positive constant $c$ that is independent of the computational parameters $\TS$ and $h$. 
Further, 
if $\eta \in W^{3,2}(\Sigma)$, we have 
\begin{equation}\label{pro_eR}
\norm{\Laph \Riesz \eta - \Lapx \eta}_{L^2(\Sigma)} \aleq h \norm{\eta}_{ W^{3,2}(\Sigma)},
\end{equation}
 see the detailed proof in Lemma~\ref{lem:w22} for not only one-dimensional $\Sigma$ but also a multi-dimensional domain.

This concludes all necessary estimates that we need for the approximation of $\eta$ as well as the structure velocity $\xi$. Building on these properties we have to extend this projection into the fluid-reference domain. 

\subsection{Interpolation operator for the fluid velocity}
Given the divergence-free velocity field $\vu$ with the boundary condition $\vu|_{\Gamma_S} =\xi \er$, our aim here is to construct an interpolation operator 
 $\PiF:W^{1,2}(\Oref))\to \hVfh$, such that for $q\in [2,\infty)$
\[
h^{1-\frac{2}{q}}\norm{\hvu - \PiF \hvu}_{L^\gamma(\Oref)} + h  \norm{\Grad \hvu - \Grad \PiF \hvu}_{L^2(\Oref)} \leq \widehat{C} h^2 \norm{\Delta \hvu}_{L^2(\Oref)},\] 
which has to satisfy also the following two restrictions:
\begin{itemize}
\item Kinematic condition $\Pif \vu |_{\Gamma_S} =\Riesz \xi \er$. 
\item Weakly divergence-free condition. 
Here, one may naively consider the form $\intOref{\hq \Grad \Pif \hvu: \M}=0$. However, it is necessary to have  
$\intOref{\hq \Grad \Pif \hvu: \Mh}=0$ for the convenience of convergence analysis, see Remark \ref{pifmh}.
\end{itemize}
The next theorem takes care of the first bullet. For it we need the following lemma of analytical extensions, see \cite[Proposition 3.5]{KamSchSpe22}.

\begin{Lemma}
\label{lem:extension}
Assume that $\Omega_{\eta_h}$ is a given subgraph, with $\eta_h>\delta$. 
Let $\phi\in C_0^\infty((\Sigma\times { [0,\delta/2)};[0,\infty))$ such that $ \int_{\Sigma\times  [0,\delta/2) } \phi \dvol=1$. Then there is an extension operator $\Ecal:W^{k,p}(\Sigma)\to W^{k,p}(\Omega_\eta)$, for $k\in \{0,1,2,...\}$ and $1<p<\infty$, such that the following hold:
\begin{enumerate}
\item $\Ecal(\xi)(\xx,\xy)=(0,\xi(\xx))\text{ for } \xy \in [\delta,\infty)$.  
\item $\Div\Ecal(\xi)=\partial_2 \phi \intS{ \xi}$. 
\item $\norm{\nabla^k\Ecal(\xi)}_{L^p(\Omega_\eta)}\leq c \norm{\xi}_{W^{k,p}(\Sigma)}$,
\end{enumerate}
where $c$ depends on $p,k$ and $\delta$ only. 
\end{Lemma}
Note that $\Div\vu=0$, implying that
$
0=\int_{\Gamma_S(t)}v\cdot \nu \ds= \intS{\pd_t \eta}.
$
It implies that 
\begin{align}
\label{eq:meanv}
\Div \Ecal(\partial_t\eta)  =   \pd_2 \phi \intS{\pd_t \eta}  =0 \text{ for any }\eta_h\geq \delta \text{ as above.}
\end{align}
Our construction follows tightly \cite[Section 2.2 and Section 8.4]{boffi}, where more details on the notation and arguments can be found. Following the argumentation there it seems more natural to work on the computed Eulerian grid, that is the grid pushed forward by $\ALEh$. For that reason, we first introduce the auxiliary operator $\Pif$ on $\Omega_{\eta_h}$ that eventually becomes the basis for the desired operator $\PiF$.
\begin{Theorem}
\label{thm:interpol1}
Let the grids $\gridf$ and $\grids$ respectively 
defined on the domains $\Oref$ and $\Sigma$ be shape regular and quasi-uniform. Moreover, let $\Gamma_S=\{ (\xx,\eta_h(\xx)) : \xx\in \Sigma\}$ satisfy $ \min_\Sigma \eta_h \geq \delta $ and $\norm{\partial_{\xx} \eta_h}_{L^\infty} \leq L$ for some positive constants $\delta$ and $L$. 
Then there exists an interpolation operator
\begin{align*}
\Pif: \;  \; \FSIh \to \;  \FSIh, 
\end{align*}
that satisfies for $(\xi ,\hbfphi)\in \hFSIh$, $\gamma<\infty$ and ${\bfphi}:= \hbfphi\circ \ALEhinv$
\begin{align*}
 \norm{ \bfphi- \Pif \bfphi}_{L^\gamma(\Omega_{\eta_h})} + h  \norm{\Grad ( \bfphi  - \Pif \bfphi)}_{L^2(\Omega_{\eta_h})} &\aleq  h^2 \norm{\hbfphi}_{W^{2,2}(\Oref)}+ h^2 \norm{ \xi }_{W^{2,2}(\Sigma)}, 
\\
\norm{\Pif \bfphi}_{L^\gamma(\Omega_{\eta_h})} +  \norm{\Grad  \Pif \bfphi }_{L^2(\Omega_{\eta_h})} &\aleq\norm{\hbfphi}{W^{1,2}(\Oref)}+\norm{ \xi }_{W^{1,2}(\Sigma)},
\end{align*}
where the bounds depend linearly on $\frac{1}{\delta}, L, \frac{L}{\delta}$.
Moreover, we find 
$\Pif \bfphi (\xx,\eta_h(\xx)) =  (0,\Riesz  \xi (\xx)) $ on $\Sigma$ and
\begin{equation*}
\intOfh{q \Div \Pif \bfphi} = \intOfh{q \Div    \bfphi} \quad \forall \; q \in \Qfh.
\end{equation*}
\end{Theorem}
The above construction on the variable domain $\Omega_{\eta_h}$ implies the following corollary for the reference domain. 
\begin{Corollary}
\label{cor:interpol1}
Under the assumption of the Theorem \ref{thm:interpol1}, there exists
\begin{align*}
\Pifh: \;  \; \hFSIh \to \;  \hFSIh
\end{align*}
 satisfying for $(\xi,\hbfphi)\in \hFSIh$ and $\gamma <\infty$ that
\begin{align*}
 \norm{ \hbfphi- \Pifh \hbfphi}_{L^\gamma(\Oref)} + h  \norm{\Grad ( \hbfphi  - \Pifh \hbfphi)}_{L^2(\Oref)} &\aleq  h^2 \norm{ \hbfphi}_{W^{2,2}(\Oref)}+ h^2 \norm{\xi}_{W^{2,2}(\Sigma)}, 
\\
\norm{\Pifh \hbfphi}_{L^\gamma(\Oref)} +  \norm{\Grad  \Pifh \hbfphi }_{L^2(\Oref)} &\aleq\norm{\hbfphi}_{W^{1,2}(\Oref)}+\norm{\xi}_{W^{1,2}(\Sigma)},
\end{align*}
where the bounds depend linearly on $\frac{1}{\delta}, L, \frac{L}{\delta}$.
Moreover, we find 
$\Pifh \bfphi (\xx,1) =  (0,\Riesz \xi(\xx)) $ on $\Sigma$ and
\begin{equation*}
\intOref{\hq \nabla \Pifh \hbfphi : \M(\eta_h)} = \intOref{\hq \nabla    \hbfphi:\M(\eta_h)} \quad \forall \; \hq \in \hQfh.
\end{equation*}
\end{Corollary}

\begin{proof}[Proof of Theorem~\ref{thm:interpol1}]
The proof is split into two parts.

\noindent
{\bf Part I: construction of a Fortin operator on $\Omega_{\eta_h}$.}

We start with the operator $\widehat{\Pi}_1$ which is the piecewise affine interpolation operator on the reference grid constructed in~\cite[Section 2.2]{boffi}, that naturally preserves zero boundary values component wisely. In particular, we may define $\PiS\xi$ as
\[
\widehat{\Pi}_1(\hbfphi)(\xx,1)=\Big( \widehat{\Pi}_1(\hphi_1)(\xx,1),\widehat{\Pi}_1(\hphi_2)(\xx,1) \Big) =: (0,\PiS\xi(\xx)),
\]
which is by the construction of a function in $\Vsh$. Accordingly, we define
\[
\Pi_1\bfphi:=\widehat{\Pi}_1(\hbfphi)\circ\ALEh.
\]
In order to show the necessary bounds, we realize by \cite[equation (2.2.20)]{boffi} and  by the uniform Lipschitz bounds of $\eta_h$ that
\begin{align*}
\norm{\nabla(\Pi_1\bfphi-\bfphi)}_{L^2(\Omega_{\eta_h})}\leq c_{\gamma,L} \norm{\widehat{\Pi}_1\hbfphi-\hbfphi}_{W^{1,2}(\Oref)}\lesssim \min\Big\{\norm{\hbfphi}_{W^{1,2}(\Oref)},h\norm{\hbfphi}_{W^{2,2}(\Oref)}\Big\}
\end{align*}
and
\begin{align*}
\norm{\Pi_1\bfphi-\bfphi}_{L^2(\Omega_{\eta_h})}\lesssim \min\Big\{\norm{\hbfphi}_{L^2(\Oref)},h\norm{\hbfphi}_{W^{1,2}(\Oref)},h^2\norm{\hbfphi}_{W^{2,2}(\Oref)}\Big\}.
\end{align*}
This finishes the construction of $\Pi_1$.

Next, we construct $\Pi_2$. 
We start by recalling that $\hVfh$ is piecewise affine. Hence composed with $\ALEh$ these objects are not any more piecewise affine. But as by our assumptions $\ALEh$ is bi-Lipschitz, all necessary bounds for $\Pi_1$ are directly inherited from the bounds of $\widehat{\Pi}_1$ with an additional dependence on $\delta, L$. Let us focus on a generic reference cell $K\in \gridf$ with its bubble function $b_K\in \mathcal{P}^3(K)\cap W^{1,2}_0(K)$. It is obvious that 
\[
\norm{b_K}_{L^p(K)}\sim h^\frac{2}{p}\text{ and }\norm{\nabla b_K}_{L^p(K)}\sim h^{\frac{2}{p}-1}.
\]
Analog estimates with dependence on $\delta$ and $L$ do also hold for $b_K\circ \ALEh$ as $\eta_h$ is uniformly bounded.

Next, we show how to map the bubble function onto the current domain $\Ofh$ according to the change of geometry. 
Let 
\[
\mathcal{B}_{\eta_h}=\Big\{\sum_K a_K b_K\circ \ALEh : a_K\in \R^2\Big\}
\]
be the set of the potential bubble functions pushed forward by $ \ALEh $. 
Our aim is to find a projector $\Pi_2: W^{1,2}(\Omega_{\eta_h};\R^2)\to \mathcal{B}_{\eta_h}^d =\Big\{\sum_K a_K b_K\circ \ALEh : a_K\in \R^2\Big\}$ that satisfies
\begin{align*}
\sum_K\int_{\ALEh(K)} (\Pi_2 \bfphi-\bfphi)\cdot \nabla q \dvol =0,
\end{align*}
for all $q=\hq\circ \ALEh$, $\hq\in \hVfh$. 
Let $\hq=c+a'x_1+ a x_2$ on $K$ for some constants $a',c,a\in \R$. Then (here we take $\nabla$ as a column vector)
\[
\nabla q(x_1,x_2)= \begin{pmatrix}
a' -\frac{ax_2}{\eta_h^2(x_1)}\pdx  \eta_h(x_1) 
\\
\frac{a}{\eta_h(x_1)}
\end{pmatrix}
=
\begin{pmatrix}
1 & -\frac{x_2}{\eta_h^2(x_1)}\pdx  \eta_h(x_1)
\\
0  & \frac{1}{\eta_h(x_1)}
\end{pmatrix}
\begin{pmatrix}
a' 
\\
a
\end{pmatrix}=:A^{\eta_h}(x_1) \begin{pmatrix}
a' 
\\
a
\end{pmatrix}
\]
and thus 
\[
\bfphi\cdot \nabla q = (
\bfphi_1,
\bfphi_2
)
A^{\eta_h} \begin{pmatrix}
a' 
\\
a
\end{pmatrix}.
\]
This allows us to define $\Pi_2(\bfphi)|_{\ALEh(K)}=\sum_{K\in \gridf }\beta_K b_K\circ\ALEh$, where $\beta_K\in \R^2$ is determined by the equation
\[
\beta_K^\rmT\int_{\ALEh(K)} b_K\circ\ALEh A^{\eta_h} \dvol=\int_{\ALEh(K)} \bfphi^\rmT A^{\eta_h} \dvol.
\]
It is easy to check that $A^{\eta_h}$ is bounded from above and below by positive constants and 
\[
\frac{c_1h^{2}}{\norm{\eta_h}_\infty}\leq \int_{\ALEh(K)} b_K\circ\ALEh(\xx,x_2) \frac{1}{\eta_h(\xx)} \dvol \leq \frac{c_2h^{2}}{\delta}
\]
Consequently the matrix $M_K=\int_{\ALEh(K)}b_K\circ\ALEh A^{\eta_h}\dvol$ is invertible with
\[
\abs{M^{-1}_K}\leq \frac{1}{\abs{\det(\int_{\ALEh(K)} b_K \circ\ALEh  A^{\eta_h} \dvol)}}\int_{\ALEh(K)} b_K \circ\ALEh \abs{A^{\eta_h}}\dvol \leq ch^{-2}, 
\]
where $c$ depends linearly on $L$ and $\frac{L}{\delta}$.
Further, we find 
\[
\abs{\beta_K}\leq \abs{M^{-1}_K}\norm{A^{\eta_h} }_\infty\norm{\bfphi}_{L^1(\ALEh(K))}\leq ch^{-2}\norm{\bfphi}_{L^1(\ALEh(K))}\leq ch^{-\frac{2}{p}}\norm{\bfphi}_{L^p(\ALEh(K))}
\]
where in the last step we have used Jensen's inequality.
Using the above estimate, we have
\begin{align*}
\norm{\nabla \Pi_2 \bfphi}_{L^p(\ALEh(K))}\leq \abs{\beta_K}\norm{\nabla (b_K\circ\ALEh )}_{L^p(\ALEh(K))}\leq ch^{-1}\norm{\bfphi}_{L^p(\ALEh(K))},
\end{align*}
and
\begin{align*}
\norm{ \Pi_2 \bfphi}_{L^p(\ALEh(K))}\leq \abs{\beta_K}\norm{(b_K\circ\ALEh )}_{L^p(\ALEh(K))}\leq c\norm{\bfphi}_{L^p(\ALEh(K))},
\end{align*}
which allows us to follow the arguments at the end of~\cite[Section 8.4]{boffi} to gain the expected estimates and bounds for the operator:
\[
\PiF(\bfphi):=\Pi_1(\bfphi)+\Pi_2( {\bfphi}-\Pi_1(\bfphi)).
\]

\noindent
{\bf Part II: a Fortin operator with appropriate boundary values}


By construction $\PiF\bfphi(\xx,\eta_h(\xx))= \Pi_1\bfphi(\xx,\eta_h(\xx))=:(0,\PiS\xi(\xx))$, with $\PiS$ being an interpolation operator for $\Vsh$ with natural stability properties and error bounds. 
The problem is that, unlike $\Riesz$, the operator $\PiS$ does not have the required second-order estimates (in particular Lemma~\ref{lem:w22} does not hold). Nevertheless by the orthogonality of the error for $\Riesz$ and the estimates of first order for $\PiS\xi(\xx)$, we find that
\begin{align}
\label{eq:boundary1}
\norm{\pdx(\Riesz\xi-\PiS\xi)}_{L^2(\Sigma)} \leq  \norm{\pdx (\xi-\PiS\xi)}_{L^2(\Sigma)}\leq ch \norm{\Lapx  \xi}_{L^2(\Sigma)}.
\end{align}
The desired projector turns out to be the solution to a discrete Stokes problem: 
{We derive it for $\eta_h$, $\bfphi$ and $\xi_h$ fixed by minimizing
\[
 \int_{\Omega_{\eta_h}}\abs{\nabla(\psi_h-\bfphi)}^2\dvol 
\]
over the class of all $\psi_h\in \Vfh$, with $\psi_h(\xx,\eta_h(\xx))=(0,\xi_h(\xx))$ on $\Sigma$, which satisfy the discrete divergence-free property: $\int_{\Ofh } \Div(\psi_h)\cdot q\dvol$ for all $q\in \Qfh$.
The minimizer is then defined as $\Pif \bfphi$. The respective Euler-Lagrange equation becomes the discrete solution to an approximate Stokes problem
\[
\intOfh{ \nabla (\Pif \bfphi-\bfphi) \cdot \nabla \psi_h }=0,
\]
for all $\psi_h\in \Vfh$ with zero boundary values and which are discretely divergence-free.}
The error of the projector is of two kinds. The first is the error stemming from the prescribed boundary values, and the second is the discretization error. 
For the first, we take the linear divergence-free extension $
\Ecal(\Riesz\xi-\xi)$ given by Lemma~\ref{lem:extension}. Now we can take  $\psi_h=\PiF(\bfphi+\Ecal(\Riesz\xi-\xi))$ as competitor in the minimization. Indeed, as $(0,\xi(\xx))=\bfphi(\xx,\eta_h(\xx))$, we find that $\PiF(\bfphi+\Ecal(\Riesz\xi-\xi))(\xx,\eta_h(\xx))=\PiF(\Ecal(\xi))(\xx,\eta_h(\xx))=(0,\PiS(\xi)(\xx)$ in $\Sigma$. This implies (as the projector is the minimizer) that
\begin{align*}
\norm{\nabla(\Pif\bfphi-\bfphi)}_{L^2(\Ofh)}^2&\leq  \norm{\nabla(\PiF(\bfphi+\Ecal(\Riesz\xi-\xi))-\bfphi)}_{L^2(\Ofh)}^2
\\
&\leq c\norm{\Grad \bfphi - \Grad \PiF \bfphi}_{L^2(\Ofh)}^2+c\norm{\nabla(\PiF\Ecal(\Riesz\xi-\xi))}_{L^2(\Ofh)}^2.
\end{align*}
The first term is estimated directly by the properties of the Fortin operator. The second one is by the stability of the Fortin operator, Lemma~\ref{lem:extension} and \eqref{eq:boundary1}.

\end{proof}
\begin{Remark}[On the importance of the proper choice of an interpolation operator]
The deep reason why the interpolation has to be solved as a discrete PDE, is that the solid matter and the fluid matter have totally different properties, even so they are coupled. Our scheme follows the direct path that is also used in the existence theory, where already in the approximation the coupling and the different matters are simultaneously (monolithically) solved. The fact that this uniform (and linear) approximation does indeed converge properly can only be revealed by imitating the coupling between two solutions of independent PDEs. This imitation is exactly performed by solving a discrete boundary value problem.
\end{Remark}

The last step for the interpolation of $\vu$ is the correction of the divergence due to the change of variables. For that, we use another analytic tool developed in \cite[Theorem 3.3]{KamSchSpe22}. It is the so-called universal Bogovskij operator. Universal it is, because it is independent of the particular (Lipschitz) geometry. We cite the important estimate in the following lemma.
\begin{Lemma}
\label{lem:bogovskij}
There is an operator $\Bcal:\{f\in L^p(\Omega_{\eta_h}):\intOfh{f} =0 \} \to W^{1,p}_0(\Omega_{\eta_h})$ for any $\Omega_{\eta_h}$ for $1<p<\infty$ that is a given subgraph, with $\min_\Sigma \eta_h>\delta$ and $\norm{\partial_x \eta_h}_{L^\infty}\leq L$, such that the following hold:
\begin{enumerate}
\item $\Div\Bcal(f)=f$.
\item $\norm{\Bcal(f)}_{W^{1,p}(\Omega)}\leq \norm{f}_{L^p}$.
\end{enumerate}
\end{Lemma}

\noindent
The above lemma and Corollary \ref{cor:interpol1} lead to the final statement of this section.  
\begin{Theorem}
\label{thm:projection-velocity}
Let $\Ofh\subset \R^2$ be a subgraph and let the assumptions of Theorem \ref{thm:interpol1} hold. 
Then there exists 
\begin{align*}
\PiF: \;  \; \hFSI\to \;  \hFSIh, 
\end{align*}
 satisfying for $(\xi,\vu)\in \FSI$ and $\gamma <\infty$ that  
\begin{align}
\label{iu}
\begin{aligned}
 \norm{ \hvu- \PiF \hvu}_{L^\gamma(\Oref)} + h  \norm{\Grad ( \hvu- \PiF \hvu)}_{L^2(\Oref)} &\aleq  h^2 \norm{\hvu}_{W^{2,2}(\Oref)}+ h^2 \norm{\xi}_{W^{2,2}(\Sigma)}+ h\norm{\eta-\eta_h}_{W^{1,2}(\Sigma)}, 
\\
\norm{\PiF \hvu}_{L^\gamma(\Oref)} +  \norm{\Grad  \PiF \hvu}_{L^2(\Oref)} &\aleq\norm{\hvu}_{W^{1,2}(\Oref)}+\norm{\xi}_{W^{1,2}(\Sigma)} + \norm{\eta-\eta_h}_{W^{2,2}(\Sigma)},
\end{aligned}
\end{align}
where the bounds depend linearly on $\frac{1}{\delta}, L, \frac{L}{\delta}$.
Moreover, we find 
$\PiF \vu (\xx,1) =  (0,\Riesz \xi(\xx)) $ on $\Sigma$ and
\begin{equation}\label{P3}
\int_{\Oref} {\hq \nabla \PiF \vu : \M(\eta_h)} \, dx= 0
\quad \forall \; \hq \in \hQfh. 
\end{equation}
\end{Theorem}
\begin{proof}
The proof takes the function $\bfphi:=\vu\circ \ALEh\circ \ALE^{-1} - \Bcal(\Div (\vu\circ \ALEh\circ \ALE^{-1}))$.
By the Gauss theorem we note that
\[
\intOfh{\Div (\vu\circ \ALEh\circ \ALE^{-1}) }=0,
\]
hence $\Bcal$ is well defined and so
$
\Div \bfphi=0\text{ on }\Ofh.
$
Further
\[
\Div (\vu\circ \ALEh\circ \ALE^{-1})=\Big(\frac{\eta}{\eta_h}-1\Big)\partial_2 u_2\circ \ALEh\circ \ALE^{-1} +\partial_1\Big(\frac{\eta}{\eta_h}\Big)x_2\partial_2 u_1\circ \ALEh\circ \ALE^{-1},
\]
which implies as $W^{1,\infty}(\Sigma)\subset W^{2,2}(\Sigma)$ that
\[
\norm{\Div \vu\circ \ALEh\circ \ALE^{-1}}_{L^2(\Omega_{\eta_h})}\leq c\norm{\eta-\eta_h}_{W^{2,2}(\Sigma)}\norm{\nabla\hvu}_{L^2(\Oref)}.
\]
Hence by Lemma~\ref{lem:bogovskij} and a change of variable we find
\[
\norm{\hbfphi-\hvu}_{W^{1,2}(\Oref)}\leq c\norm{\eta-\eta_h}_{W^{2,2}(\Sigma)}\norm{\nabla\hvu}_{L^2(\Oref)}\text{ and } \hbfphi(\xx,1)= \xi(\xx).
\]
Then, we define $\PiF \hvu = \Pifh{\hbfphi}$ for which now the result follows from the previous estimates and Corollary~\ref{cor:interpol1}.
\end{proof}

\section{Error estimates}
In this section, we study the error between the numerical solution $(\hp_h,\hvu_h,\xi_h,\eta_h)$ of \textbf{Scheme-R} and its target smooth solution $(\hp,\hvu,\xi,\eta)$. Here we assume the existence of a smooth solution of \eqref{pde_f}--\eqref{pde_bdc} in the following class 
\begin{equation}\label{STClass}
\left\{
\begin{aligned}
& \eta>\underline{\eta}, \eta \in L^2(0,T; W^{3,2}(\Sigma))\cap  W^{2,2}(0,T; W^{2,2}(\Sigma)),
\\&  
\hvu \in   
L^\infty (0,T; W^{1,2}(\Oref;\R^2)) \cap 
L^2 (0,T; W^{2,2}(\Oref;\R^2))
\\&
\pdt  \hvu \in  L^2 (0,T;W^{1,2}( \Oref;\R^2 )) , \; 
\\&
\hp \in L^\infty( 0,T; L^2(\Oref) ), \; \Grad p \in L^2 ( (0,T)\times \Oref ). 
\end{aligned}
\right.
\end{equation}

\subsection{The time projection}
Very relevant in this highly nonlinear coupled system is to choose a set of appropriate time-value $t_k^{\TS}$, $k=1,\dots,N_T$, at which we will compare the continuous equation with its numerical approximation. 
For a given $\TS$ and $k$, we denote 
\[
(\hp^k,\hvu^k,\xi^k,\eta^k):=(\hp,\hvu,\xi,\eta)(t_k^{\TS}).
\]
Then, according to our smoothness assumption \eqref{STClass}, we may choose the value $t_k^{\TS}\in [k\TS,(k+1)\TS)$ in such a way that
\begin{align*}
&\TS\Big(\norm{\hvu(t_k^{\TS})}_{W^{2,2}(\Oref)}^2+
\norm{\partial_t\hvu(t_k^{\TS})}_{W^{1,2}(\Oref)}^2+\norm{\partial_t^2\hvu(t_k^{\TS})}_{L^2(\Oref)}^2 + 
\\ & \quad \; \norm{\eta(t_k^{\TS})}_{W^{2,3}(\Sigma)}^2+\norm{\xi(t_k^{\TS})}_{W^{2,4}(\Sigma)}^2+\norm{\partial_t\xi(t_k^{\TS})}_{W^{2,2}(\Sigma)}^2\Big)
\\
&\leq 
\int_{k\TS}^{(k+1)\TS}\Big(\norm{\hvu(t)}_{W^{2,2}(\Oref)}^2+
\norm{\partial_t\hvu(t)}_{W^{1,2}(\Oref)}^2+\norm{\partial_t^2\hvu(t)}_{L^2(\Oref)}^2 
\\& \qquad \qquad \quad + \norm{\eta(t)}_{W^{2,3}(\Sigma)}^2+\norm{\xi(t)}_{W^{2,2}(\Sigma)}^2\norm{\partial_t\xi(t))}_{W^{2,2}(\Sigma)}^2\Big)\, dt,
\end{align*}  
which is possible to find by the continuity of the integral, if the right-hand side is bounded. In particular, we find that
\begin{align}
\label{eq:tk}
\begin{aligned}
&\TS\sum_{k=1}^{N_T}\Big(\norm{\hvu^k}_{W^{2,2}(\Oref)}^2+
\norm{\partial_t\hvu^k}_{W^{1,2}(\Oref)}^2+\norm{\partial_t^2\hvu^k}_{L^2(\Oref)}^2 + \norm{\eta^k}_{W^{2,2}(\Sigma)}^2+\norm{\partial_t\xi^k}_{W^{2,2}(\Sigma)}^2\Big)
\\
&\leq \int_{0}^{T}\big(\norm{\hvu}_{W^{2,2}(\Oref)}^2+
\norm{\partial_t\hvu}_{W^{1,2}(\Oref)}^2+\norm{\partial_t^2\hvu}_{L^2(\Oref)}^2 + \norm{\eta}_{W^{2,2}(\Sigma)}^2+\norm{\partial_t\xi)}_{W^{2,2}(\Sigma)}^2\big)\, dt
\end{aligned}
\end{align}
Actually, this right-hand side summarizes our regularity assumptions on the solution. All the above regularity requirements do follow from these assumptions.
\begin{Remark}[On the regularity assumptions]\label{rmq_reg}
When comparing the assumptions on the smooth solution with the theory for the heat/wave equation (or the 2D/Navier-Stokes equation), one realizes that we have the same regularity assumptions for the fluid as in the non-moving case. For the plate, which also deduces the domain essentially one more time-derivative has to be assumed, as nonlinear equations of a similar type can be expected.
\end{Remark}

\subsection{Main result}
Before introducing the main result, let us denote the following error terms for each time step $k \in\{1,\dots,N_T\}$.
\begin{equation}\label{ers}
\begin{aligned}
&e_p^k 	=   \hp_h^k  -\hp^k   =(\hp_h^k  	- \Piq \hp^k)  	+ (\Piq \hp^k -  \hp^k )      	=: \delta_p^k 	+ I_p^k , \\ 
&\eu^k  	= \hvu_h^k -\hvu^k  =(\hvu_h^k 	- \PiF \hvu^k) 	+ (\PiF \hvu^k -  \hvu^k )   		=: \delta_\vu^k 	+ I_\vu^k ,   \\ 
&e_\xi^k   	=\xi_h^k -\xi^k   	= (\xi_h^k 	- \Riesz  \xi^k )  	+ (\Riesz  \xi^k -  \xi^k)    		=: \delta_\xi^k 	+ I_\xi^k ,\\
&e_\eta^k  	= \eta_h^k -\eta^k	=(\eta_h^k	- \Riesz \eta^k)	+ (\Riesz \eta^k - \eta^k ) 	=: \delta_\eta^k	+ I_\eta^k ,  \\ 
&e_\zeta^k  	= \zeta_h^k -\zeta^k	=(\zeta_h^k	+ \Laph \Riesz \eta^k)	+ (-\Laph \Riesz \eta^k - \zeta^k )	=: \delta_\zeta^k	+ I_\zeta^k ,  \\ 
\end{aligned} 
\end{equation}
where $\zeta_h = - \Laph \eta_h$ and $\zeta = - \Lapx \eta$. 
Now we are ready to present the main result of the paper. 
\begin{tcolorbox}
\begin{Theorem}[Convergence rate]\label{theorem_conv_rate}
Let $\{(\hp_h^k, \hvu_h^k,\xi_h^k,\eta_h^{k+1})\}_{k=1}^{N_T}$ be the solution of {\bf Scheme-R} \eqref{SKM_ref}, and let  $(\hvu,\hp,\xi,\eta)(t)$, $t\in(0,T)$, be a strong  solution of \eqref{pde_f}--\eqref{pde_bdc} belonging to the class \eqref{STClass}. 
Then for any $m \in \{1,\cdots,N_T\}$ it holds
\begin{align*}
& \frac12 \vrf \intOref{|e_\vu^m|^2 \eta_h^m}   + \frac12  \intSB{ \vrs \abs{e_\xi^m}^2 +  \gamma_1 |\pdx e_\eta^{m+1} |^2 +  \gamma_2 \abs{ e_\zeta^{m+1}}^2}
\\& +  2 \mu   \TS \summ \intOref{\left| \Grad e_\vu^k (\Jacob_h^k)^{-1}\right|^2} 
+  \gamma_3 \vrs \TS \summ \intS{\abs{\pdx \delta_\xi^k}^2} 
\aleq \TS^2 +h^2 . 
\end{align*}
In particular, we have the following convergence rates
\begin{multline}\label{CR}
 \norm{e_\vu}_{L^\infty(0,T;L^2(\Oref;\R^2))}
+\norm{e_\xi}_{L^\infty(0,T;L^2(\Sigma))}
+\norm{\pdx e_\eta}_{L^\infty(0,T;L^2(\Sigma))}
+\norm{e_\zeta}_{L^\infty(0,T;L^2(\Sigma))}
\\
+\norm{\Grad e_\vu}_{L^2((0,T)\times \Oref;\R^{2\times 2})}
+\gamma_3 \norm{\pdx e_\xi}_{L^2((0,T)\times \Sigma;\R^2)}
\aleq \TS +h .
\end{multline}
\end{Theorem}
\end{tcolorbox}

\begin{proof}
First, we subtract the weak formulation \eqref{wf2_ref} from the numerical scheme \eqref{SKM_ref2} and get 
\begin{equation}\label{EEQ}
\begin{aligned}
& \intOref{  \vrf (\eta_h^k \PDt \euk  + \frac12 \PDt \eta_h^k \eu ^{k*}  ) \cdot \hbfphi  }  
+ 2 \mu \intOref{ \big( \Grad \euk  (\Jacob_h^k)^{-1} \big)^\rmS: \big(\Grad \hbfphi (\Jacob_h^k)^{-1} \big)\eta_h^k  } 
\\& 
+ \vrs\intS{\PDt e_\xi^k \psi}  
+a_s(e_\eta^{k+1},e_\zeta^{k+1},e_\xi^k,\psi)
= - \sum_{i=1}^7 R^k_i( \hbfphi, \psi),
\end{aligned}
\end{equation}
where 
\begin{equation}\label{RS}
\begin{aligned}
&R^k_1( \hbfphi, \psi)=  \vrf \intOref{\big( e_\eta^k \pdt \hvu^k+ \eta_h^k (\PDt \hvu^k -\pdt\hvu^k ) \big) \cdot \hbfphi},
\\ & 
R^k_2( \hbfphi, \psi) = \frac12\vrf \intOref{\left( (e_\xi^{k-1} - \TS \PDt \xi^k )\hvu^{k*}- \TS \pdt \eta^k \PDt \hvu^k \right) \cdot \hbfphi},
\\& 
R^k_3( \hbfphi, \psi)=\frac12 \vrf \intOref{ \Big( \hbfphi \cdot   (\Grad \euk ) 
-\euk   \cdot   (\Grad \hbfphi) \Big)
 \cdot  (\Jacob_h^k)^{-1} \hvv_h^{k-1}  \eta_h^k}
\\& \quad 
+ \frac12 \vrf \intOref{ \Big( \hbfphi \cdot (\Grad \hvu^k) - \hvu^k \cdot (\Grad\hbfphi) \Big) \cdot \left(    (\Jacob_h^k)^{-1} \hvv_h^{k-1}  \eta_h^k-   (\Jacob^k)^{-1}\hvv^{k} \eta^k \right)  }, 
 \\& 
R^k_4( \hbfphi, \psi) 
=\intOref{  e_p^k \Grad \hbfphi : \Mhk } 
+ \intOref{ \hp^k \Grad \hbfphi : \big(\Mhk -  \Mk\big)} ,
\\& 
R^k_5( \hbfphi, \psi)=2 \mu \intOrefB{\big( \Grad \hvu (\Jacob_h^k)^{-1} \big)^\rmS: (\Grad \hbfphi (\Jacob_h^k)^{-1} )\eta_h^k
-\big( \Grad \hvu (\Jacob)^{-1} \big)^\rmS: (\Grad \hbfphi (\Jacob^k)^{-1} )\eta^k} , 
\\& 
R^k_6( \hbfphi, \psi)=\vrs\intS{(\PDt \xi^k -\pdt \xi^k  ) \psi} ,
\\&
R^k_7( \hbfphi, \psi)= - \gamma_1 \intS{\Lapx (\eta^{k+1} -\eta^k) \; \psi}
- \gamma_2 \intS{\Lapx (\zeta^{k+1} -\zeta^k)\; \psi}.
\end{aligned}
\end{equation}
The precise justification of \eqref{EEQ} is given in Appendix~\ref{app_ee1}. 
By setting $(\hbfphi,\psi) = (\delta_\vu^k, \delta_\xi^k)$ in \eqref{EEQ} and sum up from $k=1$ to $m$ we derive 
\begin{equation}\label{EEQ2}
\begin{aligned}
-  \TS \summ  \sum_{i=1}^7 R^k_i( \delta_\vu^k, \delta_\xi^k)
 &=  \TS \summ  \intOref{  \vrf (\eta_h^k \PDt \euk  + \frac12 \PDt \eta_h^k \eu ^{k*}  ) \cdot \delta_\vu^k  }  
\\ & + 2 \mu  \TS \summ  \intOref{ \big( \Grad \euk  (\Jacob_h^k)^{-1} \big)^\rmS: \big(\Grad \delta_\vu^k (\Jacob_h^k)^{-1} \big)\eta_h^k  } 
\\& 
+  \TS \summ  \vrs\intS{\PDt e_\xi^k \delta_\xi^k}  
+ \TS \summ  a_s(e_\eta^{k+1},e_\zeta^{k+1},e_\xi^k,\delta_\xi^k).
\end{aligned}
\end{equation}
Further, applying the algebraic equalities \eqref{algab} and \eqref{IM6} to the above right-hand-side, we get (similarly as was performed for the stability estimate)
\begin{equation}\label{REI1}
\begin{aligned}
& -  \TS \summ  \sum_{i=1}^7 R^k_i( \delta_\vu^k, \delta_\xi^k)
=
 \TS \summ \intOref{  \vrf \big(\eta_h^k \PDt (\delta_\vu^k+I_\vu^k)  + \frac12 \PDt \eta_h^k (\delta_\vu^{k*}+I_\vu^{k*})  \big) \cdot \delta_\vu^k  }  
\\& \quad 
+ 2 \mu  \TS \summ \intOref{   \big( \Grad (\delta_\vu^k+I_\vu^k)  (\Jacob_h^k)^{-1} \big)^\rmS: (\Grad \delta_\vu^k (\Jacob_h^k)^{-1} ) \eta_h^k }  
\\& \quad
+  \TS \summ \vrs\intS{\PDt (\delta_\xi^k+I_\xi^k) \;\delta_\xi^k}  
+ \TS \summ a_s(e_\eta^{k+1},e_\zeta^{k+1},e_\xi^k,\delta_\xi^k)
\\& = \delta_E^m - \delta_E^0  + \TS \summ  D_{phys}^k+ \TS \summ  D_{num}^k + G_f+ G_s,
\end{aligned}
\end{equation}
where 
\begin{align*}
\delta_E^k =&   \intOref{ \frac12 \vrf \eta_h^k  |\delta_{\vu}^k|^2  } 
  + \frac12 \intSB{  \vrs |\delta_\xi^k|^2+ \gamma_1 |\pdx \delta_\eta^{k+1}|^2  
	+ \gamma_2|\delta_\zeta^{k+1}|^2 },
\\
 \delta_D^k = &   2 \mu \intOref{ \eta_h^k |\big(\Grad \delta_{\vu}^k (\Jacob_h^k)^{-1}\big)^\rmS|^2 } 
  +  \gamma_3 \intS{|\pdx \delta_\xi^k|^2},
\\
D_{num}^k = &
 \frac{\TS}{2}\vrf  \intOref{  \eta_h^{k-1}  |\PDt \delta_{\vu}^k|^2  } 
 +	 \frac{\TS}{2} \intSB{\vrs|\PDt \delta_\xi^k|^2+ \gamma_1|\PDt \pdx \delta_\eta^{k+1}|^2+\gamma_2 |\PDt \delta_\zeta^{k+1}|^2} \geq 0,
\\
G_f =&
\TS \summ  \intOref{  \vrf \big(\eta_h^k \PDt I_\vu^k  + \frac12 \PDt \eta_h^k I_\vu^{k*} \big) \cdot \delta_\vu^k  }  
\\&  + 2 \mu \TS \summ   \intOref{   \big( \Grad I_\vu^k  (\Jacob_h^k)^{-1} \big)^\rmS: (\Grad \delta_\vu^k (\Jacob_h^k)^{-1} ) \eta_h^k }  ,
\\ G_s=&
\gamma_1 \TS \summ \intS{\pdx \delta_\eta^{k+1} \pdx (\PDt \eta^{k+1} - \pdt\eta^k )} -
\gamma_2 \TS \summ\intS{ \delta_\zeta^{k+1} \Lapx (\PDt \eta^{k+1} - \pdt\eta^k )}
\\&  + \TS \summ  \intS{\vrs \PDt I_\xi  \delta_\xi^k}.
\end{align*}

Next, we reformulate \eqref{REI1} in the following form. 
\begin{equation}\label{REI}
\delta_E^m - \delta_E^0
+  \TS \summ  \delta_D^k
   +\TS\summ  D_{num}^k
=  - \TS \summ \sum_{i=1}^7 R^k_i -G_f - G_s.
\end{equation}
Then, by Young's inequality, H\"older's inequality, the interpolation error in Theorem~\ref{thm:projection-velocity}, and the uniform bounds \eqref{ests}, we estimate the right-hand-side of the above equation as
\begin{equation}\label{res}
\Abs{ \TS \summ \sum_{i=1}^7 R^k_i + G_f + G_s} 
\aleq \TS^2 + h^2  
+ c \TS \summ \delta_E^m 
+ 2 \alpha  \mu   \TS \summ \intOref{\left| \Grad \delta_\vu^k (\Jacob_h^k)^{-1}\right|^2 \eta_h^k}
,
\end{equation}
see Appendix \ref{app_res}. 
Further, substituting  the above estimate into \eqref{REI} and noticing the initial error $\delta_E^0=0$, we get (using also the lower bound of $\eta,\eta_h$) that
\begin{align*}
\delta_E^m  + (1-\alpha) 2 \mu \TS \summ \intOref{\left| \Grad \delta_\vu^k (\Jacob_h^k)^{-1}\right|^2 \eta_h^k } 
+  \gamma_3 \vrs \TS \summ \intS{\abs{\pdx \delta_\xi}^2} 
\aleq \TS^2 +h^2 + \TS \summ \delta_E^k. 
\end{align*}
By choosing any $\alpha\in(0,1)$ and using Gr\"onwall's inequality, we get
\begin{align*}
\delta_E^m  +  \TS \summ \delta_D^k
\aleq \TS^2 +h^2 . 
\end{align*}
Recalling the interpolation errors (Theorem~\ref{thm:projection-velocity} and \eqref{pro_eR}) and the regularity of the strong solution \eqref{STClass} we get
\begin{align*}
& \frac12 \vrf \intOref{|I_\vu^m|^2 \eta_h^m}   + \frac12  \intSB{ \vrs \abs{I_\xi^m}^2 +  \gamma_1 |\pdx I_\eta^{m+1} |^2 +  \gamma_2 \abs{ I_\zeta^{m+1} }^2}
\\& +     \TS \summ \left( 2 \mu\intOref{\left| \Grad I_\vu^k (\Jacob_h^k)^{-1}\right|^2 \eta_h^k } 
+    \gamma_3 \vrs \intS{\abs{\pdx I_\xi^k}^2} \right)
 \aleq  h^2.
\end{align*}
Finally, due to the triangular inequality, we sum up the previous two estimates and get
\begin{equation}\label{t0est}
    \begin{aligned}
& \frac12 \vrf \intOref{|e_\vu^m|^2 \eta_h^m}   + \frac12  \intSB{ \vrs \abs{e_\xi}^2 +  \gamma_1 |\pdx e_\eta^{m+1} |^2 +  \gamma_2 \abs{ e_\zeta^{m+1}}^2}
\\& +     \TS \summ \left( 2 \mu\intOref{\left| \Grad e_\vu^k (\Jacob_h^k)^{-1}\right|^2 \eta_h^k } 
+  \gamma_3 \vrs  \intS{\abs{\pdx \delta_\xi^k}^2}  \right)
\\& \aleq \TS^2 +h^2,  
\end{aligned}
\end{equation}
which provides the proof for small $T\leq T_0$ such that   \eqref{noc} is valid. 

Next, we show that $T$ can be arbitrarily large if $\eta\geq \underline{\eta}$ on $[0,T]$. 
We start with a fixed $T_0$ such that $\eta_h\geq \frac{\underline{\eta}}{2}$, this can be found by \cite[Lemma 5]{SS_FSI}.
Then, by the above estimate \eqref{t0est}, we know that 
\[
\norm{e_\eta(T_0)}_{L^\infty}\leq c(\TS+h),
\]
where the constant $c$ depends on the lower bound $\frac{\underline{\eta}}{2}$. Recalling $\eta(t) \geq \underline{\eta}$ we know that
\[
\eta_h(T_0)\geq \underline{\eta}-c(\TS+h),
\]
which actually is much larger than $\frac{\underline{\eta}}{2}$, if $\TS,h$ are small enough. 
Hence by \cite[Lemma 5]{SS_FSI}, there is a $T_1>T_0$, such that 
\[
\eta_h( T_1 )\geq \frac{2\underline{\eta}}{3}-c(\TS+h)\geq \frac{\underline{\eta}}{2}
\]
for $\TS$ and $h$ small enough, where $T_1$ depends only on the initial energy of the problem but is independent of $\TS$ and $h$. 
Hence we can repeat the above argument with the same lower bound $\frac{\underline{\eta}}{2}$. It implies for $\TS,h\to 0$ that this procedure can be repeated arbitrarily many times, thus \eqref{noc} hold for any large $T$. 
\end{proof}

\section{Numerical experiments}
\label{sec:num}
In this section, we define a problem that we use to study the convergence rate of the linear semi-implicit Scheme-R \eqref{SKM_ref} on a reference domain $\Oref$. This semi-implicit scheme is then compared with the nonlinear fully implicit scheme corresponding to the weak form \eqref{wf_ref}.
Both numerical implementations are described in detail in Appendices~\ref{implement_semiimplicit} and \ref{implement_fullyimplicit}.

\subsection{Problem description}
In our experiments, the domain $\Oref$ is a rectangle of dimensions $2\times1$ with periodic boundary conditions in the $\xx-$direction, i.e. the solution on the left boundary coincides with the solution on the right boundary. On the bottom we have no-slip boundary conditions. At $t=0$, we prescribe zero initial conditions for all unknowns. Moreover, we set $\vrf=\vrs=1$, $\gamma_1=\gamma_2=0.1$, and $\gamma_3=0$ since we wish to solve a problem with a non-dissipating elastic shell. The flow is driven by the external force $g$ periodic in $x_1$ direction. The force is applied up to $t=0.2$ such that a big amplitude of the structure deformation is produced. Next, the force is turned off and the system is left to relax. The  force $g$ reads
\begin{equation*}
g=\left\{\begin{matrix}
200t\sin(2\pi x)&t\leq0.2,\\
0&t>0.2.
\end{matrix}\right.
\end{equation*}
Snapshots of the simulation are given in Figure~\ref{fig:snapshots}.

\begin{figure}[!htbp]
\begin{center}
\includegraphics[width=5.4cm]{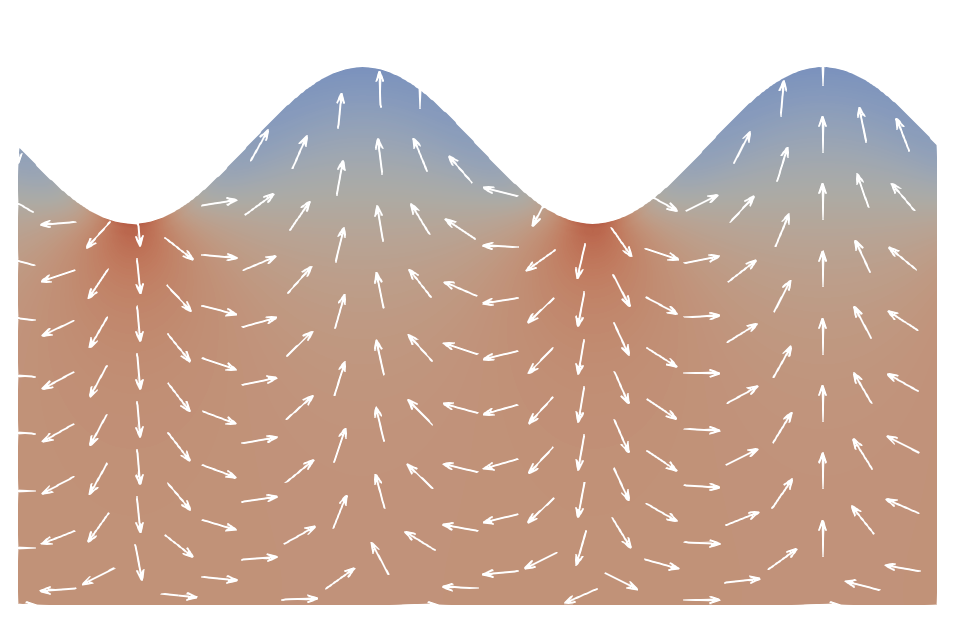}\hspace*{0.4cm}\includegraphics[width=5.4cm]{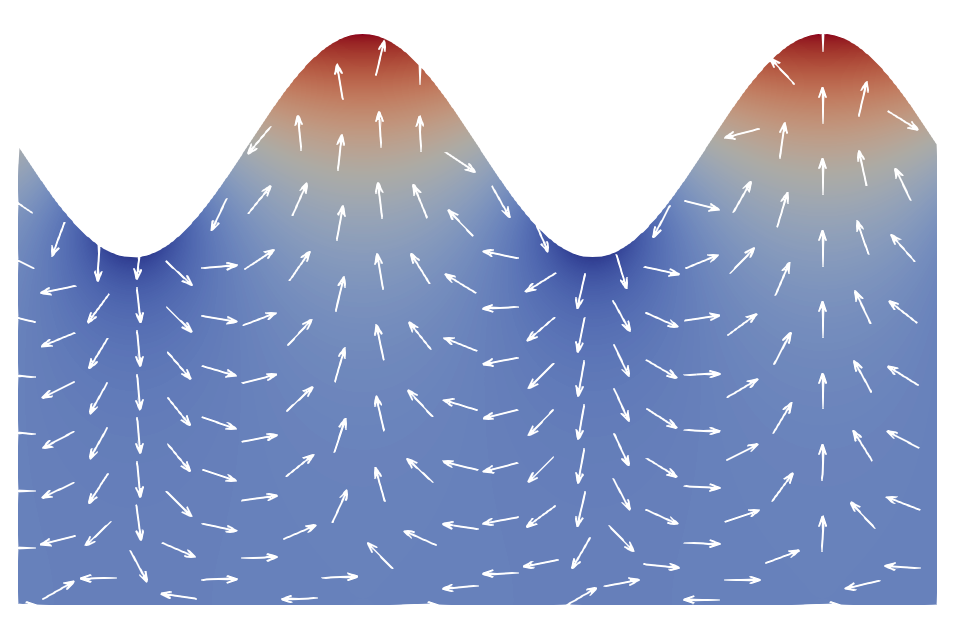}\hspace*{0.4cm}\includegraphics[width=5.4cm]{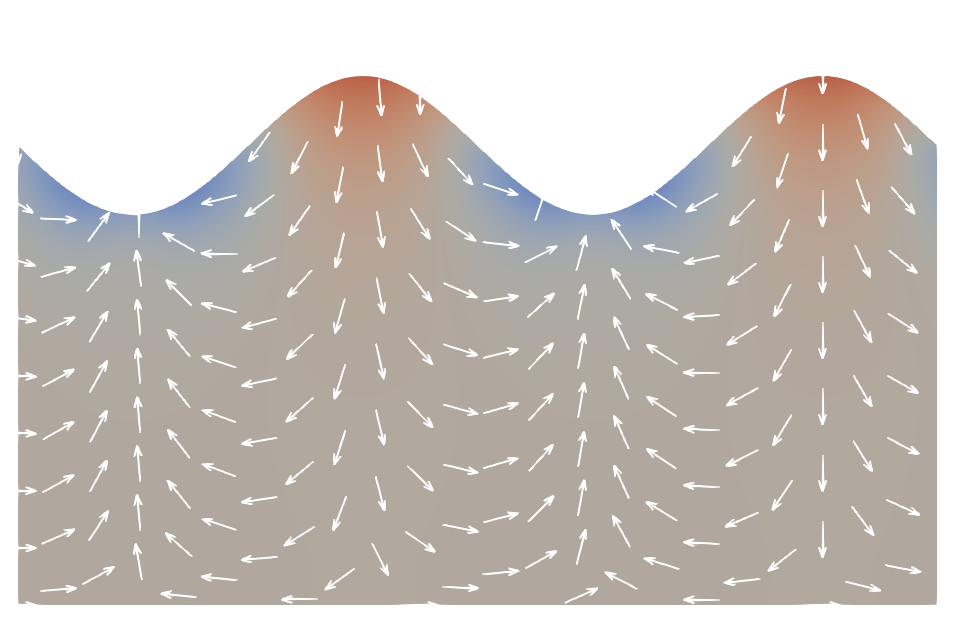}\\[3pt]
(a)\hspace*{5.4cm}(b)\hspace*{5.4cm}(c)\\[5pt]
\includegraphics[width=5.4cm]{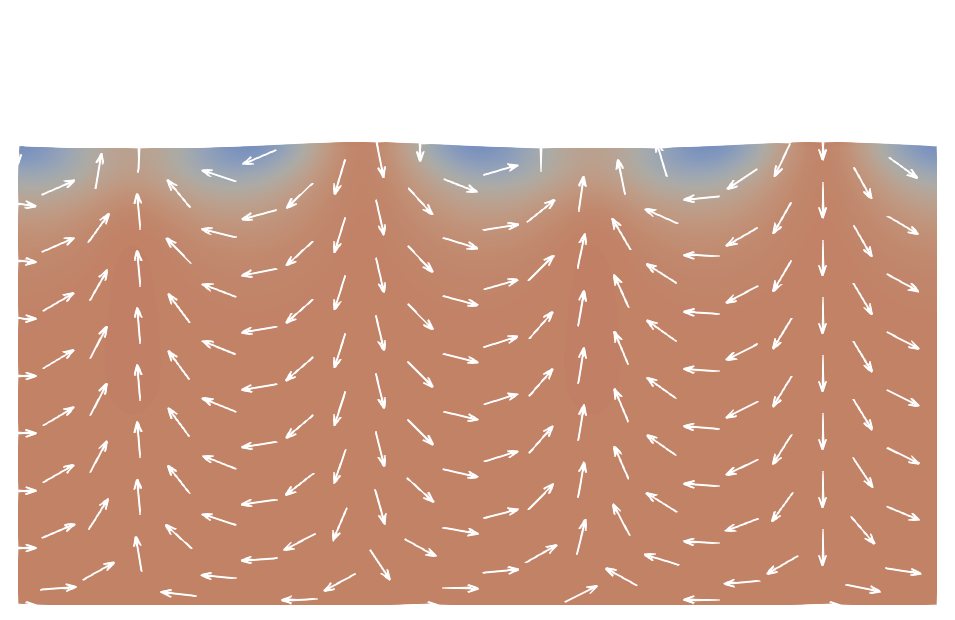}\hspace*{0.4cm}\includegraphics[width=5.4cm]{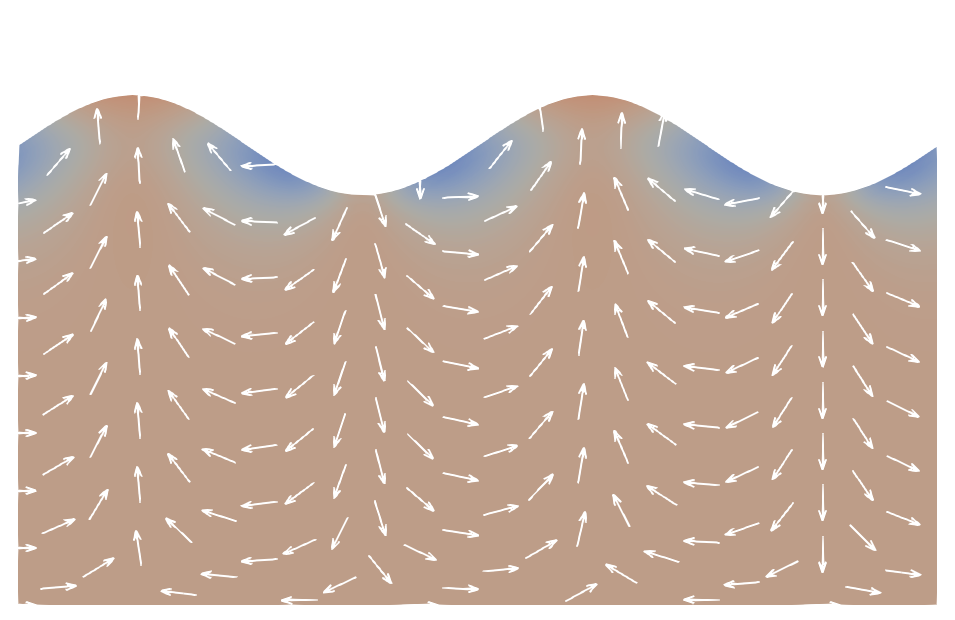}\hspace*{0.4cm}\includegraphics[width=5.4cm]{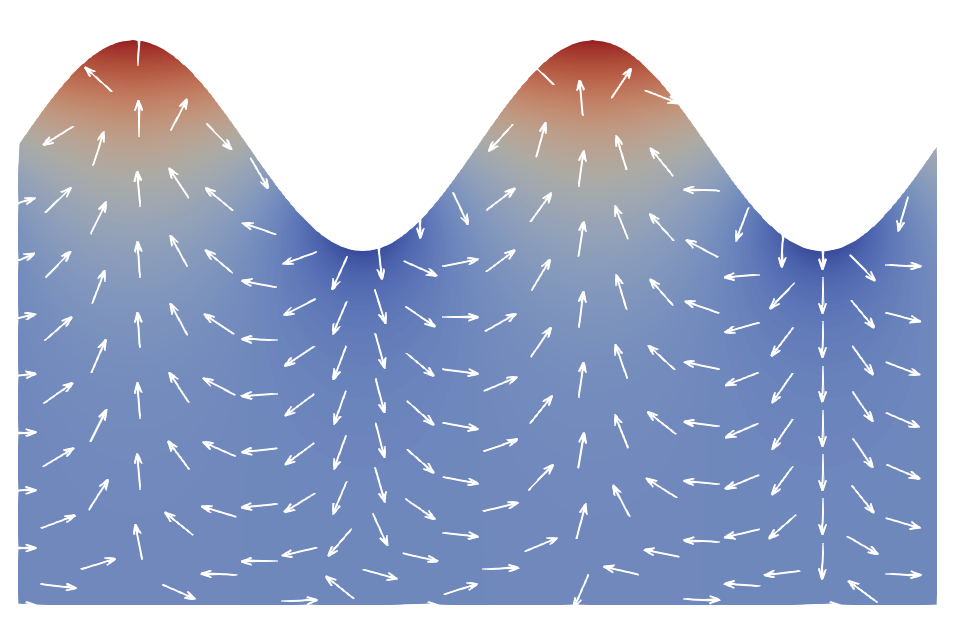}\\[3pt]
(d)\hspace*{5.4cm}(e)\hspace*{5.4cm}(f)
\caption{Snapshots of the simulation: (a) $t=0.2$ end of loading, (b) $t=0.26$ maximum of amplitude, (c) $t=0.35$, (d) $t=0.41$, (e) $t=0.45$, (f) $t=0.53$ another maximum. The color scale depicts pressure, arrows show the direction of the velocity field.}\label{fig:snapshots}
\end{center}
\end{figure}


\subsection{Convergence rates}
The simulation is computed for $t\in[0,T], T=1.0$ for six different time steps $\TS=5\times10^{-3}, 2.5\times10^{-3}, 1.25\times10^{-3}, 6.25\times10^{-4}, 3.125\times10^{-4}$ and $\TS_{\rm min}=1\times10^{-4}$ on six different meshes with the mesh sizes $h=2.83\times10^{-1}, 1.41\times10^{-1}, 7.07\times10^{-2}, 3.54\times10^{-2}, 1.77\times10^{-2}$ and $h_{\rm min}=8.84\times10^{-3}$. The solution with the finest mesh (corresponding to 410\,880 degrees of freedom (dofs) in Step 1 of the implementation of Scheme-R, see Appendix~\ref{implement_semiimplicit}) and the smallest time step is used as the reference solution.

The solutions for different mesh refinements and the smallest time step are compared to the reference solution, specifically, we record all summands of the right-hand-side of Theorem~\ref{theorem_conv_rate}, these are:
$\|e_\vu\|_{L^\infty(L^2)}$, $\|e_{\xi}\|_{L^\infty(L^2)}$, $\|e_{\eta}\|_{L^\infty(L^2)}$, $\|\nabla e_{\eta}\|_{L^\infty(L^2)}$, $\|e_{\zeta}\|_{L^\infty(L^2)}$ and $\|\nabla e_\vu\|_{L^2(L^2)}$.
The convergence with respect to the mesh size $h$ is given in the Table~\ref{tab:bilaplace_long_g0_h}. The graphs depicting the convergence rate with respect to the mesh size $h$ are shown in Figure~\ref{fig:conv_h}. The convergence with respect to the time step $\TS$ is provided in Table~\ref{tab:bilaplace_long_g0_t} and Figure~\ref{fig:conv_t}.  

\pgfkeys{
/pgf/number format/.cd,
sci,
sci zerofill,
sci generic={mantissa sep=\times,exponent={10^{#1}}}}

\pgfplotstableset{
create on use/uLiL2quotbil/.style={create col/expr={\thisrow{uLiL2}/2.45e-02}},
create on use/uLiL2quot/.style={create col/expr={\thisrow{uLiL2}/4.46e-02}},
create on use/xiLiL2quotbil/.style={create col/expr={\thisrow{xiLiL2}/3.91e-02}},
create on use/xiLiL2quot/.style={create col/expr={\thisrow{xiLiL2}/9.02e-02}},
create on use/etaLiL2quotbil/.style={create col/expr={\thisrow{etaLiL2}/1.64e-03}},
create on use/etaLiL2quot/.style={create col/expr={\thisrow{etaLiL2}/2.90e-03}},
create on use/gradetaLiL2quotbil/.style={create col/expr={\thisrow{gradetaLiL2}/2.00e-02}},
create on use/gradetaLiL2quot/.style={create col/expr={\thisrow{gradetaLiL2}/3.23e-02}},
create on use/LapetaLiL2quotbil/.style={create col/expr={\thisrow{LapetaLiL2}/3.06e-02}},
create on use/LapetaLiL2quot/.style={create col/expr={\thisrow{LapetaLiL2}/5.75e-02}},
create on use/graduL2L2quotbil/.style={create col/expr={\thisrow{graduL2L2}/1.73e-01}},
create on use/graduL2L2quot/.style={create col/expr={\thisrow{graduL2L2}/2.84e-01}},
create on use/gradxiL2L2quotbil/.style={create col/expr={\thisrow{gradxiL2L2}/7.69e-02}},
create on use/gradxiL2L2quot/.style={create col/expr={\thisrow{gradxiL2L2}/1.32e-01}},
create on use/uLiL2quotlong/.style={create col/expr={\thisrow{uLiL2}/1.20e+00}},
create on use/xiLiL2quotlong/.style={create col/expr={\thisrow{xiLiL2}/2.84e-00}},
create on use/etaLiL2quotlong/.style={create col/expr={\thisrow{etaLiL2}/2.22e-01}},
create on use/gradetaLiL2quotlong/.style={create col/expr={\thisrow{gradetaLiL2}/1.41e-00}},
create on use/LapetaLiL2quotlong/.style={create col/expr={\thisrow{LapetaLiL2}/9.22e-00}},
create on use/graduL2L2quotlong/.style={create col/expr={\thisrow{graduL2L2}/1.23e+01}},
create on use/timeuLiL2quotbil/.style={create col/expr={\thisrow{uLiL2}/3.49e-03}},
create on use/timeuLiL2quot/.style={create col/expr={\thisrow{uLiL2}/1.25e-02}},
create on use/timexiLiL2quotbil/.style={create col/expr={\thisrow{xiLiL2}/8.50e-03}},
create on use/timexiLiL2quot/.style={create col/expr={\thisrow{xiLiL2}/3.02e-02}},
create on use/timeetaLiL2quotbil/.style={create col/expr={\thisrow{etaLiL2}/2.44e-04}},
create on use/timeetaLiL2quot/.style={create col/expr={\thisrow{etaLiL2}/5.71e-04}},
create on use/timegradetaLiL2quotbil/.style={create col/expr={\thisrow{gradetaLiL2}/1.53e-03}},
create on use/timegradetaLiL2quot/.style={create col/expr={\thisrow{gradetaLiL2}/3.59e-03}},
create on use/timeLapetaLiL2quotbil/.style={create col/expr={\thisrow{LapetaLiL2}/9.63e-03}},
create on use/timeLapetaLiL2quot/.style={create col/expr={\thisrow{LapetaLiL2}/2.25e-02}},
create on use/timegraduL2L2quotbil/.style={create col/expr={\thisrow{graduL2L2}/1.23e-02}},
create on use/timegraduL2L2quot/.style={create col/expr={\thisrow{graduL2L2}/3.59e-02}},
create on use/timegradxiL2L2quotbil/.style={create col/expr={\thisrow{gradxiL2L2}/1.12e-02}},
create on use/timegradxiL2L2quot/.style={create col/expr={\thisrow{gradxiL2L2}/3.54e-02}},
create on use/timeuLiL2quotlong/.style={create col/expr={\thisrow{uLiL2}/2.55e-01}},
create on use/timexiLiL2quotlong/.style={create col/expr={\thisrow{xiLiL2}/5.50e-01}},
create on use/timeetaLiL2quotlong/.style={create col/expr={\thisrow{etaLiL2}/4.23e-02}},
create on use/timegradetaLiL2quotlong/.style={create col/expr={\thisrow{gradetaLiL2}/2.66e-01}},
create on use/timeLapetaLiL2quotlong/.style={create col/expr={\thisrow{LapetaLiL2}/1.67e-00}},
create on use/timegraduL2L2quotlong/.style={create col/expr={\thisrow{graduL2L2}/1.61e-00}},
columns/h/.style={int detect,column type=c, fixed zerofill, precision=2, column type/.add={|}{|}, column name=$h$},
columns/dt/.style={int detect,column type=c, fixed zerofill, precision=2, column type/.add={|}{|}, column name=$\TS$},
columns/uLiL2/.style={int detect,column type=c, fixed zerofill, precision=2, column type/.add={}{|}, column name=$\|e_\vu\|_{L^\infty(L^2)}$},
columns/xiLiL2/.style={int detect,column type=c, fixed zerofill, precision=2, column type/.add={}{|}, column name=$\|e_{\xi}\|_{L^\infty(L^2)}$},
columns/etaLiL2/.style={int detect,column type=c, fixed zerofill, precision=2, column type/.add={}{|}, column name=$\|e_{\eta}\|_{L^\infty(L^2)}$},
columns/gradetaLiL2/.style={column type=c, fixed zerofill, precision=2, column type/.add={}{|}, column name=$\|\nabla e_{\eta}\|_{L^\infty(L^2)}$},
columns/LapetaLiL2/.style={int detect,column type=c, fixed zerofill, precision=2, column type/.add={}{|}, column name=$\|e_{\zeta}\|_{L^\infty(L^2)}$},
columns/graduL2L2/.style={int detect,column type=c, fixed zerofill, precision=2, column type/.add={}{|}, column name=$\|\nabla e_\vu\|_{L^2(L^2)}$},
columns/gradxiL2L2/.style={int detect,column type=c, fixed zerofill, precision=2, column type/.add={}{|}, column name=$\|\nabla e_{\xi}\|_{L^2(L^2)}$},
empty cells with={--}, 
every head row/.style={before row=\hline,after row=\hline},
every last row/.style={after row=\hline},
}

\newcommand{\Convergence}[3]{
\pgfplotsextra{
\pgfmathsetmacro{\ax}{0.5}
\pgfmathsetmacro{\ay}{1}
\pgfmathsetmacro{\bx}{1.5}
\pgfmathsetmacro{\by}{(3.0^#3)}
\pgfmathsetmacro{\slope}{(3.0^#3)}
\coordinate (a) at (axis direction cs:\ax*#1,\ay*#2);
\coordinate (b) at (axis direction cs:\bx*#1,\by*#2);
\draw (a) -- (b) (a) -| (b) node [pos=0.25,anchor=north] {\small 1} node [pos=0.95,anchor=west] {\small order #3};
}
}

\begin{table}[!htbp]
  \begin{center}{
    \begin{filecontents*}{h_errors_bilaplace_long_gamma3_0.txt}
h uLiL2 xiLiL2 etaLiL2 gradetaLiL2 LapetaLiL2 graduL2L2 
2.83e-01 1.20e+00 2.84e+00 2.22e-01 1.41e+00 9.22e+00 1.23e+01
1.41e-01 3.19e-01 5.80e-01 5.99e-02 3.79e-01 2.42e+00 7.51e+00
7.07e-02 1.05e-01 1.39e-01 1.52e-02 1.34e-01 6.02e-01 4.11e+00
3.54e-02 2.78e-02 3.31e-02 3.65e-03 6.57e-02 1.44e-01 2.12e+00
1.77e-02 5.91e-03 6.64e-03 7.32e-04 2.94e-02 2.89e-02 1.04e+00
\end{filecontents*}
\pgfplotstabletypeset[font={\small}]{h_errors_bilaplace_long_gamma3_0.txt}}
\caption{Convergence of errors with mesh refinement (using fixed time step $\TS=\TS_{\rm min}$); reference solution: $h_{\rm min}=8.84\times 10^{-3}$, $\TS_{\rm min}=10^{-4}$.}
\label{tab:bilaplace_long_g0_h}
\end{center}
\end{table}
\vspace*{-3mm}
\begin{table}[!htbp]
  \begin{center}{
    \begin{filecontents*}{t_errors_bilaplace_long_gamma3_0.txt}
dt uLiL2 xiLiL2 etaLiL2 gradetaLiL2 LapetaLiL2 graduL2L2
5.00e-03 2.55e-01 5.50e-01 4.23e-02 2.66e-01 1.67e+00 1.61e+00
2.50e-03 1.36e-01 2.87e-01 2.21e-02 1.39e-01 8.74e-01 8.52e-01
1.25e-03 6.87e-02 1.43e-01 1.10e-02 6.91e-02 4.35e-01 4.28e-01
6.25e-04 3.25e-02 6.73e-02 5.17e-03 3.25e-02 2.05e-01 2.02e-01
3.12e-04 1.37e-02 2.83e-02 2.17e-03 1.36e-02 8.60e-02 8.45e-02
\end{filecontents*}
\pgfplotstabletypeset[font={\small}]{t_errors_bilaplace_long_gamma3_0.txt}}
\caption{Convergence of errors with time step refinement (using fixed mesh size $h=h_{\rm min}$); reference solution: $h_{\rm min}=8.84\times 10^{-3}$, $\TS_{\rm min}=10^{-4}$.}
\label{tab:bilaplace_long_g0_t}
\end{center}
\end{table}

\begin{figure}[!htbp]
\begin{center}
\begin{tikzpicture}[scale=0.85]
\begin{loglogaxis}[
    width=10.5cm,
    ylabel={Errors to reference solution},
    xlabel={$h$},
    legend style={at={(1.4,0.5)},anchor=east},
    ]
\addplot+[mark=o, thick, black] table[x=h,y={uLiL2quotlong}] {h_errors_bilaplace_long_gamma3_0.txt};
    \addlegendentry{$\|e_\vu\|_{L^\infty(L^2)}$}
\addplot+[mark=square, thick, blue] table[x=h,y={xiLiL2quotlong}] {h_errors_bilaplace_long_gamma3_0.txt};
    \addlegendentry{$\|e_{\xi}\|_{L^\infty(L^2)}$}
\addplot+[mark=x, thick, green] table[x=h,y={etaLiL2quotlong}] {h_errors_bilaplace_long_gamma3_0.txt};
    \addlegendentry{$\|e_{\eta}\|_{L^\infty(L^2)}$}
\addplot+[mark=asterisk, thick, orange] table[x=h,y={gradetaLiL2quotlong}] {h_errors_bilaplace_long_gamma3_0.txt};
    \addlegendentry{$\|\nabla e_{\eta}\|_{L^\infty(L^2)}$}
\addplot+[mark=+, thick, red] table[x=h,y={LapetaLiL2quotlong}] {h_errors_bilaplace_long_gamma3_0.txt};
    \addlegendentry{$\|e_{\zeta}\|_{L^\infty(L^2)}$}
\addplot+[mark=otimes, thick, cyan] table[x=h,y={graduL2L2quotlong}] {h_errors_bilaplace_long_gamma3_0.txt};
    \addlegendentry{$\|\nabla e_\vu\|_{L^2(L^2)}$}

\Convergence{0.9e-1}{1.5e-2}{2};
\Convergence{0.9e-1}{1.5e-2}{1};
\end{loglogaxis}
\end{tikzpicture}
\caption{Mesh convergence comparison for $\|e_\vu\|_{L^\infty(L^2)}$, $\|e_{\xi}\|_{L^\infty(L^2)}$, $\|e_{\eta}\|_{L^\infty(L^2)}$, $\|\nabla e_{\eta}\|_{L^\infty(L^2)}$, $\|e_{\zeta}\|_{L^\infty(L^2)}$ and $\|\nabla e_\vu\|_{L^2(L^2)}$.  For a better comparison, the plots of the errors are shifted to start from the same point.}\label{fig:conv_h}
\end{center}
\end{figure}
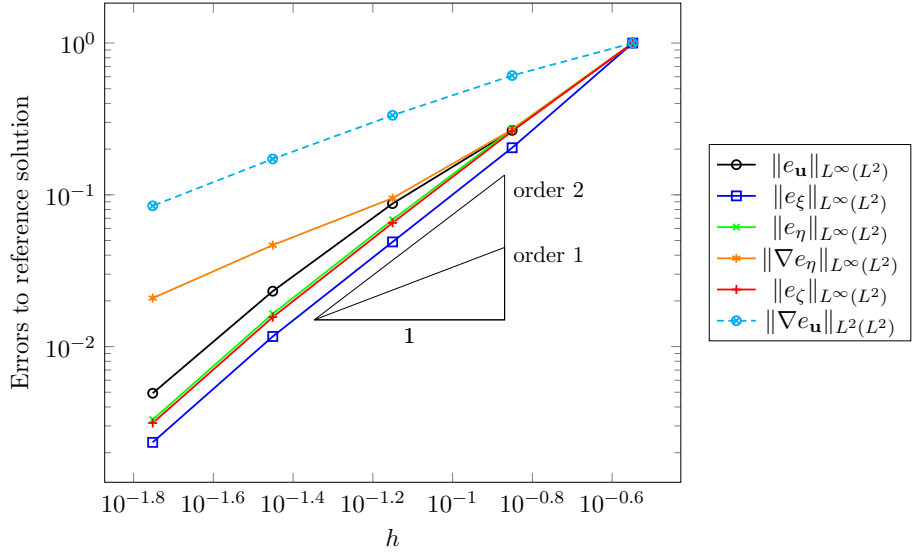

\begin{figure}[!htbp]
\begin{center}
\begin{tikzpicture}[scale=0.85]
\begin{loglogaxis}[
    width=10.5cm,
    ylabel={Errors to reference solution},
    xlabel={$\TS$},
    legend style={at={(1.4,0.5)},anchor=east},
    ]
\addplot+[mark=o, thick, black] table[x=dt,y={timeuLiL2quotlong}] {t_errors_bilaplace_long_gamma3_0.txt};
    \addlegendentry{$\|e_\vu\|_{L^\infty(L^2)}$}
\addplot+[mark=square, thick, blue] table[x=dt,y={timexiLiL2quotlong}] {t_errors_bilaplace_long_gamma3_0.txt};
    \addlegendentry{$\|e_{\xi}\|_{L^\infty(L^2)}$}
\addplot+[mark=x, thick, green] table[x=dt,y={timeetaLiL2quotlong}] {t_errors_bilaplace_long_gamma3_0.txt};
    \addlegendentry{$\|e_{\eta}\|_{L^\infty(L^2)}$}
\addplot+[mark=asterisk, thick, orange] table[x=dt,y={timegradetaLiL2quotlong}] {t_errors_bilaplace_long_gamma3_0.txt};
    \addlegendentry{$\|\nabla e_{\eta}\|_{L^\infty(L^2)}$}
\addplot+[mark=+, thick, red] table[x=dt,y={timeLapetaLiL2quotlong}] {t_errors_bilaplace_long_gamma3_0.txt};
    \addlegendentry{$\|e_{\zeta}\|_{L^\infty(L^2)}$}
\addplot+[mark=otimes, thick, cyan] table[x=dt,y={timegraduL2L2quotlong}] {t_errors_bilaplace_long_gamma3_0.txt};
    \addlegendentry{$\|\nabla e_\vu\|_{L^2(L^2)}$}
	
\Convergence{1.4e-3}{12e-2}{1};
\end{loglogaxis}
\end{tikzpicture}
\caption{Timestep convergence comparison $\|e_\vu\|_{L^\infty(L^2)}$, $\|e_{\xi}\|_{L^\infty(L^2)}$, $\|e_{\eta}\|_{L^\infty(L^2)}$, $\|\nabla e_{\eta}\|_{L^\infty(L^2)}$, $\|e_{\zeta}\|_{L^\infty(L^2)}$ and $\|\nabla e_\vu\|_{L^2(L^2)}$. For a better comparison, the errors start at the same point.}\label{fig:conv_t}
\end{center}
\end{figure}
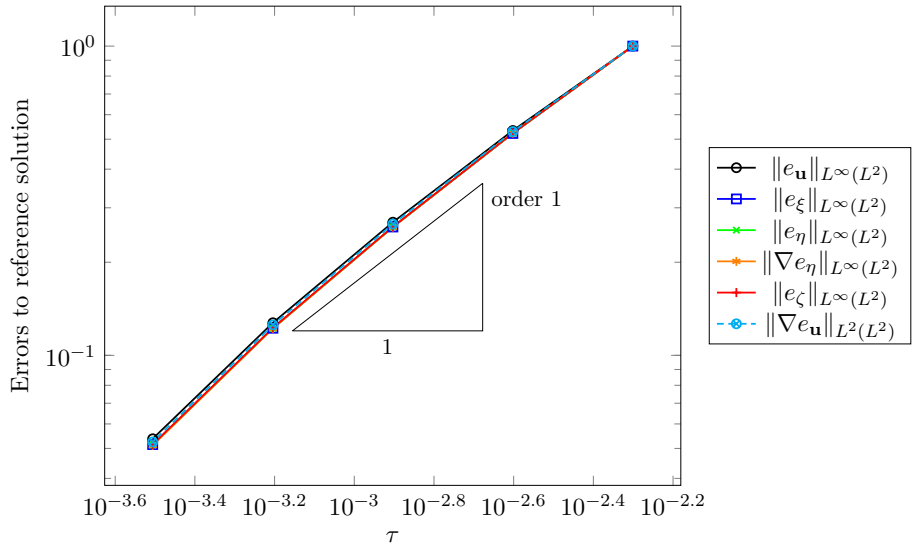

In Theorem~\ref{theorem_conv_rate} we proved that the convergence rate is linear both in $h$ (space) and $\TS$ (time) for the sum of all errors mentioned. This is justified by the experiments. In time the convergence is indeed linear for all summands (see Figure~\ref{fig:conv_t}), in space we observe a quadratic convergence for $\|e_\vu\|_{L^\infty(L^2)}$, $\|e_{\xi}\|_{L^\infty(L^2)}$, $\|e_{\eta}\|_{L^\infty(L^2)}$, $\|e_{\zeta}\|_{L^\infty(L^2)}$, but a linear convergence for $\|\nabla e_{\eta}\|_{L^\infty(L^2)}$ and 
$\|\nabla e_\vu\|_{L^2(L^2)}$ (see Figure~\ref{fig:conv_h}).

\subsection{Comparison between the semi-implicit Scheme-R and fully implicit scheme}\label{sec:comparison_schemeR-fully_implicit}
Since our proposed Scheme-R performs in accordance with the (optimal) predictions, we decided to test how well it behaves with respect to the fully implicit scheme, {as many researchers believe that a monolithic scheme should be implemented fully implicitly}. In the fully implicit scheme we solve a fully implicit nonlinear problem based on the weak form~\eqref{wf_ref}, the details of the implementation are given in Appendix~\ref{implement_fullyimplicit}. The main difference is the following. In the semi-implicit Scheme-R, as described in Appendix~\ref{implement_semiimplicit}, every time step is solved in two steps. First, we solve a linear problem for velocity $\vu$, the second order derivative of the mesh displacement $\zeta$, and the pressure $p$. This is followed by a second step in which we update the mesh displacement $\eta$. In the fully implicit scheme we solve everything at once, which, however, requires to solve a more expensive nonlinear problem. It turns out that both schemes produce the same solution.

\begin{figure}[!htbp]
\begin{center}
\begin{tikzpicture}[scale=0.85]
\begin{loglogaxis}[
    width=10.5cm,
    ylabel={Errors to reference solution},
    xlabel={$\TS$},
    legend style={at={(1.66,0.5)},anchor=east},
    ]
\addplot+[mark=o, thick, black] table[x=dt,y={explicitL2L2gradu}] {t_errors_compare_explicit_fullALE.txt};
    \addlegendentry{$\|\nabla e_\vu\|_{L^2(L^2)}$ semi-implicit}
\addplot+[mark=x, thick, blue] table[x=dt,y={fullALEL2L2gradu}] {t_errors_compare_explicit_fullALE.txt};
    \addlegendentry{$\|\nabla e_\vu\|_{L^2(L^2)}$ fully implicit}
\addplot+[mark=square, thick, green] table[x=dt,y={explicitLiL2gradeta}] {t_errors_compare_explicit_fullALE.txt};
    \addlegendentry{$\|\nabla e_{\eta}\|_{L^\infty(L^2)}$ semi-implicit}
\addplot+[mark=+, thick, orange] table[x=dt,y={fullALELiL2gradeta}] {t_errors_compare_explicit_fullALE.txt};
    \addlegendentry{$\|\nabla e_{\eta}\|_{L^\infty(L^2)}$ fully implicit}

\end{loglogaxis}
\end{tikzpicture}
\caption{Timestep convergence comparison semi-implicit vs. fully implicit for $\|\nabla e_\vu\|_{L^2(L^2)}$ and $\|\nabla e_{\eta}\|_{L^\infty(L^2)}$.}\label{comparison_fully_implicit_semiimplicit}
\end{center}
\end{figure}
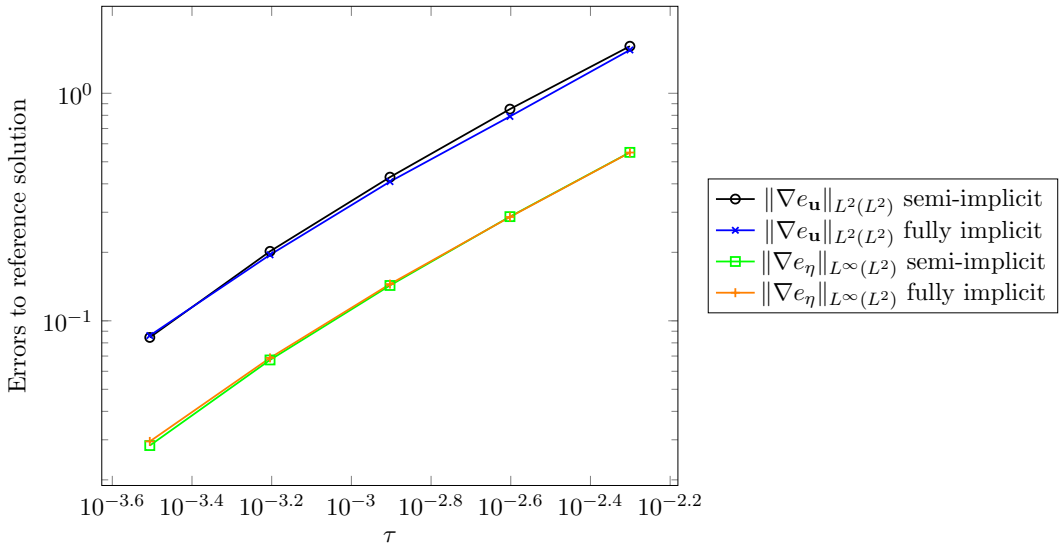

Since the main difference between the two schemes is in the time splitting, we compare the numerical errors $\|\nabla e_{\eta}\|_{L^\infty(L^2)}$ and $\|\nabla e_{\eta}\|_{L^\infty(L^2)}$ for several different time steps $\TS$. Similarly as in the previous subsection, we compute the errors 
with respect to the same reference solution obtained by the semi-implicit Scheme-R with the finest mesh $h_{\rm min}=8.84\times 10^{-3}$ and smallest time step $\TS=1\times10^{-4}$ for the numerical solutions of these two schemes. The graph of convergence in time is shown in Figure~\ref{comparison_fully_implicit_semiimplicit}. 

Note that the two discrete problems are computed on the same mesh $h=h_{\rm min}$ for both schemes. In case of the semi-implicit Scheme-R we solve a linear problem of size 410\,880 dofs in Step 1 and a linear problem of size 153\,920 dofs in Step 2 every time step. In case of the fully implicit scheme we solve a nonlinear problem of size 564\,800 dofs every time step. All problems are computed on a server equipped with Intel Xeon Gold 6240 CPU, and (although the code works in parallel) for the purpose of comparison we run them in serial. We have  recorded the CPU time of the computations for the largest and smallest time step $\TS$, see Table~\ref{tab:CPU_time}.   
For the largest time step $\TS=5\times10^{-3}$ the semi-implicit scheme is $5.53$ times faster than the fully implicit scheme that needs to solve three Newton iterations in average in every time step and solves a slightly larger problem. For the smallest time step $\TS=3.12\times10^{-4}$ the semi-implicit scheme is $3.88$ times faster because the fully implicit scheme needs in average only two Newton iterations per time step. 

\begin{table}[!htbp]
\begin{center}
\begin{tabular}{|l|ccr|}
\hline
Scheme&$\TS$&Avg Newton its&CPU time [min]\\
\hline
Fully implicit&$5.00\times10^{-3}$&3&135.5\qquad\ \\
Semi-implicit&$5.00\times10^{-3}$&--&24.5\qquad\ \\
Fully implicit&$3.12\times10^{-4}$&2&1\,310.7\qquad\ \\
Semi-implicit&$3.12\times10^{-4}$&--&338.0\qquad\ \\
\hline
\end{tabular}  
\caption{Comparison of CPU times (in minutes) for semi-implicit and fully implicit schemes.}
\label{tab:CPU_time}
\end{center}
\end{table}

\section{Conclusion}\label{sec:con}
We have introduced a novel semi-implicit and linear scheme for the approximation of the  interaction between an incompressible fluid and elastic shell allowing for {\em large deformation}. By this we mean that the domain of definition for the fluid is time changing and the changes of the domain can be arbitrarily large as long as no topological change appears.  We have proved that the scheme is energy stable and it converges to the smooth solution linearly with respect to the mesh size $h$ and time step $\TS$. We have implemented the scheme in FEniCS and observed that the convergence rates are optimal. Possibly, the rates can be improved for some terms, where we observed quadratic growth which paves the way for further research. We have compared our semi-implicit scheme with a fully implicit scheme that not only provides the same convergence rates as our scheme but does produce almost exactly the same solution. Moreover, our scheme  overperforms the fully nonlinear scheme several times in terms of consumed CPU time.

The analysis presented here, in particular the development of the interpolation operators does form the basis for new theoretical numerical investigations. It is shown that suitable interpolation operators for nonlinear equations coupled via their geometry can be constructed. Hence they motivate a respective convergence analysis for the plethora of applications involving such couplings. From their construction, it might also be possible to read where troubles of the convergence of schemes may be expected. For instance in case, when a topological change of geometry is approaching.


\section*{\centering Statements and Declarations}
{\bf Competing interests:}  On behalf of all authors, the corresponding author states that there is no conflict of interest.\\

\bibliographystyle{siamplain}

\appendix


\section{Appendix: Interpolation operators}\
In this part, we present some useful estimate/equality for the interpolation operators used in our paper. 
First, we show the approximation error of the discrete Laplace $\Laph$ given by \eqref{laph} of a Riesz projection operator $\Riesz$ defined by \eqref{Riesz1}. 
\begin{Lemma}
\label{lem:w22}
Let $\eta\in W^{3,2}\cap W^{1,2}_0(\Sigma)$,  $\Sigma\subset \R^n$, $n=2,3$, and  $\Vsh\subset W^{1,2}_0(\Sigma)$ is a closed subspace. 
For any $\psi\in \Vsh$, let $\Delta_h\phi \in \Vsh$ satisfy 
\[
- \int_\Sigma \Delta_h \phi \psi dz =\int_\Sigma \nabla  \phi \cdot \nabla \psi dz
\]
and $\Riesz$ satisfy \eqref{Riesz1} in $n$ dimensions, i.e.,
\[
\int_\Sigma (\Grad \eta-\Grad \Riesz\eta) \cdot \Grad \psi\, dz=0
\]
Moreover, we assume there exists a projection $P_h: L^2(\Sigma)\to \Vsh$ satisfying  
\[
\norm{\eta-P_h\eta}_2\leq ch\norm{\nabla \eta}_2 \; \forall\;  \eta\in W^{1,2}(\Sigma).
\]
Then 
\[
\norm{\Delta\eta-\Delta_h\Riesz\eta}_2\leq ch\norm{\nabla \Delta\eta}_2. 
\]
\end{Lemma}
\begin{proof}
By the definition of the Riesz projection, and the discrete Laplace, we find 
\begin{align*}
& \norm{\Delta\eta-\Delta_h\Riesz\eta}_2^2 =\int_\Sigma (\Delta\eta-\Delta_h\Riesz\eta)\; (\Delta\eta-\Delta_h\Riesz\eta)\, dz
\\&=\int_\Sigma (\Delta\eta-\Delta_h\Riesz\eta)\; (P_h\Delta\eta-\Delta_h\Riesz\eta)\, dz+\int_\Sigma (\Delta\eta-\Delta_h\Riesz\eta)\; (\Delta\eta-P_h\Delta\eta)\, dz
\\
&=-\int_\Sigma (\nabla\eta-\nabla\Riesz\eta)\cdot(\nabla(P_h\Delta\eta-\Delta_h\Riesz\eta))\, dz+\int_\Sigma (\Delta\eta-P_h\Delta\eta)\; (\Delta\eta-\Delta_h\Riesz\eta)\, dz
\\
&\leq \norm{\Delta\eta-P_h\Delta\eta}_2\norm{\Delta\eta-\Delta_h\Riesz\eta}.
\end{align*}
This implies the wanted estimate by the assumed property of $P_h$:
\[
\norm{\Delta\eta-P_h\Delta\eta}_2\leq ch\norm{\nabla \Delta\eta}_2 .
\]
\end{proof}
\begin{Remark}
In our setting one possibility is to choose as $P_h$ the $L^2$-Projection into $\Vsh$ defined by  $\int (\eta-P_h\eta)\; \phi_h\, dz=0$ for all $\phi\in \Vsh$, which is known to satisfy in our setting the needed estimate $\norm{\eta-P_h\eta}_2\leq ch\norm{\nabla \eta}_2$.  Note that as $P_h\eta \in \Vsh$ by definition, it is in particular a Lipschitz function and possesses a weak gradient.
\end{Remark}
\begin{Lemma}
Let  $\hvu_h$ and $\hvu$ be respectively the solution to \eqref{SKM_ref1} and \eqref{wf1_ref} with the test function $\hq \in \hQfh$. Then it holds
\begin{equation}\label{ddiv2}
0=\intOref{\hq \Grad \delta_\vu :\Mh} .
\end{equation}
\end{Lemma}
\begin{proof}
Using \eqref{SKM_ref1} and the property of the fine constructed projection \eqref{P3}, we derive
\begin{align*}
\intOref{\hq \Grad \delta_\vu :\Mh} = \intOref{\hq \Grad \hvu_h  :\Mh} -\intOref{\hq \Grad \PiF \hvu :\Mh} =0.
\end{align*}
\end{proof}
\section{Appendix: Useful equalities and estimates}
\subsection{Proof of the error equation \eqref{EEQ}}\label{app_ee1}
In this part, we show the details how to obtain the equation \eqref{EEQ} satisfied by the errors. 
First, for any $k=1,\dots,N_T$ we subtract the weak formulation \eqref{wf2_ref} from the numerical scheme \eqref{SKM_ref2} and get 
\begin{equation}\label{TIK}
  \sum_{i=1}^7 T_i^k =0,  
\end{equation} 
where $T_i^k$ reads (keeping in mind that $R_i^k$, $i=1,\dots, 7$, are given in \eqref{RS})
\begin{align*}
T_1^k  = & \vrf\intOref{ (\eta_h^{k} \PDt \hvu_h^k   -  \eta^k \pdt \hvu^k  ) \cdot \hbfphi  } 
\\ = &
\vrf \intOref{ \big( \eta_h^k \PDt (\hvu_h^k - \hvu^k) + \eta_h^k (\PDt \hvu^k -\pdt\hvu^k ) + (\eta_h^k -\eta^k) \pdt \hvu^k \big) \cdot \hbfphi } 
\\= &
 \vrf \intOref{  \eta_h^k \PDt \eu ^k  \cdot \hbfphi }  + R^k_1,
\\
T_2^k  =&  \frac12 \vrf \intOref{ ( \PDt \eta_h^k \hvu_h^{k*}   - \pdt \eta^k  \hvu^k ) \cdot \hbfphi  } 
	\\=&
\frac12\vrf \intOref{ \Big( \PDt \eta_h^k (\hvu_h^{k*} - \hvu^{k*}) + ( \PDt \eta_h^k -   \pdt \eta^k )\hvu^{k*}   + \pdt \eta^k (\hvu^{k*} -\hvu^k ) \Big) \cdot \hbfphi}
\\ =&
\frac12\vrf \intOref{ \PDt \eta_h^k \eu ^{k*}  \cdot \hbfphi } 
+R_2^k,
\\
 T_3^k = &
\frac12 \vrf \intOref{  \big( \hbfphi \cdot (\Grad\hvu_h^k)      -  \hvu_h^k \cdot (\Grad\hbfphi)   \big)  \cdot  (\Jacob_h^k)^{-1}  \cdot \hvv_h^{k-1} \eta_h^k } 
\\&  - \frac12 \vrf \intOref{  \big( \hbfphi \cdot (\Grad\hvu^k)      -  \hvu^k \cdot (\Grad\hbfphi)   \big)  \cdot  (\Jacob^k)^{-1}  \cdot \hvv^{k} \eta^k } 
\\= &
\frac12 \vrf \intOref{ \Big( \hbfphi \cdot   (\Grad \euk ) 
-\euk   \cdot   (\Grad \hbfphi) \Big)
 \cdot (\Jacob_h^k)^{-1} \hvv_h^{k-1} \eta_h^k }
\\&  
+ \frac12 \vrf \intOref{ \Big( \hbfphi \cdot (\Grad \hvu^k) - \hvu^k \cdot (\Grad\hbfphi) \Big) \cdot \left(  (\Jacob_h^k)^{-1} \hvv_h^{k-1} \eta_h^k  -  (\Jacob^k)^{-1} \hvv^{k} \eta^k  \right)  } 
 \\= & R^k_3,
\\
 T_4^k = & 
 \intOrefB{ \hp_h^k \Grad \hbfphi : \Mhk - \hp^k \Grad \hbfphi : \Mk}
\\ = &  
\intOref{  e_p^k \Grad \hbfphi : \Mhk } 
+ \intOref{ \hp^k \Grad \hbfphi : \big(\Mhk -  \Mk\big)} 
=R_4^k,
\\
T_5^k = &
2 \mu \intOrefB{  \big( \Grad \hvu_h^k (\Jacob_h^k)^{-1}\big)^\rmS  :  \big(\Grad \hbfphi  (\Jacob_h^k)^{-1}\big)\eta_h^k   
- \big( \Grad \hvu^k (\Jacob^k)^{-1}\big)^\rmS  :  \big(\Grad \hbfphi (\Jacob^k)^{-1}\big) \eta^k    }
\\ = & 2 \mu \intOref{ \big( \Grad \euk  (\Jacob_h^k)^{-1} \big)^\rmS: \big(\Grad \hbfphi (\Jacob_h^k)^{-1} \big)\eta_h^k  } 
\\ & + 2 \mu \intOrefB{  \big( \Grad \hvu^k (\Jacob_h^k)^{-1} \big)^\rmS: (\Grad \hbfphi (\Jacob_h^k)^{-1}) \eta_h^k 
-\big( \Grad \hvu^k (\Jacob^k)^{-1} \big)^\rmS: (\Grad \hbfphi (\Jacob^k)^{-1})\eta^k}
\\ = & 2 \mu \intOref{ \big( \Grad \euk  (\Jacob_h^k)^{-1} \big)^\rmS: \big(\Grad \hbfphi (\Jacob_h^k)^{-1} \big)\eta_h^k  }  + R_5^k,
\\
  T_6^k =&  \vrs  \intS{( \PDt \xi_h^k  - \pdt \xi^k  )\psi} 
 =  \vrs\intS{\PDt e_\xi^k \psi} +R_6^k,
\\
  T_7^k =& a_s(\eta_h^{k+1},  \zeta_h^{k}, \xi_h^k  ,\psi)  -
  a_s(\eta^k,  \zeta^{k} , \xi^k  ,\psi)  
 \\ = &
a_s(e_\eta^{k+1},e_\zeta^{k+1},e_\xi^k,\psi)
+ \gamma_1 \intS{\pdx (\eta^{k+1} -\eta^k) \pdx \psi}
+ \gamma_2 \intS{\pdx (\zeta^{k+1} -\zeta^k) \pdx \psi} 
 \\= &
a_s(e_\eta^{k+1},e_\zeta^{k+1},e_\xi^k,\psi)
+R_7^k.
 \end{align*}
Consequently, substituting the above expansions of the $T_i$-terms into \eqref{TIK} and shifting the $R_i$-terms to the right-hand-side, we derive \eqref{EEQ}.

\subsection{Preliminary estimates}
In this part we show some preliminary estimates and equalities. 
First, we show the estimates related to the time discretization operator $\PDt$ given by \eqref{time_D}. 
\begin{Lemma}
\label{lm_edt}
Let $\phi \in L^2((0,T)\times D)$ for $D \in \{ \Sigma, \Oref\}$. Then we have
\begin{subequations}
\begin{equation}\label{edts0}
 \TS \sumN \norm{ \PDt \phi^k }_{L^2(D)}^2 \aleq   \norm{\pdt \phi}_{L^2((0,T)\times D)}^2,
\end{equation}
\begin{equation}\label{edts1}
 \TS \sumN \norm{ \PDt \phi ^{k} -  \pdt \phi^k}_{L^2(D)}^2 \aleq \TS^2  \norm{\pdtt \phi}_{L^2((0,T)\times D)}^2,
\end{equation}
\begin{equation}\label{edts2}
 \TS \sumN \norm{ \PDt \phi ^{k+1} -  \pdt \phi^k}_{L^2(D)}^2 \aleq \TS^2  \norm{\pdtt \phi}_{L^2((0,T)\times D)}^2.
\end{equation}
\end{subequations}
\end{Lemma}
\begin{proof}
We only need to prove \eqref{edts2}, the others follow analogously.\footnote{For \eqref{edts0} and \eqref{edts1} and $p=2=q$ a proof can be found in Lemma 5.4 and  Lemma 5.3 of \cite{Boris2}, respectively.} To begin, we recall the product rule $(uv)' = u' v + u v'$ and apply it for $u = (t^{k+1} -t )/\TS$ and $v =\pdt \phi(t)$, which yields  
\[
\int_{t^k} ^{t^{k+1}}  \frac{t^{k+1} -t }{\TS} \pdtt \phi(t) \dt 
= \left[ \frac{t^{k+1} -t }{\TS} \pdt \phi(t)\right]_{t^k}^{t^{k+1}}  + \frac{1}{\TS}\int_{t^k} ^{t^{k+1}}   \pdt \phi(t) \dt 
=  \PDt \phi^{k+1} - \pdt \phi(t^k). 
\]
Thanks to the above equality and H\"older's inequality, we obtain
\begin{equation*}
\begin{aligned}
& \TS \sumN \norm{ \PDt \phi ^{k+1} -  \pdt \phi^k}_{L^2(D)}^2 
=\TS \sumN  \intD{ \Abs{\int_{t^k} ^{t^{k+1}}  \frac{t^{k+1} -t }{\TS} \pdtt \phi(t) \dt }^2 }
\\& \leq  
\TS \sumN  \intD{
\left( \int_{t^k} ^{t^{k+1}}  \Abs{ \frac{t^{k+1} -t }{\TS} }^2 \dt  \right)
\left( \int_{t^k} ^{t^{k+1}} \Abs{  \pdtt \phi(t)  }^2 \dt \right)} 
\\& \leq  
\frac13 \TS^2   \intD{\sumN 
\int_{t^k} ^{t^{k+1}} \Abs{  \pdtt \phi(t)  }^2 \dt }  
=\frac13 \TS^2  \norm{\pdtt \phi(t)}_{L^2((0,T)\times D)}^2, 
\end{aligned}
\end{equation*}
which proves \eqref{edts2}.  
\end{proof}

\begin{Lemma}
Let $\eta \in W^{2,2}(\Sigma)$, $\xi=\pdt \eta$, $\zeta = - \Lapx \eta$, $\eta_h \in \Vsh$, $\xi_h^k=\PDt \eta_h^{k+1}$, $k=1,\ldots,N_T$,  $\zeta_h = -\Laph \eta_h$, and $\psi\in \Vsh$. Let $a_s$ be given by \eqref{abbs} and 
the notation of the errors be given by \eqref{ers}. Then
\begin{equation}\label{IM1}
 \delta_\xi^k = \PDt \delta_\eta^{k+1} + \Riesz (\PDt \eta^{k+1} -\pdt \eta^k),
\end{equation}
\begin{equation}\label{IM12}
     \delta_\zeta 
     = - \Laph \delta_\eta ,
\quad
 \intS{  \psi  \delta_\zeta  } 
= - \intS{ \Laph \psi  \;  \delta_\eta   } .
\end{equation}

\begin{equation}\label{IM2}
 \intS{\pdx I_\eta \pdx \psi}=0,\quad   \intS{\pdx I_\xi \pdx \psi}=0, 
\end{equation}
\begin{equation}\label{IM23}
\intS{I_\zeta \psi} =0, \quad 
\intS{\pdx I_\zeta \pdx\psi}=0,
\end{equation}
\begin{equation}\label{IM3}
   \intS{\pdx \delta_\zeta \pdx \psi}
  = \intS{\Laph \delta_\eta \Laph \psi},
\end{equation}
\begin{equation}\label{IM4}
\begin{aligned}
	\intS{\pdx \delta_\eta^{k+1} \pdx \delta_\xi^k}
	=& \intSB{\PDt \frac{|\pdx \delta_\eta^{k+1}|^2}{2}  + \frac{\TS}{2}|\PDt \pdx  	\delta_\eta^{k+1}|^2}
\\&+\intS{\pdx \delta_\eta^{k+1} \pdx (\PDt \eta^{k+1} - \pdt\eta^k )},
\end{aligned}
\end{equation}
\begin{equation}\label{IM5}
\begin{aligned}
	\intS{\pdx \delta_\zeta^{k+1} \pdx \delta_\xi^k}
	=& \intSB{\PDt \frac{| \delta_\zeta^{k+1}|^2}{2}  + \frac{\TS}{2}|\PDt   	\delta_\zeta^{k+1}|^2}
\\&- \intS{ \delta_\zeta^{k+1} \Lapx (\PDt \eta^{k+1} - \pdt\eta^k )},
\end{aligned}
\end{equation}
\begin{equation}\label{IM6}
\begin{aligned}
&a_s(e_\eta^{k+1},e_\zeta^{k+1},e_\xi^k,\delta_\xi^k )
=
 \PDt \intSB{ \frac{\gamma_1}{2} \abs{\pdx \delta_\eta^{k+1}}^2 + \frac{\gamma_2}{2} \abs{\delta_\zeta^{k+1}}^2 } + \gamma_3 \intS{\abs{\delta_\xi^k}^2}
\\&\quad + \frac{\TS}2 \intSB{\gamma_1 |\PDt \pdx \delta_\eta^{k+1}|^2 +\gamma_2 |\PDt \delta_\zeta^{k+1}|^2  }
\\& \quad 
+\intSB{\gamma_1 \pdx \delta_\eta^{k+1} \pdx (\PDt \eta^{k+1} - \pdt\eta^k )
-\gamma_2 \delta_\zeta^{k+1} \Lapx (\PDt \eta^{k+1} - \pdt\eta^k )}.
\end{aligned}
\end{equation}
\end{Lemma}
\begin{proof}
Recalling equalities 
\[\xi^k = \pdt\eta^k, \xi_h^k = \PDt\eta_h^{k+1}
\]
and the errors defined in \eqref{ers} we get \eqref{IM1}
\begin{align*}
& \delta_\xi^k - \PDt \delta_\eta^{k+1}  = (\xi_h^k - \Riesz \xi^k ) -  (\PDt\eta_h^{k+1} - \PDt \Riesz \eta^{k+1} )
\\&= \PDt \Riesz \eta^{k+1} - \Riesz \xi^k
=\Riesz(\PDt \eta^{k+1} - \pdt\eta^k ).
\end{align*}
Moreover, it is easy to check 
\begin{align*}
& \delta_\zeta = \zeta_h + \Laph \Riesz \eta =  \Laph \Riesz \eta  - \Laph \eta_h   = - \Laph \delta_\eta .
\end{align*}
Further, recalling the discrete Laplace operator, we complete the proof of \eqref{IM12}, i.e.,
\begin{align*}
& \delta_\zeta = \zeta_h + \Laph \Riesz \eta =  \Laph \Riesz \eta  - \Laph \eta_h   = - \Laph \delta_\eta ,
\\&
 \intS{  \psi  \delta_\zeta  } =   - \intS{  \psi    \Laph  \delta_\eta  }
=   \intS{  \pdx \psi    \pdx  \delta_\eta  }
= - \intS{ \Laph \psi  \;  \delta_\eta   } .
\end{align*}
Next, recalling the Riesz projection \eqref{Riesz1} we immediate get \eqref{IM2}
\begin{align*}
&  \intS{\pdx I_\eta \pdx \psi}=   \intS{\pdx (\Riesz \eta - \eta) \pdx \psi} =0, 
 \\&
    \intS{\pdx I_\xi \pdx \psi}=   \intS{\pdx (\Riesz \xi - \xi) \pdx \psi} =0.
\end{align*}
Analogously, we find 
\begin{align*}
\intS{I_\zeta \psi}= \intS{(\Lapx \eta - \Laph \Riesz \eta)\psi}
= \intS{(\pdx \Riesz \eta - \pdx \eta ) \pdx\psi}=0.
\end{align*}
Then, setting $\Laph \psi$ in the above equality, we get \eqref{IM23}
\begin{align*}
\intS{\pdx I_\zeta \pdx\psi}
=-\intS{ I_\zeta \Laph \psi}
=0.
\end{align*}
Further, recalling \eqref{laph} and \eqref{Riesz2} we get \eqref{IM3}
\begin{align*}
&  \intS{\pdx \delta_\zeta \pdx \psi}
=-\intS{\delta_\zeta \Laph \psi}
=-\intS{(\zeta_h + \Laph \Riesz \eta) \Laph \psi}
\\& = \intS{(\Laph\eta_h - \Laph \Riesz \eta) \Laph \psi}
= \intS{\Laph \delta_\eta \Laph \psi}
\end{align*}
Using the above equalities \eqref{IM1} and \eqref{IM3}, the algebraic inequality \eqref{algab}, and the Riesz projection \eqref{Riesz1} we get \eqref{IM4}
\begin{align*}
&  \intS{\pdx \delta_\eta^{k+1} \pdx \delta_\xi^k}
= \intS{ \pdx\delta_\eta^{k+1} \Big( \pdx \PDt \delta_\eta^{k+1}
+ \pdx \Riesz(\PDt \eta^{k+1} - \pdt\eta^k ) \Big)}
\\&= \intSB{\PDt \frac{|\pdx \delta_\eta^{k+1}|^2}{2}  + \frac{\TS}{2}|\PDt  \pdx \delta_\eta^{k+1}|^2}
+\intS{\pdx \delta_\eta^{k+1} \pdx (\PDt \eta^{k+1} - \pdt\eta^k )}
\end{align*}
Analogously, using the above equalities \eqref{IM1}, \eqref{IM12}, and \eqref{IM3}, the algebraic inequality \eqref{algab}, and the Riesz projection \eqref{Riesz2}, we get \eqref{IM5}
\begin{align*}
&  \intS{\pdx \delta_\zeta^{k+1} \pdx \delta_\xi^k}
= - \intS{ \delta_\zeta^{k+1} \Laph \delta_\xi^k}
= \intS{\Laph \delta_\eta^{k+1} \Laph \delta_\xi^k}
\\& 
= \intS{\Laph \delta_\eta^{k+1} \Laph \PDt \delta_\eta^{k+1}}
+\intS{\Laph \delta_\eta^{k+1} \Laph \Riesz(\PDt \eta^{k+1} - \pdt\eta^k )}
\\&= \intSB{\PDt \frac{|\Laph \delta_\eta^{k+1}|^2}{2}  + \frac{\TS}{2}|\PDt \Laph  \delta_\eta^{k+1}|^2}
+\intS{\Laph \delta_\eta^{k+1} \Lapx (\PDt \eta^{k+1} - \pdt\eta^k )}
\\& = \intSB{\PDt \frac{|\delta_\zeta^{k+1}|^2}{2}  + \frac{\TS}{2}|\PDt  \delta_\zeta^{k+1}|^2}
- \intS{\delta_\zeta^{k+1} \Lapx (\PDt \eta^{k+1} - \pdt\eta^k )}.
\end{align*}
Using the equalities \eqref{IM2} and \eqref{IM23} we know that 
\begin{align*}
a_s(I_\eta^{k+1},I_\zeta^{k+1},I_\xi^k,\psi)=0.
\end{align*}
Consequently, collecting the above equality together with \eqref{IM4} and \eqref{IM5}, we get
\begin{align*}
&a_s(e_\eta^{k+1},e_\zeta^{k+1},e_\xi^k,\delta_\xi^k )
=a_s(\delta_\eta^{k+1},\delta_\zeta^{k+1},\delta_\xi^k,\delta_\xi^k )
\\&=
	\gamma_1 \intS{ \pdx \delta_\eta^{k+1} \pdx \delta_\xi^k}
	+\gamma_2 \intS{\pdx \delta_\zeta^{k+1} \pdx \delta_\xi^k}
	+\gamma_3 \intS{\pdx \delta_\xi^k \pdx \delta_\xi^k}
\\&=
\gamma_1\intSB{\PDt \frac{|\pdx \delta_\eta^{k+1}|^2}{2}  + \frac{\TS}{2}|\PDt  \pdx \delta_\eta^{k+1}|^2}
+\gamma_1\intS{\pdx \delta_\eta^{k+1} \pdx (\PDt \eta^{k+1} - \pdt\eta^k )}
\\& \quad +
\gamma_2\intSB{\PDt \frac{|\delta_\zeta^{k+1}|^2}{2}  + \frac{\TS}{2}|\PDt  \delta_\zeta^{k+1}|^2}
-\gamma_2 \intS{\delta_\zeta^{k+1} \Lapx (\PDt \eta^{k+1} - \pdt\eta^k )}
+ \gamma_3 \intS{\abs{\delta_\xi^k}^2}
\\& =
 \PDt \intSB{ \frac{\gamma_1}{2} \abs{\pdx \delta_\eta^{k+1}}^2 + \frac{\gamma_2}{2} \abs{\delta_\zeta^{k+1}}^2 } + \gamma_3 \intS{\abs{\delta_\xi^k}^2}
\\&\quad + \frac{\TS}2 \intSB{\gamma_1 |\PDt \pdx \delta_\eta^{k+1}|^2 +\gamma_2 |\PDt \delta_\zeta^{k+1}|^2  }
\\& \quad 
+\intSB{\gamma_1 \pdx \delta_\eta^{k+1} \pdx (\PDt \eta^{k+1} - \pdt\eta^k )
- \gamma_2 \delta_\zeta^{k+1} \Lapx (\PDt \eta^{k+1} - \pdt\eta^k )},
\end{align*}
which proves \eqref{IM6} and completes the proof. 
\end{proof}

\subsection{Secondary estimates}
\begin{Lemma}Let $(p,\vu,\xi,\eta)$ be a target smooth solution of the FSI problem \eqref{pde_f}--\eqref{pde_bdc} belonging to the class \eqref{STClass}. 
Let $(p_h, \hvu_h, \xi_h,\eta_h)$ be a solution to the numerical scheme \eqref{SKM}. Then, 
\begin{equation}\label{est1}
\norm{\Jacob}_{L^\infty L^\infty} 
+ \norm{\Jacob^{-1}}_{L^\infty L^\infty} 
+\norm{\M}_{L^\infty L^\infty} 
+ \norm{\M^{-1}}_{L^\infty L^\infty} 
\aleq 1 ,
\end{equation}
\begin{equation}\label{est2}
\norm{e_\eta}_{L^\gamma L^\infty } 
\leq  \norm{\pdx\delta_\eta}_{L^\gamma L^2 } +h\norm{\pdx^2 \eta}_{L^\gamma L^2 } , 1 \leq \gamma \leq \infty,
\end{equation}
\begin{equation}\label{est3}
\norm{\pdx e_\eta}_{L^2 L^\infty } 
\aleq \norm{ \delta_\zeta}_{L^2 L^2 } 
+ h \norm{\pdx^3 \eta}_{L^2 L^2 } 
\end{equation}
and
\begin{equation}\label{est41}
\norm{\Mh - \M}_{L^2 L^\infty } 
 \aleq \norm{\pdx \delta_\eta}_{L^2 L^2} +  \norm{ \delta_\zeta}_{L^2 L^2} + h\left(\norm{\pdx^2 \eta}_{L^2 L^2} + \norm{\pdx^3 \eta}_{L^2 L^2} \right) 
 .
\end{equation}
Further
\begin{equation}\label{est42}
\norm{\Jacob_h - \Jacob}_{L^2L^\infty}   
 \aleq \norm{\pdx \delta_\eta}_{L^2 L^2} +  \norm{ \delta_\zeta}_{L^2 L^2} + h\left(\norm{\pdx^2 \eta}_{L^2 L^2} + \norm{\pdx^3 \eta}_{L^2 L^2} \right) 
\end{equation}
and
\begin{equation}\label{est43}
\begin{aligned}
 \norm{\Jacob_h^{-1} - \Jacob^{-1}}_{L^2L^\infty}   
 \aleq &
  \frac{1 + \norm{\pdx \eta}_{L^\infty L^\infty}  }{\underline{\eta}^2}  \norm{\pdx \delta_\eta}_{L^2 L^2 } 
 + \frac{1 }{\underline{\eta}} \norm{\delta_\zeta}_{L^2 L^2 }  
\\&  + h \left( \frac{1 + \norm{\pdx \eta}_{L^\infty L^\infty}  }{\underline{\eta}^2}   \norm{\pdx^2 \eta}_{L^2 L^2 } 
+ \frac{1 }{\underline{\eta}}  \norm{\pdx^3 \eta}_{L^2 L^2 }  \right).
\end{aligned}
\end{equation}
For $p\in [1,\infty)$ we find
\begin{align}
\label{eq:trace}
\norm{{\xi}^k_h}_{L^p}\leq \norm{\nabla\vu_h^k}_{L^2},
\end{align}
and
\begin{align}
\label{eq:new}
 \norm{D_t \delta_\eta}_{L^2 L^p} \aleq \norm{\delta_\xi}_{L^2 L^p} +  \TS \norm{\partial_t^2\partial_x\eta}_{L^2L^2} .
\end{align}
For $p\in [1,\infty)$ we find
\begin{equation}\label{est5}
\norm{\hvu_h -\hvu}_{L^2L^p}
\aleq \norm{\delta_\vu }_{L^2L^p} + h \norm{ \hvu }_{L^2W^{2,2}},
\end{equation}
\begin{equation}\label{est6}
\norm{\hvw_h -\hvw}_{L^2L^p}
 \aleq \norm{\delta_\xi }_{L^2L^p} 
+\TS \norm{\pdt^2\partial_x \eta }_{L^2L^2}  + h \norm{\eta }_{L^2W^{2,2}} ,
\end{equation}
and finally
\begin{equation}\label{est7}
\begin{aligned}
& \vrf \TS \summ \bigg(\intOref{\abs{\hvv_h^{k-1} \eta_h^k  -  \Jacob_h^k  (\Jacob^k) ^{-1} \hvv^k  \eta^k}^2 }\bigg)^\frac{2}{p}
 \aleq 
 \TS \summ \alpha \intOref{ \abs{\nabla \delta_\vu^k}^2 \eta_h^k } + \frac{1}{\alpha}\intOref{ \abs{ \delta_\vu^k}^2 \eta_h^k }
\\& +  \norm{\delta_\xi}_{L^2L^p}^2
+   c_1 \norm{\delta_\eta}_{L^2 L^p }^2 
+ c_2 \norm{\pdx \delta_\eta}_{L^2L^p}^2  
+ c_3 \TS^2   +c_4 h^2 ,
\end{aligned}
\end{equation}
 where 
\begin{equation}\label{cis}
\begin{aligned}
& c_1= \vrf / \underline{\eta}   \norm{\hvv}_{L^\infty L^\infty}^2 \left(1+ \norm{\M}_{L^\infty L^\infty}^2  \right), 
	\quad 
c_2 = \vrf / \underline{\eta}  \norm{\M}_{L^\infty L^\infty}^2 \norm{\hvv  }_{L^\infty L^\infty}^2 , 
	\\& 
c_3=\vrf \Ov{\eta} \left( \norm{\pdt \hvu}_{L^2L^p}^2 +  \norm{\pdt^2 \partial_x\eta}_{L^2L^2}^2\right),
	\\&
c_4=  c_1  \norm{\pdx \eta}_{L^2 L^p }   + c_2 \norm{\Lapx \eta}_{L^2L^2} + \vrf \Ov{\eta}  \norm{\Grad \hvu}_{L^2L^p}^2  .
\end{aligned}
\end{equation}
\end{Lemma}
\begin{proof}
Here we shall frequently recall the estimates \eqref{ests}. 
First, \eqref{est1} is obvious as $\eta_h$ and $\eta$ are bounded from above and below by positive constants, as well as $\pdx \eta$ and $\pdx \eta_h$ are bounded from above.  

By the triangular inequality and the Sobolev inequality we get \eqref{est2}
\begin{align*}
& \norm{\eta_h - \eta}_{L^\gamma L^\infty} 
 \leq \norm{\delta_\eta}_{L^\gamma L^\infty } 
+  \norm{I_\eta}_{L^\gamma L^\infty }
\aleq \norm{\pdx \delta_\eta}_{L^\gamma L^2 }  + h \norm{\pdx \eta}_{L^\gamma L^\infty } 
\aleq \norm{\pdx \delta_\eta}_{L^\gamma L^2 }  + h \norm{\pdx^2 \eta}_{L^\gamma L^2 } 
.
\end{align*}
Analogously, we have \eqref{est3}
\begin{align*}
 &\norm{\pdx \eta_h - \pdx \eta}_{L^2 L^\infty } 
  \leq
  \norm{\pdx \delta_\eta}_{L^2 L^\infty } 
+  \norm{\pdx I_\eta}_{L^2 L^\infty } 
 \aleq \norm{\Laph  \delta_\eta}_{L^2 L^2 } 
+ h \norm{\Lapx \eta}_{L^2 L^\infty } 
\\& =\norm{ \delta_\zeta}_{L^2 L^2 } 
+ h \norm{\Lapx \eta}_{L^2 L^\infty } 
\aleq \norm{ \delta_\zeta}_{L^2 L^2 } 
+ h \norm{\pdx^3 \eta}_{L^2 L^2 }  .
\end{align*}
Recalling the definition of $\M$ and $\Mh$, using triangular inequality and the estimates \eqref{est2} and \eqref{est3} we obtain \eqref{est41}
\begin{align*}
& \norm{\Mh - \M}_{L^2 L^\infty}   
= \norm{ \left(\begin{array}{cc} e_\eta  & -\xrefy \pdx e_\eta \\ 0 & 0 \end{array} \right) }_{L^2 L^\infty}   
\\& \aleq  \norm{\pdx \delta_\eta}_{L^2 L^2} +  \norm{ \delta_\zeta}_{L^2 L^2} + h\left(\norm{\pdx^2 \eta}_{L^2 L^2} + \norm{\pdx^3 \eta}_{L^2 L^2} \right) 
 .
\end{align*}
Analogously, we get \eqref{est42}
\begin{align*}
& \norm{\Jacob_h - \Jacob}_{L^2L^\infty}   
= \norm{ \left(\begin{array}{cc}  0 & 0 \\ -\xrefy \pdx e_\eta  & e_\eta \end{array} \right) }_{L^2L^\infty}   
\\&  \aleq \norm{\pdx \delta_\eta}_{L^2 L^2} +  \norm{ \delta_\zeta}_{L^2 L^2} + h\left(\norm{\pdx^2 \eta}_{L^2 L^2} + \norm{\pdx^3 \eta}_{L^2 L^2} \right) 
,
\end{align*}
and \eqref{est43}
\begin{align*}
& \norm{\Jacob_h^{-1} - \Jacob^{-1}}_{L^2 L^\infty}   
= \norm{ \left(\begin{array}{cc}  0 & 0 \\ -\xrefy \frac{\pdx \eta_h}{\eta_h} +\xrefy \frac{\pdx \eta}{\eta}  & \frac{1}{\eta_h} - \frac{1}{\eta} \end{array} \right) }_{L^2 L^\infty}   
\\& \leq  
\norm{ \frac{\eta_h \pdx \eta - \eta \pdx \eta_h}{\eta \eta_h} }_{L^2 L^\infty} + \norm{\frac{\eta  - \eta_h }{\eta \eta_h} }_{L^2 L^\infty} 
 =   
  \norm{\frac{e_\eta \pdx \eta + \eta \pdx e_\eta}{\eta \eta_h}}_{L^2 L^\infty} + \norm{\frac{e_\eta }{\eta \eta_h}}_{L^2 L^\infty} 
\\& \aleq 
 \frac{1 + \norm{\pdx \eta}_{L^\infty L^\infty}  }{\underline{\eta}^2}   \norm{  e_\eta}_{L^2 L^\infty} 
+ \frac{1 }{\underline{\eta}} \norm{\pdx e_\eta}_{L^2 L^\infty} 
\\& \aleq 
\frac{1 + \norm{\pdx \eta}_{L^\infty L^\infty}  }{\underline{\eta}^2} \left( \norm{\pdx \delta_\eta}_{L^2 L^2 }  + h \norm{\pdx \eta}_{L^2 L^\infty }  \right)
+ \frac{1 }{\underline{\eta}}\left( \norm{\delta_\zeta}_{L^2 L^2 }  + h \norm{\Lapx \eta}_{L^2 L^\infty }  \right)
\\& \aleq  
  \frac{1 + \norm{\pdx \eta}_{L^\infty L^\infty}  }{\underline{\eta}^2}  \norm{\pdx \delta_\eta}_{L^2 L^2 } 
 + \frac{1 }{\underline{\eta}} \norm{\delta_\zeta}_{L^2 L^2 }  
\\& \quad + h \left( \frac{1 + \norm{\pdx \eta}_{L^\infty L^\infty}  }{\underline{\eta}^2}   \norm{\pdx \eta}_{L^2 L^\infty } 
+ \frac{1 }{\underline{\eta}}  \norm{\Lapx \eta}_{L^2 L^\infty }  \right).
\end{align*}
The estimate \eqref{eq:trace} is a consequence of the trace estimate and Sobolev embedding in 1-D. 
{
For \eqref{eq:new} we first recall \eqref{IM1} and the triangular inequality to get 
\begin{align*}
\norm{D_t\delta_\eta^{k+1}}_{L^p} \leq 
\norm{ \delta_\xi^k}_{L^p} +
\norm{\Riesz (\PDt \eta^{k+1} - \pdt \eta^k)}_{L^p}.
\end{align*}
Then the estimate follows from  the continuity of $\Riesz$ and Lemma~\ref{lm_edt} as 
\[
\norm{\Riesz (\PDt \eta^{k+1} - \pdt \eta^k)}_{L^2  L^p}\aleq \norm{\pdx (\PDt \eta^{k+1} - \pdt \eta^k)}_{L^2  L^2}\aleq \TS \norm{\pdtt \pdx \eta}_{L^2 L^2}.
\]
}
The proof of \eqref{est5} is by Sobolev embedding and the interpolation estimate,
\begin{align*}
& \norm{\hvu_h -\hvu}_{L^2L^p}
\leq \norm{\delta_\vu }_{L^2L^p} + \norm{I_\vu }_{L^2L^p} 
\aleq \norm{\delta_\vu }_{L^2L^p} + \norm{\nabla I_\vu }_{L^2L^2} \aleq
\norm{\delta_\vu }_{L^2L^p} + h \norm{\hvu }_{L^2W^{2,2}} .
\end{align*}
Recalling the definition of $\hvw_h$ and $\hvw$, using again Sobolev embedding, {Lemma \ref{lm_edt}, the estimate \eqref{eq:new} and the interpolation inequality} we get \eqref{est6}:
\begin{align*}
& \norm{\hvw_h -\hvw}_{L^2L^p}
\aleq \norm{\PDt \eta_h - \pdt \eta }_{L^2L^p} 
\\& \leq \norm{\PDt (\eta_h -  \eta) }_{L^2L^p} 
+\norm{\PDt \eta - \pdt \eta }_{L^2L^p} 
\\& \aleq \norm{\PDt \delta_\eta }_{L^2L^p} 
+ \norm{\PDt I_\eta }_{L^2L^p} 
 +\norm{\PDt \pdx\eta - \pdt \pdx\eta }_{L^2L^2} 
\\& \aleq \norm{\PDt \delta_\eta }_{L^2L^p} 
 + \norm{\nabla \PDt I_\eta }_{L^2L^2} 
 + \TS \norm{\pdt^2 \pdx \eta }_{L^2 L^2}  
 \\& \aleq \norm{\delta_\xi }_{L^2L^p} 
 + h \norm{\eta }_{L^2W^{2,2}} + \TS \norm{\pdt^2 \pdx \eta }_{L^2L^2}.
\end{align*}

By the triangular inequality, \eqref{est2}, \eqref{est5}, and \eqref{est6} we get \eqref{est7}
\begin{align*}
& \vrf \TS \summ \bigg(\intOref{\abs{\hvv_h^{k-1} \eta_h^k  -  \Jacob_h^k  (\Jacob^k) ^{-1} \hvv^k  \eta^k}^p } 
\bigg)^\frac{2}{p}
\\& = 
\vrf \TS \summ \bigg(\intOref{\abs{(\hvv_h^{k-1} -\hvv^k)\eta_h^k  + \hvv^k (\eta_h^k -\eta^k) + (\I-  \Jacob_h^k  (\Jacob^k) ^{-1}) \hvv^k  \eta^k}^p }\bigg)^\frac{2}{p}
\\& \leq  
 \vrf \TS \summ \bigg(\intOref{\abs{-\TS\PDt \hvu_h^k + e_\vu^k + \hvw_h^k - \hvw^k}^p \eta_h^k }\bigg)^\frac{2}{p}
 +\vrf \TS \summ \bigg(\intOref{\abs{ \hvv^k (\eta_h^k -\eta^k)}^p/\eta_h^k }\bigg)^\frac{2}{p}
\\& \quad   
 +\vrf \TS \summ \bigg(\intOref{\abs{(\I-  \Jacob_h^k  (\Jacob^k) ^{-1}) \hvv^k  \eta^k}^p/\eta_h^k }\bigg)^\frac{2}{p}
\\& \aleq 
 \TS \summ \bigg(\intOref{\vrf \abs{\delta_\vu^k}^p \eta_h^k }\bigg)^\frac{2}{p}
 +  \norm{\delta_\xi}_{L^2L^p}^2
+   c_1 \norm{\delta_\eta}_{L^2 L^p }^2 
+ c_2 \norm{\pdx \delta_\eta}_{L^2L^p}^2  
+ c_3 \TS^2   +c_4 h^2 ,
\end{align*}
where $c_i, i=1,2,3,4$ are given above in \eqref{cis}, interpolation implies for $\frac{2}{p}=\theta$, $p<\infty$ that
\[
\bigg(\intOref{\vrf \abs{\delta_\vu^k}^p \eta_h^k }\bigg)^\frac{2}{p}\aleq \norm{\nabla \delta_\vu^k}_{L^2}^{(1-\theta)2}\norm{ \delta_\vu^k}_{L^2}^{2\theta } \aleq \alpha \norm{\nabla \delta_\vu^k}_{L^2}^{2}+\frac{1}{\alpha}\norm{ \delta_\vu^k}_{L^2}^{2}, 
\]
which closes the proof.
%
\end{proof}

\subsection{Proof of estimates \eqref{res}}\label{app_res}

\begin{proof}
First, by Young's inequality, the interpolation error, Theorem~\ref{thm:projection-velocity}, and the uniform bounds \eqref{ests}, we find
\begin{equation}\label{REIGF}
\begin{aligned}
\Abs{G_f}  =& 
\Bigabs{\TS \summ  \intOref{  \vrf \big(\eta_h^k \PDt I_\vu^k  + \frac12 \PDt \eta_h^k I_\vu^{k*} \big) \cdot \delta_\vu^k  }  
\\&  + 2 \mu \TS \summ   \intOref{   \big( \Grad I_\vu^k  (\Jacob_h^k)^{-1} \big)^\rmS: (\Grad \delta_\vu^k (\Jacob_h^k)^{-1} ) \eta_h^k } } ,
\\ \aleq  &
    \summ \TS\intOref{ \vrf \eta_h^k  |\delta_{\vu}^k|^2  } + 2 \alpha  \mu \summ \TS\intOref{ \abs{ \Grad \delta_\vu^k (\Jacob_h^k)^{-1} }^2 \eta_h^k  } +  c h^2 
\end{aligned}
\end{equation}
for any fixed $\alpha>0$, where
\begin{align*}
c&=  \frac14 \vrf  \left( \Ov{\eta} \norm{ \pdt \Grad \hvu}_{L^2L^2}^2 + \frac{1}{4 \underline{\eta}} \norm{ \xi_h}_{L^2L^\infty}^2  \norm{ \Grad \hvu}_{L^\infty L^2}^2 \right)
 + \frac{2 \mu  \Ov{\eta}}{4\alpha}\norm{\Jacob_h}_{L^\infty L^\infty} \norm{\hvu}_{L^2W^{2,2}}.
\end{align*}
Note that this constant is indeed bounded by the stability of the discrete solution, as 
$
\norm{ \xi_h}_{L^2L^\infty}^2
$
can be bounded by \eqref{eq:trace}. 
Second, by Young's inequality, the time discretization error \eqref{edts2}, we can control $G_s$ in the following way. 
\begin{equation}\label{REIGS}
\begin{aligned}
 \Abs{ G_s} =& 
\Bigabs{ \gamma_1 \TS \summ \intS{\pdx \delta_\eta^{k+1} \pdx (\PDt \eta^{k+1} - \pdt\eta^k )} -
\gamma_2 \TS \summ\intS{ \delta_\zeta^{k+1} \Lapx (\PDt \eta^{k+1} - \pdt\eta^k )}
\\&  + \TS \summ  \intS{\vrs \PDt I_\xi  \delta_\xi^k} }
\\ \aleq &
   \summ \TS \intSB{\gamma_1\abs{\pdx \delta_\eta^{k+1}}^2 + \gamma_2\abs{\delta_\zeta^{k+1}}^2 + \vrs \abs{\delta_\xi^k}^2  } 
+ \frac14 \vrs \norm{\pdt I_\xi}_{L^2L^2}^2 
\\& 
+\frac14 \summ \TS \intSB{ \gamma_1|\PDt \pdx \eta^{k+1} - \pdt\pdx\eta^k|^2 + \gamma_2|\PDt \Lapx \eta^{k+1} - \pdt\Lapx\eta^k|^2 } 
\\ \aleq &  
\summ \TS \intSB{\gamma_1\abs{\pdx \delta_\eta^{k+1}}^2 + \gamma_2\abs{\delta_\zeta^{k+1}}^2 + \vrs \abs{\delta_\xi^k}^2  }  + c_1 \TS^2 + c_2 h^2 ,
\end{aligned}
\end{equation}
where 
\begin{align*}
c_1 & =  \frac14 \gamma_1
\norm{\pd_t^2 \pdx \eta}_{L^2((0,T)\times\Sigma)}^2  + \frac14 \gamma_2 \norm{\pd_t^2 \Lapx \eta}_{L^2((0,T)\times\Sigma)}^2,
\\ c_2 &= \frac14 \vrs \norm{\pdt \pdx \xi}_{L^2((0,T)\times\Sigma)}^2= \frac14 \vrs \norm{\pdt^2 \pdx \eta }_{L^2((0,T)\times\Sigma)}^2. 
\end{align*}
Next, we analyze the $R_i$ terms.
\paragraph{$R^k_1$-term}
\begin{align*}
& \Abs{\TS \summ R^k_1} = \Abs{ \TS \summ \vrf \intOref{\big( e_\eta^k \pdt \hvu^k+ \eta_h^k (\PDt \hvu^k -\pdt\hvu^k ) \big) \cdot \delta_\vu^k} }
\\& \leq
\TS \summ \vrf \intOref{ \eta_h^k |\delta_\vu^k|^2}
+ \frac12 \TS \summ \vrf \intOref{ |e_\eta^k|^2/\eta_h^k |\pdt \hvu^k|^2 } 
\\& \quad 	+ \frac12 \TS \summ \vrf \intOref{ \eta_h^k |\PDt \hvu^k -\pdt\hvu^k |^2 }
\\& \leq 
\TS \summ \vrf \intOref{ \eta_h^k |\delta_\vu^k|^2}
+ c_1 \TS \summ \intS{ |\pdx \delta_\eta^k|^2 }  
+ c_2 h^2 + c_3  \TS^2,
\end{align*}
where we have used \eqref{est2} and the constants read
\begin{align*}
& c_1 =  \frac{\vrf}{\underline{\eta}}\norm{ \pdt \hvu}_{L^\infty(0,T;L^2(\Oref;\R^d))}^2, 
\\& c_2= \frac{\vrf}{\underline{\eta}} \norm{\pdx \eta}_{L^2(0,T;L^\infty(\Sigma))}^2\norm{ \pdt \hvu}_{L^\infty(0,T;L^2(\Oref;\R^d))}^2, 
\\&c_3 = \frac{\vrf \Ov{\eta}}{2}  \norm{\pdt^2 \hvu}_{L^2((0,T)\times\Oref;\R^d)}^2 .
\end{align*}

\paragraph{$R^k_2$-term}
 \begin{align*}
& \Abs{ \TS \summ R^k_2} =
\Abs{  \frac12\vrf \intOref{\left( (e_\xi^{k-1} - \TS \PDt \xi^k )\hvu^{k*}- \TS \pdt \eta^k \PDt \hvu^k \right) \cdot \delta_\vu^k}
}
\\& \leq  \TS \summ \vrf \intOref{ \eta_h^k |\delta_\vu^k|^2}
	+  \TS^2  \frac{\vrf}{2 \underline{\eta}} \norm{\pd_t \hvu }_{L^\infty(0,T;L^2(\Oref;\R^d))}^2 \norm{\pdt \eta}_{L^2(0,T; L^\infty(\Sigma))}^2
\\& \quad 
+  \frac{\vrf}{ \underline{\eta}} \left(\TS^2  \norm{\pd_t \xi }_{L^2((0,T)\times\Sigma)}^2
 + \TS \summ \intS{|\delta_\xi^{k-1} + I_\xi^{k-1}|^2} 
 \right)  \norm{\hvu}_{L^\infty((0,T)\times\Oref;\R^d)}	
\\& \aleq 
\TS \summ \vrf \intOref{ \eta_h^k |\delta_\vu^k|^2}
+  \TS \summ \intS{|\delta_\xi^{k}|^2} 
+ c_1 \TS^2 + c_2 h^2 ,
 \end{align*}
 where 
\begin{align*}
& c_1 = \frac{\vrf}{2 \underline{\eta}} \norm{\pd_t \hvu }_{L^\infty(0,T;L^2(\Oref;\R^d))}^2 \norm{\pdt \eta}_{L^2(0,T; L^\infty(\Sigma))}^2
+ \frac{\vrf}{ \underline{\eta}}   \norm{\pdt^2 \eta }_{L^2((0,T)\times\Sigma)}^2\norm{\hvu}_{L^\infty((0,T)\times\Oref;\R^d)}	,
\\& 
c_2=  \frac{\vrf}{ \underline{\eta}}   \norm{\pdt \pdx \eta }_{L^2((0,T)\times\Sigma)}   \norm{\hvu}_{L^\infty((0,T)\times\Oref;\R^d)} .	
\end{align*}
\paragraph{$R^k_3$-term} For this term we recall Ladyzenskaja's estimate in 2D
\begin{align}
\label{eq:lady}
\norm{f}_{L^4}^2\leq \norm{\nabla f}_{L^2}\norm{f}_{L^2}.
\end{align}
We use it to find by the previous arguments that
\begin{align*}
&\biggabs{\TS \summ R^k_3 }
\leq \biggabs{\TS \summ  
\frac12 \vrf \intOref{ \Big(  \delta_\vu^k \cdot   (\Grad I_\vu^k ) 
-I_\vu^k   \cdot   (\Grad  \delta_\vu^k) \Big)
 \cdot (\Jacob_h^k)^{-1} \hvv_h^{k-1} \eta_h^k } 
\\& \quad 
+ \TS \summ  \frac12 \vrf \intOref{ \Big(  \delta_\vu^k \cdot (\Grad \hvu^k) - \hvu^k \cdot (\Grad  \delta_\vu^k) \Big) \cdot (\Jacob_h^k)^{-1} \left(   \hvv_h^{k-1} \eta_h^k  - \Jacob_h^k (\Jacob^k)^{-1} \hvv^{k}   \eta^k \right)  } 
}
\\& \aleq 
\TS \summ \intOref{\abs{\delta_\vu^k}\abs{\Grad I_\vu^k }\abs{\hvv_h^{k-1}} + \abs{\delta_\vu^k}\abs{\Grad \vu^k }\abs{\hvv_h^{k-1} \eta_h^k  - \Jacob_h^k (\Jacob^k)^{-1} \hvv^{k} }}
\\
&\quad 
+\TS \summ \intOref{ \abs{\Grad \delta_\vu^k}\abs{I_\vu^k }\abs{\hvv_h^{k-1}}+ \abs{\Grad\delta_\vu^k}\abs{\vu^k }\abs{\hvv_h^{k-1} \eta_h^k  - \Jacob_h^k (\Jacob^k)^{-1} \hvv^{k} }} 
\\
&\aleq \TS \summ  \norm{\hvv_h^{k-1}}_{L^{4}}\Big(\norm{\delta_\vu^k}_{L^{4}}\norm{\Grad I_\vu^k}_{L^2} + \norm{\nabla\delta_\vu^k}_{L^{2}}\norm{I_\vu^k}_{L^{4}}\Big) 
\\
&+  \TS \summ \norm{\hvv_h^{k-1} \eta_h^k  - \Jacob_h^k (\Jacob^k)^{-1} \hvv^{k}}_{L^{4}}\Big(\norm{\delta_\vu^k}_{L^{2}}\norm{\nabla\vu^k}_{L^4} + \norm{\nabla\delta_\vu^k}_{L^{2}}\norm{\vu^k}_{L^{4}}\Big)
\\
&\aleq  \frac{ c_5 h^2  \TS}{\alpha}\summ \norm{\hvv_h^k}_{W^{1,2}}^2 +\alpha \TS\summ \norm{\nabla \delta_\vu^k}_{L^2}^2
\\
&+  \TS \summ \norm{\hvv_h^{k-1} \eta_h^k  - \Jacob_h^k (\Jacob^k)^{-1} \hvv^{k}}_{L^{4}}\Big(  \norm{\delta_\vu^k}_{L^{2}}\norm{\vu^k}_{W^{2,2}}+\norm{\nabla \delta_\vu^k}_{L^{2}}\Big),
\end{align*}
with
\begin{align*}
c_5=\norm{\nabla^2\vu}_{L^2 (L^2)}^2\sup_{k}\norm{\hvv^{k}}_{L^2}^2.
\end{align*}
Now \eqref{est7} and Young's inequality implies
that
\begin{align*}
&\biggabs{\TS \summ R^k_3 }\aleq \alpha\TS\summ \norm{\nabla \delta_\vu^k}_{L^2}^2 +\frac{1}{\alpha} \intOref{ \abs{ \delta_\vu^k}^2 \eta_h^k } + \norm{ \delta_\xi}_{L^2L^4}^2
\\
&\quad 
+   c_1 \norm{\delta_\eta}_{L^2 L^4 }^2 
+ c_2 \norm{\pdx \delta_\eta}_{L^2L^4}^2  
+ c_3 \TS^2   + \tilde{c}_4 h^2 
\end{align*}
for $c_1$, $c_2$, $c_3$, $c_4$ given in \eqref{cis} with 
\[
\tilde{c}_4=c_4+\frac{ c_5 h^2  \TS}{\alpha}\summ \norm{\hvv_h^k}_{W^{1,2}}^2 +\norm{\vu}_{L^2W^{2,2}}^2.
\]
Finally by using  \eqref{eq:trace}  we estimate
\[
\norm{ \delta_\xi^k}_{L^4}^2\aleq \norm{ \delta_\xi^k}_{L^2}\norm{ \delta_\xi^k}_{L^\infty}\aleq \norm{ \delta_\xi^k}_{L^2}\norm{\nabla \delta_\vu^k}_{L^2}\aleq \alpha \norm{\nabla \delta_\vu^k}_{L^2}^2+\frac{1}{\alpha} \norm{ \delta_\xi^k}_{L^2}^2,
\]
With that and Sobolev embedding we conclude that
\begin{align*}
&\biggabs{\TS \summ R^k_3 }\aleq \alpha\TS\summ \norm{\nabla \delta_\vu^k}_{L^2}^2 +\frac{1}{\alpha} \intOref{ \abs{ \delta_\vu^k}^2 \eta_h^k} +\frac{1}{\alpha}  \norm{ \delta_\xi^k}_{L^2L^2}^2 
\\
&\quad 
+   c_1 \norm{\partial_{x_1}\delta_\eta}_{L^2 L^2 }^2 
+ c_2 \norm{\pdx^2 \delta_\eta}_{L^2L^2}^2  
+ c_3 \TS^2   + \tilde{c}_4 h^2 .
\end{align*}

%
%
\paragraph{$R^k_4$-term} 
By using the identity \eqref{ddiv2}, Young's inequality, H\"older's inequality, and the interpolation error \eqref{proe}, we obtain 
\begin{align*}
& 
\Abs{\TS \summ R^k_4 }
=\Abs{\TS \summ  \intOref{  e_p^k \Grad \delta_\vu^k : \Mhk } 
+ \intOref{ \hp^k \Grad \delta_\vu^k : \big(\Mhk -  \Mk\big)}}
\\& =
\Abs{ \TS \summ \intOrefB{ I_p^k \Grad \delta_\vu^k : \Mhk +  \hp^k \Grad \delta_\vu^k : \big(\Mhk -  \Mk\big)} } 
\\& \aleq  2 \alpha  \mu \TS \summ \intOref{\abs{\Grad \delta_\vu^k (\Jacob_h^k)^{-1}}^2\eta_h^k} + \frac{\Ov{\eta}}{4 \alpha  \mu }\norm{I_p}_{L^2((0,T)\times\Oref)}^2 +
\\&\quad + \frac{1}{4 \alpha  \mu \underline{\eta}} \norm{\hp}_{L^\infty L^\infty}^2 \norm{\Jacob_h}_{L^\infty L^\infty}^2  \norm{\Mh -  \M}_{L^2(0,T;L^\infty(\Oref))}^2
\\& \aleq  2 \alpha  \mu \TS \summ \intOref{\abs{\Grad \delta_\vu^k (\Jacob_h^k)^{-1}}^2\eta_h^k}  +
c_1\left( \norm{\pdx \delta_\eta}_{L^2((0,T)\times\Sigma) }^2 + \norm{\Lapx \delta_\eta}_{L^2((0,T)\times\Sigma) }^2 \right)
+ c_2 h^2 ,
\end{align*}
where we have used \eqref{est41} and the constants read
\begin{align*} 
c_1=& \frac{1}{4 \alpha  \mu \underline{\eta}} \norm{ \hp}_{L^\infty L^2}^2 \norm{\Jacob_h}_{L^\infty ((0,T)\times\Oref;\R^{d\times d})}^2,
\\
c_2=&\frac{1}{4 \alpha  \mu \underline{\eta}} 
\norm{\Grad \hp}_{L^2 L^2}^2 
+ c_1 \left(  \norm{\pdx \eta}_{L^2(0,T;L^\infty(\Sigma))}^2 + \norm{\Lapx \eta}_{L^2(0,T;L^\infty(\Sigma))}^2 \right)
\end{align*}
\begin{Remark}\label{pifmh}
Thanks to the nice interpolation operator which produces the divergence-free condition \eqref{P3} with the covariance $\Mh$ instead of $\M$. Otherwise, we would lose the equality \eqref{ddiv2} and has to estimate $\TS \summ  \intOref{  \delta_p^k \Grad \delta_\vu^k : \Mhk} $, in which the pressure error $\delta_p$ is not available in our setting.  
\end{Remark}

\paragraph{$R^k_5$-term} Applying Young's inequality, H\"older's inequality,    \eqref{est41}, and \eqref{est43}, we obtain 
\begin{align*}
& \Abs{ \TS \summ R^k_5 } 
=  \Abs{\TS \summ \intOrefB{ 
	\big( \Grad \hvu^k (\Jacob_h^k)^{-1} \big)^\rmS: \left(\Grad \delta_\vu^k (\Jacob_h^k)^{-1} \eta_h^k \right)
	- \big( \Grad \hvu^k (\Jacob^k)^{-1} \big)^\rmS: \left(\Grad \delta_\vu^k (\Jacob^k)^{-1} \eta^k \right)}  }
\\& \leq 
	\Abs{ \TS \summ \intOref{ 
	\Big( \Grad \hvu^k \big( (\Jacob_h^k)^{-1} - (\Jacob^k)^{-1} \big)\Big) ^\rmS: \left(\Grad \delta_\vu^k (\Jacob_h^k)^{-1} \eta_h^k \right)}}
\\& \quad	+ \Abs{ \TS \summ \intOref{ \big( \Grad \hvu^k (\Jacob^k)^{-1} \big)^\rmS: \left(\Grad \delta_\vu^k (\Jacob_h^k)^{-1}   (\eta_h^k - \Jacob_h^k (\Jacob^k)^{-1} \eta^k \right) }} 
\\& \aleq 
	\alpha  \TS \summ \intOref{\Abs{\left(\Grad \delta_\vu^k (\Jacob_h^k)^{-1} \right)^\rmS}^2\eta_h^k} 
	+	\frac{\Ov{\eta}}{2\alpha}\norm{\Grad \hvu}_{L^\infty L^2}^2 \norm{ (\Jacob_h)^{-1} - (\Jacob)^{-1} }_{L^2 L^\infty}^2
\\& \quad 
+	\frac{\Ov{\eta}}{2\alpha}\norm{\Grad \hvu}_{L^\infty L^2}^2 \norm{\Jacob^{-1}}_{L^\infty L^\infty}^2 \norm{\Jacob_h}_{L^\infty L^\infty}^2 \norm{  \Mh^\rmT - \M^\rmT  }_{L^2 L^\infty}^2
\\& \aleq 
\alpha  \TS \summ \intOref{\Abs{\left(\Grad \delta_\vu^k (\Jacob_h^k)^{-1} \right)^\rmS}^2\eta_h^k}   
	  +  c_1  \norm{\pdx \delta_\eta}_{L^2((0,T)\times\Sigma) }^2 + c_2 \norm{ \delta_\zeta}_{L^2((0,T)\times\Sigma) }^2 
+ c_3 h^2 
\end{align*}
where 
\begin{align*}
c_1 = &
\frac{\Ov{\eta}}{2\alpha}\norm{\Grad \hvu}_{L^\infty L^2}^2 
\left(  \norm{\Jacob^{-1}}_{L^\infty L^\infty}^2 \norm{\Jacob_h}_{L^\infty L^\infty}^2 
+ \frac{1 +\norm{\pdx \eta}_{L^\infty L^\infty} }{\underline{\eta}^2}   \right),
\\ c_2=& 
	\frac{\Ov{\eta}}{2\alpha}\norm{\Grad \hvu}_{L^\infty L^2}^2 
\left(  \norm{\Jacob^{-1}}_{L^\infty L^\infty}^2 \norm{\Jacob_h}_{L^\infty L^\infty}^2  + \frac{1}{\underline{\eta}} \right),
\\ c_3=& 
c_1 \norm{\pdx^2 \eta}_{L^2((0,T)\times\Sigma)}^2 + c_2 \norm{\pdx^3 \eta}_{L^2((0,T)\times\Sigma)}^2 .
\end{align*}

\paragraph{$R^k_6$-term} By Young's inequality and \eqref{edts1} we obtain
\begin{align*}
 \abs{\TS \summ R^k_6} & = \Abs{ \TS \summ \vrs\intS{ (\PDt \xi^k -\pdt \xi^k) \delta_\xi^k} }
  \aleq \frac{\TS^2}{4\vrs}  \norm{\pd_t^2 \xi }_{L^2((0,T)\times\Sigma)}^2 + \TS \summ \intS{\vrs \abs{\delta_\xi}^2 } .
\end{align*}
\paragraph{$R^k_7$-term.} 

\begin{align*}
& \Abs{ \TS \summ R^k_7 } =  \Abs{ \TS \summ  \gamma_1 \intS{\pdx (\eta^{k+1} -\eta^k) \pdx  \delta_\xi^k} 
+  \TS \summ  \gamma_2 \intS{\pdx (\zeta^{k+1} -\zeta^k) \pdx \delta_\xi^k} }
\\
&\aleq \TS \summ \TS \Big(\norm{\pdx^2D_t\eta^k}+\norm{\pdx^2D_t\zeta^k}_{L^2}\Big)\norm{\delta_\xi^k}_{L^2}
\\
&\aleq \TS \summ \norm{\delta_\xi^k}_{L^2}^2 + \TS^2(\norm{\pdx^4\xi}_{L^2 L^2}^2+\norm{\pdx^2\xi}_{L^2 L^2}^2)
\end{align*}
%
Consequently, collecting all the above estimates we get 
\begin{equation*}
\Abs{ \TS \summ \sum_{i=1}^7 R^k_i + G_f + G_s} 
\aleq \TS^2 + h^2  
+ c \TS \summ \delta_E^k 
+ 2 \alpha  \mu   \TS \summ \intOref{\left| \Grad \delta_\vu^k (\Jacob_h^k)^{-1}\right|^2 \eta_h^k}
,
\end{equation*}
which proves \eqref{res}.

\end{proof}

\section{Numerical Implementation}
In this appendix we provide details of the numerical implementations of semi-implicit Scheme-R \eqref{SKM_ref} and monolithic fully implicit \eqref{wf_ref} both computed on the reference domain $\Oref$ and both implemented using FEniCS finite element method \cite{Fenics}. Here, let us point out that, instead of implementing the height of the structure $\eta$, {we take a shift $\eta=\eta-1$ (independent of $\xrefy$) and then linearly extend it to the whole domain via $\eta = \eta \xrefy$. Moreover, the structure velocity $\xi$ on $\Gamma$ is directly replaced by the second component of the fluid velocity $\xi =u_2$. Further, instead of $\zeta$ we shall use $z$ as the second order derivative of the new $\eta$. Hereinafter, we shall frequently drop the superscript `` $\widehat{}$ " for simplicity of the notation.}

\subsection{Implementation of semi-implicit Scheme-R}\label{implement_semiimplicit}
We implemented Scheme-R \eqref{SKM_ref}, the monolithic method on the reference domain $\Oref$. The domain $\Oref$ is approximated by regular triangles $K\in\gridf$ with the typical mesh size $h$. The problem comprises four global unknowns: velocity $\vu$, pressure $p$, mesh displacement in $\xy-$direction $\eta$ and its second order derivative $z$. The velocity-pressure pair is approximated with the inf-sup compatible MINI element \cite{Arnold1984}, where the velocity is approximated by the piecewise linear continuous elements enlarged with the cubic bubbles, mesh displacement by the same elements as the velocity and the second order derivative of the mesh displacement is approximated by piecewise linear elements, for the definitions of the discrete function spaces see \eqref{discrete_spaces}.

For the time stepping we use a backward Euler method with a fixed time step $\TS$, we denote by $\vu^k, {\bf z}^k, p^k$ and $\eta^k$ the unknowns at the $k^{\rm th}$ time step, i.e. at time $t=k \TS$ and $\vu^0, {\bf z}^0, p^0, \eta^0$ are prescribed initial conditions (in our case equal to zero). 
{Since in case of zero initial conditions, it holds $\eta^1 = \eta^0$, we may shift the time index $k$ (resp. $k+1$) to $k-1$ (resp. $k$) for the structure variables.}  
This semi-implicit scheme is linear and the corresponding system of linear equations is solved with the direct solver MUMPS \cite{mumps}.
Components of the velocity $\vu$ are denoted by $(u_1, u_2)$.

Displacement $\eta^k$ is computed explicitly using the $y$-component of the velocity $u_2^k$ on the top boundary $\Gamma$, i.e.
$$
\eta^k=\eta^{k-1}+\TS\, u_2^k \quad\text{on}\ \Gamma.
$$
The following quantities are used in the discretized weak form. The deformation gradient $\Jacob$ is obtained from the displacement, $J$ is its determinant and $\dot{J}$ its time derivative. All are evaluated at the $(k-1)$st time level
$$
\Jacob^{k-1} = \mathbb{I} + \nabla (0, \eta^{k-1})^\rmT,\quad
J^{k-1} = \det(\Jacob^{k-1}),\quad
\dot{J}^{k-1}=D_tJ^{k-1}=\frac{J^{k-1}-J^{k-2}}{\TS}.
$$
Finally, $\vv$ is the relative velocity of the fluid and $\hbftau$ is the Cauchy stress tensor after the ALE transformation 
$$
\vv^{k-1} = \vu^{k-1} - (0, \xi^{k-1})^\rmT,\quad
\hbftau^{k}= -p^k\I + 2 \mu\left(\nabla\vu^k(\Jacob^{k-1})^{-1}\right)^\rmS,
$$
where $\xi^{k-1}=D_t \eta^{k-1}$ is the mesh velocity that is computed after Step 2 when the displacement $\eta$ is prolongated into the whole domain $\Oref$.

The whole simulation consists of two steps. In Step 1 we solve for velocity $\vu$, its Laplace ${\bf z}$ and pressure $p$, explicitly compute the value of $\eta$ on the top boundary $\Gamma$ and in Step 2 we linearly expand it to the whole domain $\Oref$.
\begin{description}
\item[Step 1] We solve for $\vu, {z}$ and $p$.
\begin{equation*}
\intOref{ J^{k-1}\tr\left(\nabla\vu^k\,\left(\Jacob^{k-1}\right)^{-1}\right) q } =0,
\end{equation*}
\begin{equation*}
\intOref{ \left(zb + \left(\partial_{x_1}\eta^{k-1}+\TS\, \partial_{x_1}u_2^k\right)\,\partial_{x_1}b\right) } =0,
\end{equation*}
\begin{equation*}
\begin{aligned}
&\rho_f \intOref{J^{k-1} D_t \vu^k\cdot\pmb{\varphi} }
+\frac12\rho_f\intOref{\dot{J}^{k-1}(2\vu^k-\vu^{k-1})\cdot\pmb{\varphi}}
+\frac12\rho_f\intOref{J^{k-1}\nabla\vu^k\left(\Jacob^{k-1}\right)^{-1}\vv^{k-1} \cdot\pmb{\varphi}}
\\
&-\frac12\rho_f\intOref{ J^{k-1}\nabla\pmb{\varphi}\left(\Jacob^{k-1}\right)^{-1}\vv^{k-1} \cdot\vu^k}
+\intOref{ J^{k-1}\tr\left(\hbftau^k\nabla\pmb{\varphi}(\Jacob^{k-1})^{-1})\right)}
\\
&+\int_{\Gamma}\rho_s D_t u_2^k \,\varphi_2\dSx+\gamma_1\int_{\Gamma}\left(\partial_{x_1}\eta^{k-1}+\TS\, \partial_{x_1}u_2^k\right)\,\partial_{x_1}\varphi_2\dSx
\\
&-\gamma_2\int_{\Gamma} \partial_{x_1}z\,\partial_{x_1}\varphi_2\dSx+\gamma_3\int_{\Gamma}\partial_{x_1}u_2\,\partial_{x_1}\varphi_2\dSx
+\int_{\Gamma}f{\varphi_2}\dSx=0,
\end{aligned}
\end{equation*}
where $\pmb{\varphi}=(\varphi_1,\varphi_2)$ is the test function corresponding to the velocity $\vu$ and $\varphi_1,\varphi_2$ its components, $q$ is the test function for the pressure $p$ and $b$ the test function for $z$. Finally, $f$ denotes the $y$-component of the force acting on the boundary $\Gamma$.

\item[Step 2]
We linearly prolongate the displacement $\eta$ to the whole domain $\Oref$ by solving
\begin{equation*}
\intOref{ \partial_{x_2} \eta^k \ \partial_{x_2} \psi }=0,
\end{equation*}
for all $\psi\in B$. Here, $\eta^k = \eta^{k-1} + \TS\, u_2^{k}$, where $u_2^k$ is obtained in Step 1.
\end{description}

\subsection{Implementation of fully implicit scheme}\label{implement_fullyimplicit}
In Section~\ref{sec:comparison_schemeR-fully_implicit} we compare our Scheme-R to the fully implicit method based on the weak form \eqref{wf_ref}. As in the case of implementation of Scheme-R, the domain $\Oref$ is approximated by regular triangles $T\in\mathcal{T}_h$ and the problem comprises four global unknowns: velocity $\vu$, pressure $p$, mesh displacement in $\xy-$direction $\eta$ and its second derivative $z$. The velocity-pressure pair $(\vu, p)$  is approximated with the MINI element, mesh displacement $\eta$ is from the same space as velocity, and the second derivative of the displacement $z$ is approximated by piecewise linear elements.

The time derivatives are approximated by the backward Euler time scheme, the nonlinearities are treated with the Newton solver, and the consequent set of linear equations by direct solver MUMPS. Knowing the solution $\vu^{k-1},\eta^k,z^k,p^k$ on the previous time level, we are solving the fully implicit nonlinear problem. Thus, we solve for $\vu^k, \eta^k, z^k$ and $p^k$ satisfying the continuity equation
\begin{align*}
\intOref{ J^k\tr\left(\nabla\vu^k\,(\Jacob^k)^{-1}\right) q } &=0,\\
\end{align*}
the coupled momentum equation
\begin{align*}
&\rho_f \intOref{ J^k D_t \vu^k\cdot\pmb{\varphi} }
+\frac12\rho_f\intOref{D_t{J}^k \,\vu^k\cdot\pmb{\varphi}}
+\frac12\rho_f\intOref{J^{k}\nabla\vu^k(\Jacob^k)^{-1}\vv^{k} \cdot\pmb{\varphi}}
\\
&-\frac12\rho_f\intOref{ J^{k}\nabla\pmb{\varphi}(\Jacob^k)^{-1}\vv^{k} \cdot\vu^k}
+\intOref{ J^{k}\tr\left(\hbftau^k\nabla\pmb{\varphi}(\Jacob^{k})^{-1})\right)}
\\
&+\int_{\Gamma}\rho_s D_t u_2^k \,\varphi_2\dSx+\gamma_1\int_{\Gamma} \partial_{x_1}\eta^{k} \,\partial_{x_1}\varphi_2\dSx
\\
&-\gamma_2\int_{\Gamma} \partial_{x_1}z\,\partial_{x_1}\varphi_2\dSx+\gamma_3\int_{\Gamma}\partial_{x_1}u_2\,\partial_{x_1}\varphi_2\dSx
+\int_{\Gamma}f{\varphi_2}\dSx=0.
\end{align*}
Further, the discrete Laplace equation for $z^k$ and the harmonic extension of $\eta^k$
\begin{align*}
\intOref{ \left(z^k b + \partial_{x_1}\eta^k\,\partial_{x_1}b\right) } &=0,
\\
\intOref{ \partial_{x_2} \eta^k \ \partial_{x_2} \psi }&=0,
\end{align*}
for all test functions $q, b, \pmb{\varphi}$ and $\psi$. Here, we used the same notion as above
\begin{align*}
\Jacob^{k} &= \mathbb{I} + \nabla(0, \eta^{k})^\rmT,\quad J^k=\det\Jacob^k,\quad
\vv^{k} = \vu^{k} - (0, D_t \eta^{k})^\rmT,\\
\hbftau^{k} &= -p^k\I + 2 \mu\left(\nabla\vu^k(\Jacob^{k})^{-1} \right)^\rmS
\end{align*}
and the components of the velocity $\vu$ are denoted by $(u_1,u_2)$.
The problem is periodic in $x_1$ direction, with a homogeneous Dirichlet boundary conditions for $\vu$ and $\eta$ on the bottom boundary, and  $D_t\eta^k = u_2^{k}$ on the top boundary.

\end{document}